%% file: arXiv-V3.tex
\documentclass[letterpaper,11pt,reqno]{amsart}

\usepackage[margin=1in]{geometry}

\usepackage{hyperref}
\hypersetup{
     colorlinks   = true,
     citecolor    = black,
     linkcolor    = blue,
     urlcolor     = black
}

\usepackage{graphicx,amsmath,amssymb,amsthm,paralist,color,tikz-cd}
\usepackage{mathrsfs}
\usepackage{mathtools}
\usepackage{setspace}

\let\Horig\H
\usepackage[utf8]{inputenc} 
\usepackage[T1]{fontenc}

\usepackage{amscd}

\usepackage{bbm} 
\usepackage{enumitem}

\input{macros}

\title[Fourier Decay from $L^2$-Flattening]{Fourier Decay from $L^2$-Flattening}

\author{Simon Baker}
\address{Department of Mathematical Sciences, Loughborough University, Loughborough, LE11 3TU, UK}
\email{simonbaker412@gmail.com}

\author{Osama Khalil}
\address{Department of Mathematics, Statistics, and Computer Science, University of Illinois Chicago, IL}
\email{okhalil@uic.edu}

\author{Tuomas Sahlsten}
\address{Department of Mathematics and Statistics, University of Helsinki, Finland}
\email{tuomas.sahlsten@helsinki.fi}

\date{}

\begin{document}

\begin{abstract}

We develop a unified approach for establishing rates of decay for the Fourier transform of a wide class of dynamically defined measures. 
Among the key features of the method is the systematic use of the $L^2$-flattening theorem obtained in~\cite{Khalil-Mixing}, coupled with non-concentration estimates for the derivatives of the underlying dynamical system. This method yields polylogarithmic Fourier decay for Diophantine self-similar measures, and polynomial decay for Patterson-Sullivan measures of convex cocompact hyperbolic manifolds, Gibbs measures associated to non-integrable $C^2$ conformal systems, as well as stationary measures for carpet-like non-conformal iterated function systems. Applications include an essential spectral gap on convex cocompact hyperbolic manifolds independent of the doubling constant through a fractal uncertainty principle and an equidistribution theorem for typical vectors on self-similar sets and other fractals in $\mathbb{R}^d$.

\end{abstract}

\maketitle

\section{Introduction}

\subsection{Background and summary of main results}

The Fourier transform of a Borel probability measure $\mu$ on $\R^d$ is defined as follows:
\begin{align}
    \widehat{\mu}(\xi) := \int_{\R^d} e^{2\pi i\langle \xi, x\rangle} \;d\mu(x), \qquad \xi\in\R^d.
\end{align}
Rates of decay of $|\widehat{\mu}(\xi)|$ as $\|\xi\| \to \infty$ for $\mu$ arising from dynamical systems have been extensively studied in recent years (see survey \cite{Sahlsten-survey} for history and recent developments). Beyond its intrinsic interest, this question has found many applications in other areas of mathematics; e.g.~essential spectral gaps on hyperbolic manifolds~\cite{DyatlovZahl,BourgainDyatlov,LiNaudPan}, the uniqueness problem~\cite{LiSahlsten}, quantum chaos and fractal uncertainty principles~\cite{Dyatlov-IntroFUP}, normality to arbitrary integer bases~\cite{DavenportErdosLeVeque}, and geometric measure theory~\cite{Shmerkin-AbsContBernoulliConv,Mattila-FourierBook} to name a few, see Section~\ref{sec:applications} for further discussion. Moreover, this problem has motivated the development of many new methods. These methods draw on a wide variety of tools ranging from spectral gaps of the underlying dynamics~\cite{AlgomRHWang-Polynomial,BakerSahlsten}, to renewal theory~\cite{Li-Stationary}, sum-product phenomena~\cite{BourgainDyatlov,LiNaudPan,Leclerc-JuliaSets}, large deviation estimates for Fourier transforms~\cite{MosqueraShmerkin,AlgomChangWuWu,BakerBanaji,BaYu}, and many more; cf.~\cite{AlgomHertzWang-Normality,AlgomHertzWang-Log,Sahlsten-Gauss,SahlstenStevens,LiSahlsten,LiSahlsten-SelfAffine,VarjuYu,Bremont} for a non-exhaustive list.

In this article, we systematically develop a unified approach to obtain rates of Fourier decay for a wide class of dynamically defined measures.
Our strategy can be summarized as follows:

\medskip
{\noindent \textbf{Step 1. Averaging.}}

Use the dynamics (or the multiscale/convolution structure of $\mu$) to express the Fourier transform of $\mu$ at frequency $\xi$ as an \textit{average} over the Fourier transforms of measures $\mu_{x}$, $x\in X$, evaluated at a range of frequencies determined by $\xi$ and the underlying dynamics.
The family $\set{\mu_x}$ is typically given by scaled images of $\mu$ under the dynamics or rescaled restrictions of $\mu$ to pieces of its support.

\medskip

{\noindent \textbf{Step 2. Flattening.}}
    
    Find a mechanism to show that for each $x\in X,$ the Fourier transform of $\mu_{x}$ has the desired rate of decay for a large set of frequencies. Here it is crucial that the size of the set of exceptional frequencies can be bounded independently of $x.$

\medskip

 {\noindent \textbf{Step 3. Separation.}}   
    
    Show (through a non-linearity or Diophantine assumption on the dynamics) that the frequencies appearing in the average in \textbf{Step 1} are sufficiently well-distributed that they have a suitably small intersection with the exceptional set of frequencies arising in \textbf{Step 2}.

\medskip

 A key to this strategy is provided by recent developments towards \textbf{Step 2}, where verifiable criteria have been shown to imply polynomial Fourier decay\footnote{We say that that a measure $\mu$ has \textit{polynomial Fourier decay} if $|\widehat{\mu}(\xi)| = \cO(\norm{\xi}^{-\kappa})$ for some $\kappa>0$ as $\norm{\xi}\to \infty$. Similarly, we say that $\mu$ has \textit{polylogarithmic Fourier decay} if $|\widehat{\mu}(\xi)| = \cO((\log\norm{\xi})^{-\kappa})$ for some $\kappa>0$ as $\norm{\xi}\to \infty$.} outside a sparse set of frequencies; cf.~Theorem~\ref{thm:flattening} below. In particular, such criteria are known to hold for a very wide class of dynamically defined measures, thus distilling the difficulty in implementing the above strategy to achieving the separation in \textbf{Step 3}. In this article, we develop techniques for overcoming this difficulty, thus yielding the following broad range of results.

\medskip
\begin{enumerate}
    \item\label{item:result 1} Polylogarithmic Fourier decay for \textbf{self-similar measures} satisfying Diophantine conditions either on their contraction ratios or their rotation parts: Theorems~\ref{thm:selfsim} and~\ref{thm:selfsim rotations}.

    \medskip
    
    \item\label{item:result 2} Polynomial Fourier decay for \textbf{Patterson-Sullivan measures} of convex cocompact hyperbolic manifolds: Theorem~\ref{thm:PS}.

    \medskip

    \item\label{item:result 3} Polynomial Fourier decay for \textbf{Gibbs measures for conformal IFSs}
    satisfying a spectral gap assumption: Theorem~\ref{thm:mainNonlinear}.

    \medskip

    \item\label{item:result 4} 
    Polynomial Fourier decay for the stationary measures of a class of \textbf{carpet like non-conformal IFSs} exhibiting non-linearity in each principal direction: Theorem~\ref{Theorem:Non-conformal}.

    \medskip
\end{enumerate}

Note that polylogarithmic decay is optimal for self-similar measures under our hypotheses; cf.~Remarks~\ref{rem:selfsim rate dependence} and~\ref{rem:covers} for further discussion of the rates of decay in Items~\ref{item:result 1} and~\ref{item:result 2}.
Important recent developments have yielded several special cases of the above results in low dimensions by a variety of different methods. These are recalled below in detail. 
We remark that our argument provides a new streamlined proof in all these cases, in addition to extending them to arbitrary dimensions, and to new settings such as Item~\ref{item:result 4}.
This requires handling the significant added difficulty posed by the presence of subspaces of intermediate dimension, which is also a well-known difficulty in the related study of \textit{Fractal Uncertainty Principles} in higher dimensions addressed in Corollary~\ref{cor:gap} below; cf.~\cite{Dyatlov-IntroFUP,CladekTao,BackusLengTao,Cohen-FUP} and Section \ref{sec:FUP} for further discussion.

Variants of our strategy have been utilized in prior works, with perhaps the earliest precursor being Kaufman's work on non-linear pushforwards of homogeneous self-similar measures on the line, where the analog of \textbf{Step 2} was achieved using Erd\Horig{o}s-Kahane arguments~\cite{Kaufman}.
Our strategy also bears similarities to the one employed by Bourgain and Dyatlov in \cite{BourgainDyatlov}, who considered Patterson-Sullivan measures in dimension one, achieving the analog of \textbf{Step 2} using Bourgain’s discretized sum-product theorem~\cite{BourgainRing,Bourgain-SumProduct}.
In both cases, the argument was tailored to the special structure of the setting in question (linearity in the former, and non-linearity in the latter). 
By contrast, one of the new features of our approach is that it avoids Erd\Horig{o}s-Kahane and sum-product estimates, rendering it applicable to a much broader range of settings in a unified fashion. 

\smallskip

The above results also have the following applications:

\smallskip

\begin{enumerate}
    \item \textbf{Normality and equidistribution of vectors in fractal sets}: let $\mu$ be one of the measures in Items~\ref{item:result 1}-\ref{item:result 4} above viewed as a measure on $\R^d$. Let $A$ be any expanding integer matrix on $\R^d$. Then, $\mu$-almost every $x$ is $A$-normal, i.e., the orbit $(A^n x)_{n\geq 1}$ is equidistributed $\mrm{mod} \; 1$ for $\mu$-almost every $x$: Corollary~\ref{cor:mainEquidistribution}.

\medskip
    
    \item \textbf{Essential spectral gaps and Fractal Uncertainty Principles}: let $\Gamma$ be a Zariski-dense convex cocompact group of isometries of real hyperbolic space $\H^{d+1}$ with limit set $\L_\G$, critical exponent $\d_\G \leq d/2 $, and quotient manifold $M=\H^{d+1}/\Gamma$.
    Then, $\L_\G$ satisfies a generalized Fractal Uncertainty Principle with exponent $\beta= d/2 - \d_\G + \e$, $\e>0$ and, hence, the resolvent of the Laplace-Beltrami operator $\Delta_M$ of $M$ admits an essential spectral gap of size $ \beta$. The improvement $\e$ depends only on the non-concentration parameters of the PS measure of $\G$, but not on its doubling constant: Corollary~\ref{cor:gap}.
\end{enumerate}

Details and context of the above applications are discussed in detail in Section~\ref{sec:applications} below.

\subsection{Discussion of the method}\label{sec:method}
\label{sec:discussion of method}

As noted above, a very general result towards \textbf{Step 2} was obtained in~\cite[Corollary 11.5]{Khalil-Mixing} using methods from additive combinatorics.
Before recalling this result, it will be convenient to make the following definition.

\begin{definition}\label{def:non-conc}
We say that a Borel measure $\mu$ on $\R^d$ is \textit{uniformly affinely non-concentrated} if for every $\e>0$, there exists $\d(\e)>0$ so that $\d(\e)\to 0$ as $\e\to 0$ and for all $x\in \mrm{supp}(\mu)$, $0<r\leq 1$, and every affine hyperplane $W<\R^d$, we have
        \begin{align}\label{eq:uniform affine non-conc}
            \mu( W^{(\e r)} \cap B(x,r)) \leq \d(\e) \mu(B(x,r)),
        \end{align}
        where $W^{(r)}$ and $B(x,r)$ denote the $r$-neighborhood of $W$ and the $r$-ball around $x$ respectively.
        We refer to the function $\d(\e)$ as \textit{the non-concentration parameters} of $\mu$.

\end{definition}

\begin{remark}
    It can be shown by an inductive argument that a measure $\mu$ is uniformly affinely non-concentrated if and only if~\eqref{eq:uniform affine non-conc} holds with $\d(\e) = C\e^\alpha$ for some constants $C\geq 1$ and $\alpha>0$.
    Hence, we often refer to such $\mu$ as being $(C,\alpha)$-affinely non-concentrated when we wish to emphasize the dependence on these parameters.
\end{remark}

The following result shows that affine non-concentration is sufficient to guarantee polynomial Fourier decay outside a sparse set of frequencies.

\begin{thm}[{\cite[Corollary 11.5]{Khalil-Mixing}}]
\label{thm:flattening}
Let $\mu$ be a compactly supported, Borel probability measure on $\R^d$, which is uniformly affinely non-concentrated.
Then, $\mu$ is \textit{polynomially decaying on average}, that is, for every $\epsilon>0$, there exists effectively computable $\tau=\tau(\epsilon)>0$, depending only on $\epsilon$ and the non-concentration parameters on $\mu$, such that the set
 $$\set{\norm{\xi}\leq T: |\widehat{\mu}(\xi)|>T^{-\tau}} $$
can be covered by $\cO_{\epsilon}(T^{\epsilon})$ balls of radius $1$. Moreover, the implicit constants in the $\cO_{\epsilon}(T^{\epsilon})$ bound only depend upon the non-concentration parameters of $\mu$ as well as the diameter of its support.
    
\end{thm}

Special cases of Theorem~\ref{thm:flattening} for measures on $\R$ were known before by different methods by works of Kaufman~\cite{Kaufman} and Tsujii~\cite{Tsujii-selfsimilar} for self-similar measures, and by Rossi-Shmerkin~\cite{RossiShmerkin} more generally.

In light of Theorem~\ref{thm:flattening}, the main difficulty in implementing the above strategy lies in achieving \textbf{Step 3} concerning separation of the images of the frequency under the dynamics. We note that such separation is the heart of the difficulty in obtaining Fourier decay. 
Indeed, all the other steps of the above strategy, including the conclusion of Theorem~\ref{thm:flattening}, apply equally well to the Hausdorff measure on Cantor's middle thirds set.
However, it is well-known that the Fourier transform of this measure does not tend to $0$ at infinity~\cite{Mattila-FourierBook}.
Nonetheless, we believe that the directness of the strategy will enable it to be applied in much broader contexts than those covered in this article.

\begin{remark} 
We end this discussion with some remarks on the scope of Theorem~\ref{thm:flattening}, and hence of the method developed in this article.
\begin{enumerate}
    \item It is shown in \cite[Corollary 11.5]{Khalil-Mixing} that the conclusion of Theorem~\ref{thm:flattening} holds under a much weaker non-concentration hypothesis which allows concentration to happen at some scales and on small measure sets.
    This property is satisfied for instance by Patterson-Sullivan measures for \textit{cusped} real hyperbolic manifolds. 

    \item In~\cite{BaYu}, it was observed that the proof of Theorem~\ref{thm:flattening} goes through under the following weaker form of~\eqref{eq:uniform affine non-conc} allowing the ball on the right side to have a larger radius:
    \begin{align}\label{eq:expanded ball affine non-conc}
            \mu( W^{(\e r)} \cap B(x,r)) \leq \d(\e) \mu(B(x,cr)),
    \end{align}
    where $c\geq 1$ is a fixed constant. 
    This property holds for instance for certain self-similar measures which do not satisfy~\eqref{eq:uniform affine non-conc}, e.g.~in the absence of separation conditions.
\end{enumerate}
\end{remark}

We now describe in detail the main results of this article.

\subsection{Self-similar systems}

Consider a self-similar iterated function system (IFS) of the form $\Phi = \set{x \mapsto f_{a}(x) = r_aO_{a} x + t_a: a\in \cA}$, where $\cA$ is a finite set, $|r_{a}|\in (0,1)$, $O_{a}\in O(d)$ and $t_{a}\in \R^{d}$. We say that $\Phi$ is affinely irreducible if $\Phi$ does not preserve an affine subspace of $\R^{d}.$
Given
a probability vector $\p=(p_a)_{a\in \cA}$ then it is well-known that there exists a unique Borel probability measure satisfying 
$$\mu = \sum_{a \in \cA} p_a f_a\mu.$$ 
Here, $f_{a}\mu$ denotes the pushforward of $\mu$ under $f_a$. We call this measure the self-similar measure corresponding to $\Phi$ and $\p$. Throughout this paper we will always assume that our IFS is non-trivial in the sense that the similarities do not all share a common fixed point. This ensures that the associated self-similar measures are non-atomic.

The Fourier decay properties of self-similar measures are well studied in the literature. Recently Rapaport \cite{Rapaport} completed a classification of those self-similar measures satisfying $$\limsup_{\|\xi\|\to\infty}|\widehat{\mu}(\xi)|>0.$$ This built upon earlier work of Br\'emont \cite{Bremont} and Varj\'u-Yu \cite{VarjuYu}. 

Far less in known about explicit rates of decay when $\widehat{\mu}(\xi)\to 0$ as $\norm{\xi}\to \infty$.
In dimension $d = 1$, a result of Solomyak \cite{Solomyak} states that typical self-similar measures in $\R$ have polynomial Fourier decay. That being said, polynomial Fourier decay has only been verified under certain restrictive algebraic assumptions on the contractions ratios. See the articles of Dai, Feng, and Wang \cite{DFW}, and Streck \cite{Streck}. 

Li and the third author in \cite{LiSahlsten} have shown that, under the assumption that the log-contraction ratios $\log |r_a| / \log |r_b|$, $a \neq b$, contain a  Diophantine number: for some $c > 0$ and $l > 2$ we have for all rational numbers $p/q$ that:
	$$\left|\frac{\log |r_a|}{\log |r_b|} - \frac{p}{q}\right| \geq \frac{c}{q^l},$$
	then $|\widehat{\mu}(\xi)| \to 0$ polylogarithmically. This happens for example if the IFS contains contractions by $1/2$ and $1/3$. Li and the third author's proof used ideas from renewal theory. We also refer the reader to \cite{AlgomHertzWang-Normality} where similar results are proven under a weaker Diophantine condition. Varj\'u-Yu \cite{VarjuYu} considered the complementary case where the contraction ratios are all integer powers of a common number. In this setting, they proved polylogarithmic Fourier decay for self-similar measures satisfying certain number theoretic assumptions.
 
 In higher dimensions, the work of Lindenstrauss and Varj\'u \cite{LindenstraussVarju} introduced a representational theoretic method to bound Fourier transforms of self-similar measures with dense rotation parts $O_a$.
 In dimension $d\geq 3$, this method implies sub-polynomial Fourier decay for all self-similar measures, where the rotational components $O_a$ of the contractions give a dense subgroup; cf.~the survey \cite{Sahlsten-survey}.
 Moreover, under a spectral gap condition for the transfer operator $f \mapsto \sum_{a \in \cA} p_a f \circ O_a^{-1}$ on $L^2(\mrm{SO}(d))$, $d\geq 3$, it is possible to show that the associated self-similar measure has polynomial Fourier decay.

 Our first result for self-similar measures is a higher dimensional generalization of the aforementioned result of Li and Sahlsten \cite{LiSahlsten}.
 
 \begin{thm}
		\label{thm:selfsim}
	Let $\Phi = \set{x \mapsto f_{a}(x) = r_aO_{a} x + t_a}_{a\in \cA}$ be an affinely irreducible self-similar IFS such that two contractions have log-contraction ratio $\frac{\log |r_a|}{\log |r_b|}$ which is Diophantine.
    Then every self-similar measure for this IFS has polylogarithmic Fourier decay.
	\end{thm}

 \begin{remark}[Optimality of Theorem~\ref{thm:selfsim}]
We emphasize that we do not impose any separation assumptions on the IFS in Theorem \ref{thm:selfsim}. 
Moreover, the affine irreduciblility assumption appearing in this theorem is necessary. 
Indeed, if it is not satisfied, then any measure supported on the invariant set of $\Phi$ cannot be Rajchman as such measure will be supported on a finite union of proper invariant affine subspaces. 
Finally, in the upcoming article~\cite{PSS}, Paukkonen, Sahlsten, and Streck  construct an inhomogeneous self-similar measure in $\R$, satisfying the assumptions of Theorem \ref{thm:selfsim}, for which there exists $C > 0$ and a sequence of frequencies $(\xi_{n})$ tending to infinity such that for all $n \in \N$ one has $C^{-1} \leq |\widehat{\mu}(\xi_n)| \sqrt{\log|\xi_n|}\leq C$. Thus, polylogarithmic Fourier decay is the best rate of decay possible under the assumptions of Theorem \ref{thm:selfsim}.    
\end{remark}

It is also possible to use the Diophantine properties of the orthogonal matrices in our IFS to prove Fourier decay results. 
Before stating our result in this direction we need to introduce some notation. Given a matrix $O\in SO(2)$, then $O$ acts on $\mathbb{R}^{2}$ by rotating anticlockwise by $2\pi \theta$ for some $\theta\in [0,1]$.  We define $\theta_{O}$ to be the unique $\theta$ defined this way. Moreover, given a self-similar IFS $\set{x \mapsto f_{a}(x) = r_aO_{a} x + t_a}_{a\in \cA}$ acting on $\mathbb{R}^{2}$ such that  $O_{a}\in SO(2)$ for all $a\in \cA$, we define $\theta_{a}=\theta_{O_{a}}$. We say that a pair of real numbers $(\theta_1,\theta_2)$ is Diophantine if there exists $C,l>0 $ such that for any non-zero $(p,q,r)\in\mathbb{Z}$ we have 
$$\left|p\theta_{1}+q\theta_{2}+r \right| \geq \frac{C}{\max\set{|p|,|q|}^{l}}.$$  
With this terminology we can state our theorem.
\begin{thm}
\label{thm:selfsim rotations}
Let $\Phi = \set{x \mapsto f_{a}(x) = r_aO_{a} x + t_a}_{a\in \cA}$ be a self-similar IFS acting on $\mathbb{R}^{2}$ such that  $O_{a}\in SO(2)$ for all $a\in \cA$, and there exists two contractions such that $(\theta_{a},\theta_{b})$ is Diophantine. Then every self-similar measure for this IFS has polylogarithmic Fourier decay. 
\end{thm}

\begin{remark}\label{rem:selfsim rate dependence}

Although we do not pursue this here, we remark that the exponent appearing in our polylogarithmic rate of Fourier decay provided by our proof of Theorems~\ref{thm:selfsim} and~\ref{thm:selfsim rotations} can be made explicit.
Explicit rates obtained in this manner depend only on the Diophantine exponent $l$ as well as on the function $\tau=\tau(\epsilon)$ provided by the polynomial Fourier decay on average hypothesis; cf. Theorem \ref{thm:flattening}.
In particular, when the self-similar measure is uniformly affinely non-concentrated, the proof of Theorem~\ref{thm:flattening} provides explicit bounds on the average decay rate $\tau(\epsilon)$, which in turn depend only on the non-concentration parameters of $\mu$.
Additionally, in the case of one-dimensional self-similar measures, arguments of the first author together with Banaji in~\cite{BakerBanaji} can also be used to provide explicit bounds on $\tau(\epsilon)$.
\end{remark}

\subsection{Patterson-Sullivan measures}

Our next result concerns Patterson-Sullivan (PS for short) measures for convex cocompact hyperbolic manifolds.
In this setting, the driving mechanism behind Fourier decay comes from non-linearity of the dynamics.
Namely, the fact that the strong stable and unstable foliations of the geodesic flow are not jointly integrable.

To formulate our result, let $\G$ be a discrete, Zariski-dense, convex cocompact, group of isometries of real hyperbolic space $\H^{d+1}$, $d\geq 1$.
Let $\L_\G$ be the limit set of $\G$ on $\partial \H^{d+1}$ and $\mu$ be the Patterson-Sullivan probability measure on $\L_\G$; cf.~Section~\ref{sec:PS prelims} for detailed definitions.
The following is our main result in this setting.

\begin{thm}\label{thm:PS}
There exists $\kappa>0$ such that the following holds for all $\vp\in C^2(\partial\H^{d+1})$, $\psi\in C^1(\partial\H^{d+1})$ satisfying
\begin{align*}
    \norm{\vp}_{C^2}+ \norm{\psi}_{C^1} \leq A,  \qquad \inf_{x\in \L_\G} \norm{\nabla_x\vp}> a,
\end{align*}
for some constants $a>0$ and $A\geq 1$. 
There exists a constant $C=C(A,a,\mu)\geq 1$, so that for all $\l\neq 0$, we have
\begin{align*}
    \left|\int_{\L_\G} e^{2\pi i\l \vp(x)} \psi(x) \;d\mu(x)\right| \leq C |\l|^{-\kappa}.
\end{align*}
\end{thm}

Theorem~\ref{thm:PS} generalizes prior work of Bourgain and Dyatlov in the case of hyperbolic surfaces \cite{BourgainDyatlov} and of Li, Naud, and Pan in the case of Schottky hyperbolic $3$-manifolds \cite{LiNaudPan}.
These results are based on Bourgain's sum-product theory, while the proof of Theorem~\ref{thm:PS} is based on Theorem~\ref{thm:flattening}, which was obtained using purely additive methods.
Another feature of our proof of Theorem~\ref{thm:PS} compared to~\cite{BourgainDyatlov,LiNaudPan} is that it does not require any symbolic coding of the action of $\G$ on $\L_\G$.

\begin{remark}\label{rem:covers}

\begin{enumerate}
    \item 
    Our proof shows that the rate $\kappa$ provided by Theorem~\ref{thm:PS} depends only on non-concentration parameters of $\mu$ in the sense of Definition~\ref{def:non-conc}. In particular, the rate of decay does not change upon replacing $\G$ by a finite index subgroup since the measure $\mu$ remains the same in this case~\cite{Roblin-2}.

\item

    PS measures for convex cocompact manifolds are known to be $(C,\alpha)$-uniformly affinely non-concentrated \cite{Dasetal}.
    In this language, the exponent of polynomial Fourier decay obtained in~\cite{BourgainDyatlov,LiNaudPan} is shown to depend only on the exponent $\alpha$, while the exponent provided by Theorem~\ref{thm:PS} depends on both $C$ and $\alpha$.
    Roughly speaking, this difference comes from the fact that the rates provided by Bourgain's sum-product theorem depend only on $\alpha$, while those provided by the flattening theorem, Theorem~\ref{thm:flattening}, depend on both $C$ and $\alpha$. 
    On the other hand, the flexible nature of Theorem~\ref{thm:flattening} enables our unified approach to the variety of different settings considered in this article.
    It is thus of interest to strengthen Theorem~\ref{thm:flattening} to remove such dependence on the constant $C$; cf.~Remark \ref{rem:FUP} for further consequences.
\end{enumerate}
\end{remark}

To keep the presentation clear, we restricted our setup to the case of convex cocompact groups.
Using the recurrence results obtained in~\cite{Khalil-Mixing}, the proof of Theorem~\ref{thm:PS} can be adapted to handle the general case of geometrically finite manifolds.

Theorem \ref{thm:PS} has immediate applications in the setting of Fractal Uncertainty Principles. We will explain this in detail in Section \ref{sec:FUP}. 

\subsection{Gibbs measures for self-conformal iterated function systems and their subshifts}

Next, let us consider the problem of studying Fourier transforms of iterated function systems $\Phi=\set{f_a}_{a \in \cA}$ for general $C^2$ contractions $f_a : [0,1]^d \to [0,1]^d$, $a \in \cA$, where $d \geq 1$. In this context, we say that an iterated function system is \textit{conformal} if for every $a\in \cA$ and $x\in [0,1]^{d}$ the Jacobian of $f_{a}$ evaluated at $x$ satisfies $$D_{x}f_{a}=\lambda_{a}(x)O_{a}(x)$$ for some $\lambda_{a}(x)\in (-1,1)\setminus\{0\}$ and $O_{a}(x)\in O(d),$ i.e. an iterated function system is conformal if, for every $a\in \cA$ and $x\in [0,1]^{d}$, the Jacobian is a similarity map. Given $x\in [0,1]^{d}$ and $\a=(a_1,\ldots,a_n)\in \cA^{n},$ we also define 
$$\lambda_{\a}(x)=\prod_{i=1}^{n}\lambda_{a_{i}}(f_{a_{i+1},\ldots a_{n}}(x)).$$ It follows from the chain rule that $$|\lambda_{\a}(x)|=\|D_{x}f_{\a}\|$$ for any $\a\in \cup_{n=1}^{\infty}\cA^{n}$ and $x\in [0,1]^{d}.$

In this context, we will prove Fourier decay results for Gibbs measures associated to subshifts of finite type and $C^{1}$ potentials. This framework incorporates self-similar measures. However, instead of using the number-theoretic properties of an IFS, this section will focus on how the non-linearity within the IFS can be used to prove polynomial Fourier decay. Such an assumption rules out self-similar measures.

\subsubsection{Background}
The idea that the non-linearity present within an IFS can be used as a tool for proving Fourier decay is well-established within the literature. In the case $d=1$, the first result in this direction was due to Kaufman \cite{Kaufman2}. For each $N\in \N$, he considered the set $B_{N}$ consisting of those badly approximable numbers whose digits in the continued fraction expansion are bounded above by $N$. Kaufman proved that $B_{N}$ supports a measure that has polynomial Fourier decay for $N\geq 3$. 
This result was later extended by Queff\'{e}lec and Ramar\'{e} \cite{QueRam} to the case $N=2$, and later to the case of more general invariant measures for the Gauss map by Jordan and the third author \cite{Sahlsten-Gauss}.

In \cite{BourgainDyatlov}, Bourgain and Dyatlov introduced a new method to study Fourier transforms of measures using ideas from Additive Combinatorics. They applied this method to prove polynomial Fourier decay for Patterson-Sullivan measures for limit sets of Fuchsian groups. This method was subsequently refined by the third author and Stevens \cite{SahlstenStevens}, and applied to a more general class of iterated functions systems. 

In recent years, more substantial progress has been made in this one-dimensional setting.
In particular, the first and third authors in \cite{BakerSahlsten}, and Algom, Rodriguez-Hertz, and Wang \cite{AlgomRHWang-Polynomial}, established polynomial Fourier decay for a large class of self-conformal measures without any separation assumptions on the underlying IFS. These results were recently built upon by the first author and Banaji \cite{BakerBanaji}, and by Algom et al. \cite{AlgomChangWuWu}, who independently proved that if an IFS consists of analytic contractions one of which is not an affine map, then every self-conformal measure has polynomial Fourier decay. 

In an as yet unpublished work, Avila, Lyubich and Zhang \cite{AvLyZh} proved polylogarithmic Fourier decay for certain measures coming from $C^{1+\alpha}$ iterated function systems. We also refer the reader to the papers \cite{AlgomHertzWang-Log,AlgomHertzWang-Normality} for further results in this direction. 
In higher dimensions, we know far less.  A result of Leclerc \cite{Leclerc-JuliaSets} gives sufficient conditions for the Julia set of a hyperbolic rational map to support a measure with polynomial Fourier decay. 
More recently, a generalization of the results of~\cite{BakerSahlsten,AlgomRHWang-Polynomial} was obtained in~\cite{AlgomRHWang-SelfConformalDimensionTwo} for irreducible $C^\omega$-IFSs in dimension $d=2$.

\subsubsection{Statement of results}

As mentioned above, our results hold under a spectral gap hypothesis on the underlying conformal IFS. To properly formulate this result, we need to introduce some preliminaries.

Let $\Phi=\{f_{a}\}_{a\in \cA}$ be an iterated function system. As the maps are contracting, there exists a unique non-empty compact set satisfying the relation
    $$X_\Phi = \bigcup_{a \in \cA} f_a (X_\Phi).$$ For each $a\in \cA$ we let $X_{a}=f_{a}(X)$. We say that $\Phi$ satisfies the strong separation condition if $$X_{a}\cap X_{b}=\emptyset$$ for distinct $a,b\in \cA.$
    $X_{\Phi}$ can be viewed as the image of the full shift $\cA^{\N}$ under the map $\pi:\cA^{\N}\to X_{\Phi}$ given by $$\pi(\a) = \lim_{n \to \infty} (f_{a_1}\circ \cdots \circ f_{a_n})(x_0),$$ where $x_0 \in X_{\Phi}$ is an arbitrary point. By considering a small neighborhood of each $X_{a},$ it can be shown that under the strong separation condition for each $a\in \cA$ there exists an open set $U_{a}\subset \R^d$ satisfying the following properties:
    \begin{itemize}
        \item $U_{a}\cap U_{b}=\emptyset$ for $a\neq b$
        \item If we let $U:=\cup_{a\in \cA}U_{a}$ then $f_{a}(U)\subset U_{a}$ for each $a\in \cA.$
    \end{itemize}
In what follows, when an IFS satisfies the strong separation condition we will fix a choice of $U_{a}$ for each $a\in \cA$ and a corresponding $U.$

Given a $\#\cA\times \#\cA$ matrix $A$ taking values in $\{0,1\}$ we can define the associated  \textit{subshift of finite type} to mean the subset $\Sigma_A \subset \cA^\N$ given by $$\Sigma_A = \set{(a_i) \in \cA^\infty : A(a_{i},a_{i+1}) = 1}.$$ Thus each such matrix $A$ defines a compact subset $X_A \subset X_\Phi$ via the projection $X_A = \pi(\Sigma_A)$. We will always assume that $\Sigma_{A}$ is topologically mixing, this is equivalent to $A^{n}>0$ for all $n$ sufficiently large. We write $\cW_{A}$ as the set of all admissible words, that is, words that exist as a prefix of any word in $\Sigma_A$. Given $\a=(a_1,\ldots,a_n)\in \cW_{A}$ we define  
 $$X_{\a}^{A}=\pi\left(\left\{(b_i)\in \Sigma_{A}:b_{1},\ldots b_{n}=\a\right\}\right).$$ 
When the choice of matrix $A$ is implicit, we will by an abuse of notation write $X_{\a}$ instead of $X_{\a}^{A}.$ Given $\a,\b \in \cW_{A}$ we write $\a \leadsto \b$ if $\a\b\in \cW_{A}$. Let us assume that $\Phi$ satisfies the strong separation condition, given a point $x\in U$ then $x\in U_{b}$ for some unique $b\in \cA$, we write $\a \leadsto x$ for $\a\in \cW_{A}$ if $\a \leadsto b$.

  Given a subshift of finite type $\Sigma_A$ and a $C^1$ function $\psi : U \to \R,$ if $x\in U,$ $a\in \cA$ and $a\leadsto x,$ then we define $w_{a}(x)=e^{\psi(f_{a}(x))}.$ More generally, given $\a=(a_1,\ldots,a_n)\in \cW_{A}$ such that $\a \leadsto x,$ we define $w_{\a}(x)=e^{\sum_{j=0}^{n-1}\psi(f_{a_{j+1}\ldots a_{n}}(x))}.$
  We define the \textit{pressure} of $\psi$ to be
$$P(\psi) = \lim_{n \to \infty} \frac{1}{n} \log\Big(\sum_{\a \in \cA^n: \a\in \cW_{A}}\sup_{x\in \Sigma_{A}:\a\leadsto x} w_{\a}(x)\Big).$$   For a proof that the pressure of a $C^1$ function always exists see \cite{Walters}. For each $\psi$, there exists a unique probability measure $\mu_{\psi}$ supported on $X_{A}$ called the Gibbs measure of $\psi.$ The defining property of this measure is that it satisfies the Gibbs condition: There exists $C>0$ such that for all $\a\in \cW_{A}$ and $x\in X_{A}$ satisfying $\a\leadsto x$ we have \begin{align}
\label{eq:GibbsProperty}C^{-1} \leq \frac{\mu_{\psi}(X_{\a})}{w_{\a}(x)e^{nP(\psi)}} \leq C.
\end{align} Gibbs measures are well studied within the ergodic theory community. See \cite{Bowen} and \cite{ParryPollicott} for a more thorough introduction to these measures.

Given a conformal IFS $\Phi$ satisfying the strong separation condition, a subshift of finite type $\Sigma_{A}$ and a $C^{1}$ potential $\psi:U\to \mathbb{R}$ satisfying $P(\psi)=0$, we define the twisted transfer operators $\mathcal{L}_{ib}$, $b \in \R$, by the formula
$$\mathcal{L}_{ib} h (x) := \sum_{\substack{a\in \cA \\ a \leadsto x}}w_{a}(x)|\lambda_{a}(x)|^{ib} h(f_a(x)), \quad h \in C^{1,b}(U), x \in U,$$
in the Banach space $C^{1,b}(U)$ of differentiable functions equipped with the norm
$$\|h\|_b := \|h\|_\infty + \frac{\sup_{x\in U}\|D_{x}h\|}{|b|}.$$ 

\begin{definition}
    Let $\Phi$, $\Sigma_{A}$, and $\psi$ be given. We say that $(\Phi,\Sigma_{A},\psi)$ has a spectral gap if there exists $0 < \rho,\upsilon < 1$ such that for all large enough $|b|$, $n \in \N$ and $h \in C^{1,b}(U)$, we have
$$\|\mathcal{L}_{ib}^{n} h\|_{b} \ll \rho^{n}|b|^{\upsilon}\|h\|_b.$$
\end{definition}

Our main result in this section is the following general theorem.

\begin{thm}
\label{thm:mainNonlinear}
Let $\Phi$ be a conformal IFS satisfying the strong separation condition, let $\Sigma_{A}$ be a subshift of finite type, and let $\psi$ be a $C^{1}$ potential satisfying $P(\psi)=0$. Assume that $(\Phi,\Sigma_{A},\psi)$ has a spectral gap and suppose that $\mu_{\psi}$ is uniformly affinely non-concentrated. Then, $\mu_{\psi}$ has polynomial Fourier decay. 
\end{thm}
To the best of the authors' knowledge, Theorem \ref{thm:mainNonlinear} is the first polynomial Fourier decay result for self-conformal iterated function systems and Gibbs measures in arbitrary dimensions.

It is a well established fact that a suitable uniform non-integrability (UNI) condition can imply the existence of a spectral gap for a twisted transfer operator. See for instance the works of Naud \cite{Naud-Cantor} and Stoyanov \cite[Theorem 1.1]{Stoyanov} who adapted Dolgopyat's method \cite{Dolgopyat} to prove spectral gap results. We also refer the reader to the works of Avila, Gou\"ezel and Yoccoz \cite{AGY}, Sarkar-Winter \cite{SarkarWinter} and Li-Pan \cite{LiPan} for similar spectral gap results. The uniform non-integrability condition we will work with is the following definition. It is inspired by the UNI condition introduced by Li and Pan \cite[Lemma 4.5]{LiPan} in the context of Patterson-Sullivan measures.

\begin{definition}[UNI]\label{def:UNI}
Let $\Phi$ be a conformal IFS satisfying the strong separation condition and $\Sigma_{A}$ be a subshift of finite type. We say that the pair $(\Phi,\Sigma_{A})$ satisfies the \textit{Uniform Non-Integrability} (UNI) condition if  there exists $\epsilon_{0}>0$ such that the following holds for infinitely many $n\in \N$: There exists $x\in X_{A}$ such that for any unit vector $e\in\mathbb{R}^{d}$ there exists $\a_{1},\a_{2}\in \cW_{A}\cap \cA^{n}$ such that: 
    \begin{itemize}
    \item We have $$\left|\partial_{e}\left(\log |\lambda_{\a_1}(x)|-\log |\lambda_{\a_2}(x)|\right)\right|\geq \epsilon_{0}.$$
        \item $\a_1\leadsto x$ and $\a_2\leadsto x$.
    \end{itemize} 
\end{definition} 

Roughly speaking, the UNI condition holds if for infinitely many $n$ one can find an $x\in \Sigma_{A}$ at which we can observe non-linearity in all directions. We include some further discussion on this definition and some examples of iterated function systems satisfying it at the end of this subsection. The significance of our UNI condition is demonstrated by the following proposition.

\begin{prop}[Spectral gap] \label{prop:spectralgap} 
Let $\Phi$ be a conformal IFS satisfying the strong separation condition and $\Sigma_{A}$ be a subshift of finite type. Suppose that $(\Phi,\Sigma_{A})$ satisfies the UNI condition, then for any $C^{1}$ potential $\psi$ satisfying $P(\psi)=0,$  there exists $0 < \rho < 1$ such that for all large enough $|b|$, $n \in \N$ and $h \in C^{1,b}(U)$, we have
$$\|\mathcal{L}_{ib}^{n} h\|_{b} \ll \rho^{n}|b|^{1/2}\|h\|_b.$$ In particular, if $(\Phi,\Sigma_{A})$ satisfies the UNI condition, then for any $C^{1}$ potential $\psi$ satisfying $P(\psi)=0$ the triple $(\Phi,\Sigma_{A},\psi)$ has a spectral gap.
\end{prop}
The proof of Proposition \ref{prop:spectralgap} is provided in Appendix \ref{appendix:Spectral gap}. Our proof of Proposition \ref{prop:spectralgap} relies heavily on arguments due to Li and Pan \cite{LiPan}, and to Naud \cite{Naud-Cantor}.

Given a Gibbs measure $\mu_{\psi}$ associated to some $C^1$ potential $\psi,$ it can be shown that $\mu_{\psi}$ is the Gibbs measure for another $C^{1}$ function $\psi'$ that satisfies $P(\psi')=0.$ Combining this observation with Theorem \ref{thm:mainNonlinear} and Proposition \ref{prop:spectralgap} yields the following statement. 

 \begin{thm}
		\label{thm:UNINonlinear}
		Let $\Phi$ be a conformal IFS satisfying the strong separation condition and $\Sigma_{A}$ be a subshift of finite type. Suppose that $(\Phi,\Sigma_{A})$ satisfies the UNI condition. Then any Gibbs measure $\mu_{\psi}$ that is uniformly affinely non-concentrated has polynomial Fourier decay.	
	\end{thm}
In \cite{LiPan} Li and Pan introduced the following seemingly stronger uniform non-integrability condition: there exists $r>0$ and $\epsilon_{0}>0$ such that for any large $n\in \N$, any $x\in X_A$ and unit vector $e\in\mathbb{R}^{d}$, there exist $\a_1, \a_2\in \cW_{A}\cap \cA^n$ such that:
    \begin{itemize}
    \item For all $y\in B(x,r)$ we have $$\left|\partial_{e}\left(\log |\lambda_{\a_1}(y)|-\log |\lambda_{\a_2}(y)|\right)\right|\geq \epsilon_{0}.$$
        \item For all $y\in B(x,r)$ we have $\a_1\leadsto y$ and $\a_2\leadsto y$.
    \end{itemize} 
This condition and our UNI condition are in fact equivalent. We include a proof of this fact in Appendix \ref{appendix:Weak uni}. This proof follows an argument of Avila, Gou\"{e}zel and Yoccoz \cite{AGY}. It will be this latter condition that we will use in Appendix \ref{appendix:Spectral gap} during our proof of Proposition \ref{prop:spectralgap}.

	\begin{remark}
 \label{Remark}
		\begin{itemize}
	\item[(a)] Our proof of Theorem \ref{thm:mainNonlinear} relies on showing that the derivatives arising from our IFS satisfy a non-concentration property. We prove that this property holds using our spectral gap assumption. A similar non-concentration statement is needed in the proof of Theorem \ref{thm:PS}. However, in our proof of Theorem \ref{thm:PS} we do not use a spectral gap for the transfer operator and instead use explicit formulas for the derivatives directly.
    It would be interesting if the non-concentration of the derivatives in the setting of Theorem~\ref{thm:mainNonlinear} could be also deduced directly from a suitable UNI condition without using a spectral gap assumption. 
    We leave this as an interesting problem to pursue.
 
	\item[(b)] In dimension $d=1$, Theorem \ref{thm:mainNonlinear} was proved in the work of the third author and Stevens \cite{SahlstenStevens}, but as in Bourgain and Dyatlov \cite{BourgainDyatlov}, it relied on Bourgain's sum product theorem. Here we do not need it, so Theorem \ref{thm:mainNonlinear} also gives a new proof in the case $d = 1$.  
 
 \item[(c)]
 For dimension $d = 2$, Theorem \ref{thm:mainNonlinear} extends the recent work of Leclerc \cite{Leclerc-JuliaSets}, regarding Fourier dimension of Julia sets, which used a spectral gap result for twisted transfer operators provided by Oh-Winter \cite{OhWinter}. 
    More recently, in dimension $d=2$, the spectral gap theorem of Oh-Winter was generalized in~\cite{AlgomRHWang-SelfConformalDimensionTwo} to prove polynomial Fourier decay for stationary measures for nonlinear $C^\omega$ conformal IFSs without any separation assumptions for the underlying IFS. For comparison, Theorem~\ref{thm:mainNonlinear} only requires the IFS to be $C^2$ and it applies in arbitrary dimensions and to general Gibbs measures.

 \item[(d)] 
Theorem \ref{thm:UNINonlinear} is a statement that guarantees polynomial Fourier decay for Gibbs measures assuming the subshift satisfies a non-linearity assumption that is formulated in terms of the Jacobian matrices $\{D_{x}f_{a}=\lambda_{a}(x)O_{a}(x)\}_{a\in \cA}$. In particular, we use properties of the contraction ratios $\{\lambda_{a}(\cdot)\}_{a\in \cA}$ to prove our result.
In the spirit of Theorem~\ref{thm:selfsim rotations}, it would be interesting if an analogous result to Theorem~\ref{thm:mainNonlinear} could be proved using properties of the orthogonal matrices $\{O_{a}(\cdot )\}_{a\in \cA}$ instead.

	\end{itemize}
	\end{remark}

 \subsubsection{Examples}
    In this subsection, we discuss examples where the hypotheses of Theorems~\ref{thm:mainNonlinear} and~\ref{thm:UNINonlinear} are satisfied.

  Whether a Gibbs measure satisfies the uniformly affinely non-concentrated property is a well investigated problem. In the context of the full shift, the simplest example of a Gibbs measure satisfying this property is when $\psi$ is the geometric potential, $\Phi$ satisfies the strong separation condition and $\dim_{H}(X_{\Phi})>d-1$. Under these assumptions it is well known that $\mu_{\psi}(B(x,r))\asymp r^{\dim_{H}(X)}$ for any $x\in X_{\Phi}.$ It is straightforward to apply these properties to prove that $\mu_{\psi}$ satisfies the uniformly affinely non-concentrated property.
 In the case of arbitrary Gibbs measures we refer the reader to \cite{Dasetal,Urbanski} and the references therein. In particular, Theorem 1.5 from \cite{Urbanski} (see also Theorem 1.10 from \cite{Dasetal}) shows that under natural assumptions on the conformal IFS and the Gibbs measure $\mu_{\psi},$ if the attractor of the IFS is not contained in a proper real analytic submanifold of $\R^{d}$ then the Gibbs measure $\mu_{\psi}$ will be uniformly affinely non-concentrated.

Turning our attention to our UNI hypothesis, the condition in Definition~\ref{def:UNI} is inspired by the work of Li and Pan \cite{LiPan}, where a similar condition was introduced in the context of PS measures on limit sets of cusped, geometrically finite, groups of M\"obius transformations.
In light of the close relationship between such limit sets and attractors of conformal IFSs, it seems possible to extend the techniques of the proof of Theorem~\ref{thm:mainNonlinear} to recover the polynomial Fourier decay result of Theorem~\ref{thm:PS} as well as its generalizations to other Gibbs measures. To this end, the results of~\cite{LiPan} can be adapted to construct a coding of limit sets of convex cocompact groups using suitable Markov systems satisfying our UNI hypotheses. 
 As noted before, this method yields worse dependence of the rates of decay on geometric non-concentration properties of the measure due to our use of spectral gap results in the proof.
 
 Nonetheless, motivated by this connection, we provide explicit further examples of attractors of conformal IFSs in $\R^d$ satisfying the hypotheses of Theorem~\ref{thm:UNINonlinear}. The examples we give are based on certain involution M\"obius transformations on $\R^d$ (see \eqref{Eq:nonlinear conformal hd} below). We note that, in dimension $d \geq 3$, Liouville's Theorem (see \cite[Chapter 5]{IwaMar}) states that \textit{any} conformal mapping is either of this form or is a similarity map, so the class of maps in this example is not very restrictive. Our examples will be formulated in the special case when $\Sigma_{A}$ is the full shift $\cA^{\N}$, but they can easily be adapted to incorporate more general subshifts of finite type. 

\begin{example}
\label{Example:UNI example}
Let $\Phi=\{f_{a}:[0,1]^{d}\to [0,1]^{d}\}_{a\in \cA}$ be a conformal IFS. Assume that $\Phi$ contains $d+1$ contractions of the form
\begin{equation}
\label{Eq:nonlinear conformal hd}
f_{a_i}(x)=t_{a_i}+\frac{\lambda_{a_i}O_{a_i}(x-u_{a_i})}{|x-u_{a_{i}}|^{2}}
\end{equation} where $t_{a_i}\in \mathbb{R}^{d},$ $u_{a_i}\in \mathbb{R}^{d}\setminus [0,1]^{d},$ $\lambda_{a_i}\in (0,1),$ and $O_{a_i}\in O(d).$ Let us also suppose that the set of vectors $\{u_{a_i}\}_{i=1}^{d+1}$ do not belong to an affine subspace of $\mathbb{R}^d$. Maps satisfying \eqref{Eq:nonlinear conformal hd} have the following useful property: Let $x\in [0,1]^{d}$ and $e\in \mathbb{R}^{d}$ be a unit vector. Then for each $a_{i}$ we have 
\begin{equation}
\label{Eq:partial derivative for conformal maps}
\partial_{e}\log|\lambda_{a_{i}}(x)|=-\frac{2\langle x-u_{a_i},e\rangle}{|x-u_{a_i}|^{2}}.
\end{equation} This fact was proved in \cite[Lemma 6.75]{LiPan}. It can be shown to hold by a direct calculation. Let us also assume now that $\Phi$ contains a similarity given by:
$$f_{a_s}(x)=\lambda_{a_s}O_{a_{s}}x+t_{a_s}$$ for some $t_{a_s}\in \mathbb{R}^{d},$ $\lambda_{a_s}\in (0,1),$ and $O_{a_s}\in O(d).$

We now bring our attention to proving that the UNI condition is satisfied by any conformal IFS satisfying the above conditions. Let $n\in \N$ be arbitrary and $x\in X_{\Phi}.$ Then for any unit vector $e\in \mathbb{R}^{d}$ and $a_i$ we have 
\begin{align*}
\partial_{e}\left(\log|\lambda_{a_{s}^{n-1}a_{i}}(x)|-\log|\lambda_{a_{s}^{n}}(x)|\right)
=&\partial_{e}\left(\log \lambda_{a_{s}}^{n-1}|\lambda_{a_{i}}(x)|-\log \lambda_{a_{s}}^{n}\right)\\
=&\partial_{e}\log|\lambda_{a_{i}}(x)|\\
=& -\frac{2\langle x-u_{a_i},e\rangle}{|x-u_{a_i}|^{2}},
\end{align*} where the final line follows by \eqref{Eq:partial derivative for conformal maps}. In particular, the expression obtained for $\partial_{e}(\log|\lambda_{a_{s}^{n-1}a_{i}}(x)|-\log|\lambda_{a_{s}^{n}}(x)|)$ does not depend upon $n$. Thus, the UNI condition is satisfied if for any unit vector $e\in\mathbb{R}^{d}$ we can find $a_i$ such that $$\langle x-u_{a_i},e\rangle\neq 0.$$ Such an $a_i$ has to exist because of our assumption that the set of vectors $\{u_{a_i}\}_{i=1}^{d+1}$ do not belong to an affine subspace of $\mathbb{R}^d.$

We emphasise that it is possible to choose the contractions in this example in such a way that $\dim_{H}(X_{\Phi})>d-1$ and the strong separation condition is satisfied by $\Phi$. Thus by the discussion in above, the Gibbs measure corresponding to the geometric potential would in this case provide an explicit example where Theorem \ref{thm:UNINonlinear} guarantees polynomial Fourier decay.
\end{example}

We end this discussion with several remarks and questions. Firstly, it is easy to check that the set of $C^2$ conformal IFSs satisfying our UNI condition is open, while the above example shows that this set is non-empty. Moreover, close examination of Example~\ref{Example:UNI example} suggests that failure of the UNI condition places algebraic relations among the maps in the IFS. It is thus natural to expect that our UNI condition holds generically.

This discussion, along with Liouville's Theorem, raises the following question: Suppose $\Phi$ is a conformal IFS acting on $\R^{d}$ for $d\geq 3$. If $\Phi$ contains a map of the form described by \eqref{Eq:nonlinear conformal hd}, then must $\Phi$ satisfy the UNI condition? If this question could be answered in the affirmative, then Theorem \ref{thm:UNINonlinear} would go a long way towards classifying polynomial Fourier decay for Gibbs measures in dimension $d\geq 3.$

Finally, in view of the restriction placed by Liouville's Theorem on conformal IFSs in dimensions $d\geq 3$, it is natural to consider generalizations of Theorem~\ref{thm:UNINonlinear} by relaxing the conformality hypothesis to non-algebraic systems provided by \textit{Quasiconformal-} or \textit{Quasiregular Mappings}, with the price of lack of regularity of the distortion function $x \mapsto \log |f_a'(x)|$.
The study of such systems has received a lot of attention in recent years; cf. \cite{IwaniecMartin2,Meyer,OkuyamaPankka}.
In particular, it is of interest to extend the proofs of Theorems~\ref{thm:mainNonlinear} and~\ref{thm:UNINonlinear} to these closely related systems thus yielding a wider source of examples of polynomial Fourier decay; see Section \ref{sec:UQR} for further discussion of this direction.

\subsection{Non-conformal systems}\label{sec:nonconf}
Our techniques also allow us to prove polynomial Fourier decay for certain stationary measures arising from non-conformal iterated function systems. It is well-known that such measures are more difficult to analyze than self-similar and self-conformal measures. Indeed the question of whether such a measure is uniformly affinely non-concentrated is far more delicate than in the self-similar and self-conformal setting. 
As such our knowledge of the Fourier decay properties of stationary measures for non-conformal iterated function systems is far less extensive. 

A result of Solomyak \cite{Solomyak2} demonstrates that amongst parameterised families of self-affine measures one should expect a typical member to have polynomial Fourier decay. A result of Li and the third author \cite{LiSahlsten-SelfAffine} established polynomial Fourier decay for self-affine measures when the underlying IFS satisfies suitable algebraic assumptions. To the best of the authors' knowledge, our main result (Theorem \ref{Theorem:Non-conformal}) is the first general result in this setting that establishes polynomial Fourier decay using the non-linearity within the IFS.

The family of non-conformal IFSs we consider is defined as follows. Suppose that, for each $1\leq i\leq d$, we are given a $C^{2}$ IFS $\set{f_{a}^{(i)}:a\in \cA_i}$ acting on $[0,1].$ Given $\overline{a}=(a_1,\ldots,a_d)\in \cA_1\times \cdots \times\cA_{d}$ we define the map $$F_{\overline{a}}(x_1,\ldots,x_{d})=(f_{a_{1}}^{(1)}(x_1),\ldots,f_{a_{d}}^{(d)}(x_{d})).$$ Using this notation, given any $\cA\subset \cA_1\times \cdots \times\cA_{d}$ we can define a new IFS $\set{F_{\overline{a}}}_{\overline{a}\in \cA}$ acting on $[0,1]^{d}$. We will call such an IFS a \textit{restricted product IFS}.

Given a restricted product IFS $\set{F_{\overline{a}}}_{\overline{a}\in \cA}$ and a probability vector $\p=(p_{\overline{a}})_{\overline{a}\in \cA},$ we let $\mu_{\p}$ denote the unique Borel probability measure satisfying $$\mu_{\p}=\sum_{\oa\in \cA}p_{\oa}F_{\oa}\mu_{\p}.$$ We call $\mu_{\p}$ a stationary measure. 

In this setting, it is reasonable to expect that if the IFS is suitably non-linear then every stationary measure will have polynomial Fourier decay. In this section our notion of non-linearity comes from the following definition.

\begin{definition}
Let $\Phi=\{f_a\}$ be a $C^2$ IFS acting on $[0,1]$ with attractor $X_{\Phi}$. We say that $\Phi$ satisfies the Uniform Non-integrability (UNI) condition if there exists $\epsilon_{0}>0$ such that for all $n\in\N$ sufficiently large there exists $\a,\b\in \cA^{n}$ satisfying $$\epsilon_{0}<\left|(\log |f_{\a}'|- \log |f_{\b}'|)'(x)\right|\footnote{This definition often also includes the upper bound $\left|(\log |f_{\a}'|- \log |f_{\b}'|)'(x)\right|<\epsilon_{1}$ for all $x\in X_{\Phi}$ for some $\epsilon_{1}>0$. However the existence of such an $\epsilon_{1}>0$ follows automatically from the regularity of the underlying IFS. Hence we omit it.}$$ for all $x\in X_{\Phi}.$
\end{definition}

The Fourier decay properties of measures arising from iterated function systems satisfying this UNI condition have been well studied (see \cite{AlgomRHWang-Polynomial,BakerSahlsten,SahlstenStevens}). 
In the context of one dimensional $C^{2}$ iterated function systems, it is known (see for example \cite{AlgomRHWang-Polynomial}) that if an IFS $\Phi$ cannot be $C^{2}$ conjugated to a linear IFS, i.e. there exists no $C^2$ diffeomorphism $h:\R\to \R$ such that $\{g_{a}:=h\circ f_{a}\circ h^{-1}\}_{a\in \cA}$ satisfies $g_{a}''(x)=0$ for all $x\in X_{\Phi}$ and $a\in \cA$, then $\Phi$ satisfies the UNI condition. Using the argument given in the proof of Proposition \ref{prop:weaktostrongUNI} it can be shown that this UNI condition is equivalent to that given in Definition \ref{def:UNI} when $d=1.$ This justifies our terminology. We choose to use this definition as it coincides with the UNI condition used in \cite{BakerSahlsten}.

Our main result in this section is the following statement.

\begin{thm}
    	\label{Theorem:Non-conformal}
Let $\set{F_{\overline{a}}}_{\overline{a}\in \cA}$ be a restricted product IFS. Assume that the following properties are satisfied: \begin{enumerate}
	\item For each $1\leq i\leq d$, the one-dimensional IFS $\set{f_{a}^{(i)}}_{a\in \cA_i}$ satisfies the strong separation condition.
	\item For each $\overline{a}\in \cA$ and $1\leq i\leq d$, there exists $\oa'\in \cA$ such that $a_{j}=a_{j}'$ for all $j\neq i$ and $a_{i}\neq a_{i}'.$ 
	\item For each $1\leq i\leq d$, there exists $\oa\in \cA$ such that the IFS $$\set{f_{a'_{i}}^{(i)}: (a_1',\ldots,a_d')\in \cA,\, a'_{j}=a_{j}\, \forall j\neq i}$$ satisfies the UNI condition. 
\end{enumerate}
Then, every stationary measure has polynomial Fourier decay.
\end{thm}

To help illuminate Theorem \ref{Theorem:Non-conformal} we include an example of an IFS satisfying its assumptions: Let
$$\Phi:=\left\{f_{(i,j)}(x,y)=\left(\frac{1}{x+i},\frac{1}{y+j}\right):(i,j)\in \{(1,2),(1,3),(2,1),(2,3),(3,1),(3,2)\}\right\}.$$ Theorem \ref{Theorem:Non-conformal} implies that all of the stationary measures for this IFS have polynomial Fourier decay.

\begin{remark}
The significance of defining our restricted product IFSs using products of one-dimensional IFSs, as opposed to products of $d$-dimensional IFSs, is that this property will allow us to disintegrate a stationary measure into an integral of random one-dimensional measures. These random measures were studied in \cite{BakerSahlsten}. In particular, our proof of Theorem \ref{Theorem:Non-conformal} will rely upon a result from \cite{BakerSahlsten} which establishes a spectral gap for typical random compositions of transfer operators (see Proposition \ref{Prop:Spectral gap on average}).
Hence, a higher dimensional generalization of the results in~\cite{BakerSahlsten} would lead to a version of Theorem~\ref{Theorem:Non-conformal} for products of higher dimensional IFSs.

\end{remark}


\subsection{Applications}\label{sec:applications}
We now discuss some of the consequences of the results presented here to equidistribution of vectors in fractal sets and quantum chaos.

\subsubsection{Equidistribution of vectors on fractals}
\label{sec:equidistribution}

The study of Diophantine properties of typical points in the supports of dynamically defined measures has witnessed considerable activity in recent years; cf.~\cite{DavenportErdosLeVeque,SW,DGW,Host,PVZZ} for instance and references therein.

In dimension one, an underlying principle behind many of the results in the subject can be summarized as follows: the continued fraction expansion as well as the digit expansions in different bases of a typical typical point should be independent of one another.
Higher dimensional results also echo analogous principles.

This principle is exemplified by the following results.
Host's theorem~\cite{Host}, and its generalization in~\cite{HochmanShmerkin-Host}, assert that if multiplicatively independent integers $p,q\in\N$ and a $\times p$-invariant and ergodic measure $\mu$ on $[0,1)$ of positive entropy are given, then the $\times q$-orbit of $\mu$-almost every point is equidistributed with respect to Lebesgue measure; see also~\cite{Hochman-Host} for a Fourier analytic proof of this result.

A second example is given by a result of Simmons and Weiss~\cite{SW} asserting that if $\mu$ is a non-trivial self-similar measure on $[0,1)$, then the orbit of $\mu$-almost every point under the Gauss map is equidistributed towards the unique absolutely continuous Gauss measure.

Much less is known in higher dimensions. In~\cite{SW}, the authors obtain higher dimensional analogues of their aforementioned results for the Gauss map.
In~\cite{DGW} Dayan, Ganguly and Weiss show almost sure $A$-normality for certain self-affine measures with linear parts given by negative powers of $A$ and satisfying additional Diophantine conditions on their translation parts.
In both cases, the results are based on measure rigidity results for certain random walks on homogeneous spaces.

In this vein, Fourier decay has provided an important avenue for obtaining new results for measures that are not amenable to analysis by other methods.
Most notably, in~\cite{DavenportErdosLeVeque}, Davenport, Erd\Horig{o}s, and LeVeque (DEL) showed that any\footnote{The reference \cite{DavenportErdosLeVeque} considered the case of dimension $1$, see Appendix \ref{appendix:DELhigh} for a proof in dimension $d \geq 2$.} probability measure $\mu$ on $\R^d$ with polylogarithmic Fourier decay has that $\mu$-almost every $x$ is $A$-normal with respect to any expanding integer matrix $A$, i.e., the orbit $(A^n x)_{n\in \N}$ is equidistributed in the torus $\R^d/\Z^d$ when reduced $\mrm{mod}\; 1$; cf.~\cite{PVZZ} for a recent extension of this result.
Here, a matrix is expanding if its determinant has modulus larger than one.
In light of this criterion, our results provide numerous new examples where the above principle holds.

\begin{cor}\label{cor:mainEquidistribution} Let $\mu$ be any of the measures in the Theorems \ref{thm:selfsim}, \ref{thm:selfsim rotations} , \ref{thm:PS}, \ref{thm:mainNonlinear} or \ref{Theorem:Non-conformal}. Then, we have that $\mu$-almost every $x$ is $A$-normal with respect to any expanding integer valued matrix $A$.
\end{cor}

We also note that Corollary~\ref{cor:mainEquidistribution} opens the possibility of generalizing the results of~\cite{PVZZ} to higher dimensions.

\subsubsection{Fractal Uncertainty Principles, spectral gaps, and scattering resonances}
\label{sec:FUP}

Our next application concerns essential spectral gaps for Laplacians on convex cocompact hyperbolic manifolds. To properly place our results in context, we introduce the relevant notions and previous progress on this problem.

Let $\Gamma$ be a convex cocompact subgroup of $\mrm{Isom}^+(\H^{d+1})$, $d\geq 1$, and denote by $M$ the quotient manifold $\H^{d+1}/\G$.
As before, we let $\d_\G$ denote the Hausdorff dimension of the limit set $\L_\G$ of $\G$.
We let $\Delta_M$ denote the Laplace-Beltrami operator on $M$ and for $\l\in \C$ with sufficiently large imaginary part, we define the resolvent of $\Delta_M$ by $R(\l) := (-\Delta_M - \frac{d^2}{4} - \lambda^2)^{-1} $.

A topic of very active interest concerns meromorphic continuation of the analytic family of operators $\l\mapsto R(\l)$ as well as the existence and location of its poles.
Results of this type have important applications to rates of decay of solutions to wave equations in quantum chaos as well as to rates of decay of correlation in the field of hyperbolic dynamics.
We refer the reader to~\cite{DyatlovZworski} for an introduction to the modern aspects of the topic.

In this regard, Patterson-Sullivan theory based on studying fractal geometric properties of PS measures supported by $\L_\G$ provides a powerful method for establishing such results.
It follows from this theory that $R(\l)$ is well-defined and analytic on the domain $Im(\l)>-\beta_{PS}$, for $\beta_{PS} = \max\set{0,d/2-\d_\G}$.
Moreover, $\l=-\beta_{PS}$ is the only pole for $R(\l)$ on the line $\l = -\beta_{PS}$.
The constant $\beta_{PS}$ is referred to as the \textit{Patterson-Sullivan (PS) gap}.
We say that $R(\l)$ has an \textit{essential spectral gap} of size $\beta \geq \beta_{PS}$ if $R(\l)$ admits a meromorphic continuation with at most finitely many poles to the domain $Im(\l) \geq -\beta$.

The topic of producing essential spectral gaps and studying their dependence on geometric invariants of the underlying manifold has received considerable interest in recent years.
When $\d_\G \leq d/2$, Naud~\cite{Naud-Cantor} in the case $d=1$ and Stoyanov~\cite{Stoyanov} in higher dimensions proved the existence of an essential spectral gap using a refinement of the Dolgopyat method~\cite{Dolgopyat}.
In a major breakthrough, Bourgain and Dyatlov~\cite{BourgainDyatlov-Annals} produced an improvement over the PS gap in the regime $\d_\G >d/2$ and $d=1$ using tools from harmonic analysis by proving a Fractal Uncertainty Principle for Ahlfors-David regular sets (see the definition \eqref{eq:FUPdef} below). Extending the results of~\cite{BourgainDyatlov-Annals} to higher dimensions remains a major open problem, due to counterexamples arising from affinely concentrated sets \cite{Dyatlov-IntroFUP}. See the work by Cohen \cite{Cohen-FUP} for a recent breakthrough on this question and e.g. related works \cite{BackusLengTao,ADM,CladekTao,HS}.

In \cite{DyatlovZahl}, Dyatlov and Zahl introduced the notion of \textit{Fractal Uncertainty Principle} (FUP) that provides a way to produce essential spectral gaps with explicit dependence on the geometry of $M$.
We recall this notion here.
Two sets $X,Y \subset \R^d$ are said to satisfy the \textit{generalized FUP} with exponent $\beta > 0$ if for any open set $U \subset \R^d \times \R^d$, compact subset $V \subset U$, smooth functions $\Phi \in C^3(U,\R)$ and $G \in C^1(U,\C)$ satisfying $\mathrm{supp} (G) \subset V$ and $\|\Phi\|_{C^3} + \|G\|_{C^1} \leq C_{\Phi,G}, \quad \inf |\partial_{xy}^2 \Phi | \geq C_{\Phi,G}^{-1},$ for some constant $C_{\Phi,G} > 0$, then for any $0 < \rho < 1$ and $h \ll 1$, we have 
\begin{equation}\label{eq:FUPdef}
    \|\1_{X(h^\rho)} \cB(h) \1_{Y(h^{\rho})}\|_{L^2(\R^d) \to L^2(\R^d)} \leq C h^{\beta},    
\end{equation}
 where $\cB(h)$ is the Fourier integral operator defined by
 \begin{align*}
     \cB(h)u(x) = \frac{1}{(2\pi h)^{d/2}} \int_{\R^d} \exp\Big(\frac{i\Phi(x,y)}{h}\Big) G(x,y) u(y) \, dy,
 \end{align*}
  for all $u\in L^2(\R^d)$ and $x \in \R^d$. In \cite{DyatlovZahl}, it is shown that if the generalized FUP holds for $X = Y = \Lambda_\Gamma$  with some exponent $\beta > 0$, then there is an essential spectral gap of size $\beta > 0$.
  In particular, it follows from their results that if $\d_\G\leq d/2$ and $\mu_\G$ has polynomial Fourier decay with rate $\kappa>0$, then the generalized FUP holds with exponent $\beta =\beta_{PS}+\e$, for a constant $\e>0$ depending explicitly on $\kappa$.

When $d = 1$, using Dolgopyat's method \cite{Dolgopyat}, Dyatlov and Jin \cite{DyatlovJin} proved that $\L_\G$ satisfies a generalized FUP with exponent $\beta = \beta_{PS}+\e$  for a constant $\eps > 0$ depending on the Hausdorff dimension of the limit set $\delta_\Gamma$ and the \textit{Ahlfors-David regularity constant} $C_\Gamma$ of the PS measure $\mu_\Gamma$. Here, $C_\G\geq 1$ is a constant satisfying 
\begin{equation}\label{eq:AD reg}
    C_{\Gamma}^{-1} r^{\delta_\Gamma} \leq \mu_\Gamma(B(x,r)) \leq C_\Gamma r^{\delta_\Gamma}, \quad r>0, x \in \mathrm{supp}( \mu).    
\end{equation}
Dependence on such constants is a byproduct of the use of Dolgopyat's method.

To remove the $C_\Gamma$ dependency, Bourgain and Dyatlov \cite{BourgainDyatlov} proved that, in the case $d=1$, the PS measures have polynomial Fourier decay with a rate depending only on the Hausdorff dimension $\d_\G$.
When $\d_\G\leq 1/2$, this in particular produced an essential spectral gap of size independent of the constant $C_\G$.
Their proof used Bourgain's sum-product theorem as well as the non-linearity of the action of $\G$ on its limit set $\L_\G$ by M\"obius transformations.
See also~\cite{LiNaudPan} for an extension of this result to the case of Schottky Kleinian groups in the case $d=2$.

In higher dimensions, Backus, Leng and Tao \cite{BackusLengTao} generalized the work of Dyatlov and Jin \cite{DyatlovJin} using Dolgopyat's method \cite{Dolgopyat}.
They obtained an improvement over the PS gap that depends on the non-concentration properties of $\mu_\G$ as well as the \textit{upper doubling constant constant} $C_D$ of the PS measure. 
Here, $C_D\geq 1$ is a constant such that
$$\mu_\Gamma(B(x,2r)) \leq C_D \mu_\Gamma(B(x,r)), \quad r>0, x \in \mathrm{supp} (\mu).$$
Note that we can always choose $C_D \leq 2^{\delta_\Gamma} C_\Gamma^2$, where $C_\G$ is as in~\eqref{eq:AD reg}.
See also~\cite{CladekTao} for related results by different methods.
However, the question of producing an improvement that is independent of such constant \`a la Bourgain and Dyatlov remains open in higher dimensions.

Now, the advantage of Theorem \ref{thm:PS} is that it relies on the $L^2$-flattening theorem and non-concentration estimates for PS measures while avoiding Dolgopyat's method. 
In particular, we are able to produce an improvement over the PS gap that is independent of the doubling constant $C_D$.

	\begin{cor}
		\label{cor:gap}
  Let $\Gamma$ be a discrete, Zariski-dense, convex cocompact, group of isometries of real hyperbolic space $\H^{d+1}$, $d\geq 1$.
 Assume $\delta_\Gamma \leq d/2$. Then there exists $\eps > 0$ such that the limit set $\Lambda_\Gamma$ satisfies the generalized Fractal Uncertainty Principle with exponent $\beta = \tfrac{d}{2}-\delta_\Gamma + \eps$ and that there exists an essential spectral gap of size $\tfrac{d}{2}-\delta_\Gamma + \eps$. Here $\eps$ depends on the constants $C$ and $\alpha$ such that the Patterson-Sullivan measure $\mu_\Gamma$ is $(C,\alpha)$-uniformly affinely non-concentrated, but not on the doubling constant $C_D$.
	\end{cor}

 \begin{remark}\label{rem:FUP}
     Corollary~\ref{cor:gap} is a partial generalization of the results of Bourgain and Dyatlov to higher dimensions.
     A full generalization requires the removal of the dependence on the constant $C$ coming from uniform non-concentration.
     This can be achieved by removing the dependence of the function $\tau(\epsilon)$ in the flattening Theorem~\ref{thm:flattening} on $C$.
      \end{remark}

\subsection{Generalizations and future directions}

We end the introduction with some discussions of possible generalizations and future directions this work could take. 

\subsubsection{Lower regularity} Firstly given the interest in quantum chaos problems in variable curvature \cite{DyatlovZworski}, it would be interesting to obtain generalizations of our Fourier decay and FUP results for Patterson-Sullivan measures in this context, where one has to contend with the additional difficulty arising from lower regularity of the dynamics. 
In a similar vein, our proof of Theorem~\ref{thm:mainNonlinear} is restricted to $C^2$ IFSs, so it would also be interesting to have polynomial decay for $C^{1+\alpha}$ self-conformal IFSs, possibly using methods of Tsujii-Zhang~\cite{TsujiiZhang}. The problem of studying lower regularity also relates deeply to the UNI condition we use and going beyond the conformal category, which we discuss below.

\subsubsection{UNI condition and Julia sets in quasiregular dynamics}\label{sec:UQR}  Given that Liouville's theorem on the rigidity of conformal maps in dimensions $d \geq 3$, in developing analogues of complex dynamics in higher dimensions, the study of analogues of this theory for wider classes of maps known as quasiconformal- and quasiregular maps has received considerable interest; see e.g.  \cite{IwaMar,Meyer,OkuyamaPankka,Kangasniemi,KOPS} and references therein.
 In dimension $d = 2$, by a result of Oh and Winter \cite{OhWinter}, if $f : \mathbb{S}^2 \to \mathbb{S}^2$ is a hyperbolic rational map, then $\tau = \log |f'|$ satisfies a \textit{Non-Local Integrability} property on the Julia set $J_f$ (which is related to UNI, see \cite[Proposition 5.5]{Naud-Cantor}) if and only if $f$ is not conjugated to the power map $z \mapsto z^d$ for any integer $d \neq 0$. This was adapted by Leclerc \cite{Leclerc-JuliaSets} to prove polynomial Fourier decay for equilibrium states on Julia sets for such rational hyperbolic maps not conjugated to $z\mapsto z^d$.

In dimensions $d \geq 3$, analogues of hyperbolic rational maps are given by \textit{Uniformly Quasiregular (UQR) Maps}  $f : M \to M$ on a Riemannian manifold $M$ \cite{OkuyamaPankka,IwaniecMartin2}. These maps are known to have a fractal Julia set only when $M$ is a rational homology sphere \cite{Kangasniemi}. In this situation in $d = 3$, the analogues for the power maps $z \mapsto z^d$ come from certain Latt\`{e}s map constructions \cite{Meyer}.
In particular, it is natural to expect that maps not conjugated to such Latt\`{e}s maps provide a potential source of Julia sets with equilibrium states having polynomial Fourier decay. However, the problem with UQR maps is that they no longer satisfy the smoothness properties, which would allow us to define e.g. the UNI condition using $\partial_e \log |f'|$ as $\log |f'|$ is not generally even H\"older. Similar complications also appear in the case of Thurston maps of $\mathbb{S}^2$, where Li and Zheng \cite{LiZheng} generalized the work of Oh and Winter with a notion of $\alpha$-strong non-integrability condition allowing them to prove a spectral gap theorem. It would be an interesting problem to pursue these notions in dimensions $d \geq 3$ for Julia sets of UQR maps in order to prove and classify Fourier decay for the equilibrium states for these systems.

\subsubsection{Non-conformal measures} In a different direction, it would be interesting to extend our methods to obtain rates of decay for self-affine measures recovering and generalizing~\cite{LiSahlsten-SelfAffine}. Similarly, it would be interesting to see if the $L^2$ flattening approach could improve the results on polynomial decay for Furstenberg measures on projective spaces arising as stationary measures for random matrix products in the cases that have not yet been studied. 

\subsubsection{Self-similar measures} Given that we are able to get \textit{polylogarithmic} Fourier decay for a wide class of self-similar measures, this raises the question of whether our method can be pushed to obtain \textit{polynomial} decay for self-similar measures beyond known cases, see \cite{Sahlsten-survey} for further discussion of the existing examples.

\subsubsection{Beyond stationary measures} Finally, going beyond the deterministic category, there has recently been a lot of activity in the study of Fourier decay properties of randomly defined measures such as spatially independent martingales \cite{SuomalaShmerkin}, Liouville Quantum Gravity \cite{FalconerJin2019}, random multiplicative cascades \cite{GaVargas} and random Cantor measures in the study of Fractal Uncertainty Principles \cite{HanP}. It would be interesting to explore whether the statistical multiscale structure of these measures can be be used as a substitute for \textbf{Step 1} in our strategy towards studying rates of Fourier decay.

\subsection{Organisation of the article} 

  For convenience of the reader, the proofs of the results for different classes of measures appearing in the introduction are presented in such a way that they can be read independently of one another.
  
After collecting general preliminaries and notations in Section~\ref{Sec:preliminaries}, we prove Theorems~\ref{thm:selfsim} and~\ref{thm:selfsim rotations} in Section~\ref{sec:selfsim}.
Section \ref{sec:proofps} is dedicated to the case of Patterson-Sullivan measures where we prove Theorem~\ref{thm:PS}. In Section \ref{sec:proofnon-linear}, the general non-linear self-conformal case is established and, finally, in Section \ref{sec:proofnonconf}, we address the non-conformal case.

The reader will find that the overall structure of the proof in the case of self-similar, self-conformal, and non-conformal systems is very similar due to the presence of a natural underlying symbolic coding of the dynamics.
In the case of PS measures, we do not rely on any coding and instead prove Theorem~\ref{thm:PS} directly using dynamics of the geodesic flow as a substitute for the shift.
However, the strategy remains the same in all cases.

The article has four appendices establishing for completeness some results not readily available in the literature.
In Appendix~\ref{appendix:DELhigh}, we prove a higher dimensional generalization of the Davenport-Erd\Horig{o}s-LeVeque criterion for equidistribution which yields Corollary~\ref{cor:mainEquidistribution} when combined with our results on Fourier decay.  Appendix~\ref{appendix:Multinomial} is dedicated to the proof of an auxiliary estimate on multinomial distributions required for the proof of the self-similar case.
In Appendix~\ref{appendix:Spectral gap}, we provide a proof of Proposition~\ref{prop:spectralgap} on the spectral gap of twisted transfer operators.
Finally, in Appendix~\ref{appendix:Weak uni} we prove Proposition \ref{prop:weaktostrongUNI} which establishes equivalence of the UNI criterion we use to the one introduced by Li and Pan \cite{LiPan} in their work on exponential mixing.

\begin{acknowledgement}
    S.B. is supported by an EPSRC New Investigator Award (EP/W003880/1). O.K. is partially supported under NSF grant DMS-2247713 and DMS-2337911. T.S. is supported by the Academy of Finland via the project \emph{Quantum chaos of large and many body systems}, grant Nos. 347365, 353738. The proof in Appendix~\ref{appendix:DELhigh} was done by the third author and Jonathan Fraser in an unpublished preprint, and we thank Jonathan for allowing us to include it in the appendix. We also thank Jialun Li for comments on an earlier draft.
\end{acknowledgement}

\section{Preliminaries}
\label{Sec:preliminaries}
In this section, we collect some notation and technical results that we use throughout the article.

\subsection{Preliminaries for iterated function systems and thermodynamic formalism.}

Given a finite set $\cA$ we let $\cA^{*}=\cup_{n=1}^{\infty}\cA^{n}$ denote the corresponding set of finite words. Given $\a=(a_1,\ldots,a_n)\in \cA^{*}$ we let $|\a|$ denote its length and let $\a^{-}=(a_1,\ldots,a_{n-1})$ denote $\a$ with the last digit removed. Given $\a,\b\in \cA^{*}$ we let $\a\wedge \b$ denote the maximal common prefix of $\a$ and $\b$. We define $\a\wedge \b$ analogously for $\a,\b\in \cA^{\N}.$

Suppose now that we are given an IFS $\set{f_{a}}_{a\in \cA}$ then given $\a=(a_1,\ldots,a_n)\in \cA^{*}$ we let $$f_{\a}=f_{a_1}\circ \cdots \circ f_{a_n}.$$ Suppose in addition that our IFS consisted of similarities so $f_{a}(x)=r_{a}O_{a}x+t_{a}$ for some $|r_{a}|\in (0,1),$ $O_{a}\in O(d)$ and $t_{a}\in \R^{d},$ then given $\a=(a_1,\ldots,a_n)\in \cA^{*}$ we let $$r_{\a}=\prod_{i=1}^{n}r_{a_{i}} \quad \textrm{ and }\quad O_{\a}=O_{a_1}\cdots O_{a_n}.$$ Similarly, suppose we are given a probability vector $\p=(p_{a})_{a\in \cA}$ then we for any $\a\in \cA^{*}$ we let $$p_{\a}=\prod_{i=1}^{n}p_{a_i}.$$

Suppose now that $\set{f_a}_{a\in \cA}$ is a $C^2$ conformal IFS acting on $[0,1]^{d}$. Then, there exists $C>0$ such that $$\|D_{x}f_{a}-D_{y}f_{a}\|\leq C\|x-y\|$$ for all $x,y\in [0,1]^{d}$ and $a\in \cA$. Moreover, we have the following stronger statement that follows from \cite[Lemma 7.3]{AngelevskaKaenmakiTroscheit}.
\begin{lem}
\label{Lemma:Bounded distortions}
   Let $\set{f_a}_{a\in \cA}$ be a $C^2$ self-conformal IFS acting on $[0,1]^{d}.$ Then there exists $C>0$ such that $$\|D_{x}f_{\a}-D_{y}f_{\a}\|\leq C\sup_{z\in [0,1]^{d}}\|D_{z}f_{\a}\|\|x-y\|$$ for all $x,y\in [0,1]^d$ and $\a\in \cA^{*}$.
\end{lem}Lemma \ref{Lemma:Bounded distortions} will play a crucial role in this paper when we want to linearize a function $f_{\a}.$ We also have the following well known result \cite[Lemma 2.2]{Patzschke}
\begin{lem}
\label{Lemma:Derivative and Diameter}
    Let $\set{f_a}_{a\in \cA}$ be a $C^2$ self-conformal IFS acting on $[0,1]^{d}.$ Then there exists $C>0$ such that for any $\a\in \cA^{*}$ we have $$C^{-1}\mrm{Diam}(f_{\a}([0,1]^{d})\leq \|D_{x}f_{\a}\|\leq C\mrm{Diam}(f_{\a}([0,1]^{d})$$ for all $x\in[0,1]^d.$
\end{lem}

We also recall some useful properties of Gibbs measures. Suppose $\psi:U\to\R$ is a $C^{1}$ potential satisfying $P(\psi)=0$. Under this assumption it can be shown that a Gibbs measure $\mu_{\psi}$ satisfies the following invariance property:
\begin{equation}
\label{Eq:Gibbs measure invariance}
\int g\, d\mu_{\psi}=\int \sum_{a\leadsto x}w_{a}(x)g(f_{a}(x))\, d\mu_{\psi}
\end{equation}for all $g:X_{A}\to \C$ continuous. For a proof of \eqref{Eq:Gibbs measure invariance} in the symbolic setting we refer the reader to \cite{ParryPollicott}. Using the Gibbs property \eqref{eq:GibbsProperty}, it can be shown that Gibbs measures satisfy the following quasi-Bernoulli property. Let $\a,\b\in \cW_{A}$ be such that $\a\b\in \cW_{A}$. Then, there exists $C>1$ such that
\begin{align}
  \label{eq:quasibernoulli}
   C^{-1}\mu_{\psi}(X_\a)\mu_{\psi}(X_\b) \leq \mu_{\psi}(X_{\a\b}) \leq C \mu_{\psi}(X_\a)\mu_{\psi}(X_\b).
\end{align}

\subsection{Preliminaries on Patterson-Sullivan measures}
\label{sec:PS prelims}

We collect here some preliminary facts needed for the proof of Theorem~\ref{thm:PS}.

	\subsubsection{Convex cocompact manifolds}\label{sec:prelimscoco}

	The standard reference for the material in this section is~\cite{Bowditch1993}.
	Let $G $ denote the group of orientation preserving isometries of real hyperbolic space, denoted $\H^{d+1}$, of dimension $d\geq 1$.
    In particular, $G\cong \mrm{SO}(d+1,1)^0$.
	
    Fix a basepoint $o\in \H^{d+1}$. Then, $G$ acts transitively on $\H^{d+1}$ and the stabilizer $K$ of $o$ is a maximal compact subgroup of $G$.
    We shall identify $\H^{d+1}$ with $K\backslash G$.
    Denote by $A=\set{g_t:t\in\R}$ a one parameter subgroup of $G$ inducing the geodesic flow on the unit tangent bundle of $\H^{d+1}$.
    Let $M<K$ denote the centralizer of $A$ inside $K$.

    Let $\G<G$ be an infinite discrete subgroup of $G$.
    The limit set of $\G$, denoted $\L_\G$, is the set of limit points of the orbit $\G\cdot o$ on $\partial \H^{d+1}$.
    Note that the discreteness of $\G$ implies that all such limit points belong to the boundary.
    Moreover, this definition is independent of the choice of $o$ in view of the negative curvature of $\H^{d+1}$.
    We often use $\L$ to denote $\L_\G$ when $\G$ is understood from context.
    We say $\G$ is \textit{non-elementary} if $\L_\G$ is infinite.

    The \textit{non-wandering set} for the geodesic flow, denoted by $\Omega \subseteq G/\G$, is the closure of the set of periodic $A$-orbits.
	We say $\G$ is \textit{convex cocompact} if $\Omega$ is compact, cf.~\cite{Bowditch1993}.
    Denote by $N^+$ (resp. $N^-$) the expanding (resp. contracting) horospherical subgroup of $G$ associated to $g_t$, $t\geq 0$.

   Given $g\in G$, we denote by $g^+$ the coset of $P^-g$ in the quotient $P^-\backslash G$, where $P^-=N^-AM$ is the stable parabolic group associated to $\set{g_t:t\geq 0}$.
   Similarly, $g^-$ denotes the coset $P^+g$ in $P^+\backslash G$.
   Since $M$ is contained in $P^\pm$, such a definition makes sense for vectors in the unit tangent bundle $M\backslash G$.
   Geometrically, for $v\in M\backslash G$, $v^+$ (resp.~$v^-$) is the forward (resp.~backward) endpoint of the geodesic determined by $v$ on the boundary of $\H^{d+1}$.
   Given $x\in G/\G$, we say $x^{\pm}$ belongs to $\L$ if the same holds for any representative of $x$ in $G$; this notion being well-defined since $\L$ is $\G$ invariant.

\subsubsection{Patterson-Sullivan measures}\label{sec:prelimsps}

    The \textit{critical exponent}, denoted $\d_\G$, is defined to be the infimum over all real number $s\geq 0$ such that the Poincar\'e series
    \begin{align}\label{eq:Poincare}
        P_\G(s,o) := \sum_{\g\in\G} e^{-s d(o,\g\cdot o)}
    \end{align}
    converges.
    This exponent coincides with the Hausdorff dimension of the limit set as well as the topological entropy of the geodesic flow on the quotient orbifold $\H^{d+1}/\G$.
    We shall simply write $\d$ for $\d_\G$ when $\G$ is understood from context.
    
    The \textit{Busemann function} is defined as follows: given $x,y\in \H^{d+1}$ and $\xi\in \partial \H^{d+1}$, let $\g:[0,\infty)\to\H^{d+1}$ denote a geodesic ray terminating at $\xi$ and define
    \begin{equation*}
        \beta_\xi(x,y) = \lim_{t\to\infty}
        \dist (x,\g(t)) - \dist(y,\g(t)).
    \end{equation*}
    A $\G$-invariant conformal density of dimension $s$ is a collection of Radon measures $\set{\nu_x}$ on the boundary indexed by $ x\in \H^{d+1}$ which satisfy the following equivariance property:
    \begin{equation*}
        \g_\ast \nu_x = \nu_{\g x}, \qquad \text{and} \qquad
        \frac{d\nu_{y}}{d\nu_x}(\xi) = e^{-s\beta_{\xi}(x,y)}, \qquad
        \forall x,y\in\H^{d+1}, \xi\in \partial \H^{d+1}, \g\in\G.
    \end{equation*}

    Patterson~\cite{Patterson} and Sullivan~\cite{Sullivan} showed the existence of a unique (up to scaling) $\G$-invariant conformal density of dimension $\d_\G$, denoted $\set{\ps_x:x\in \H^{d+1}}$.
    These measures are known as the \textit{Patterson-Sullivan measures} (PS measures for short).
    We refer the reader to~\cite{Roblin} and~\cite{PaulinPollicottSchapira} and references therein for details of the construction in much greater generality.

\subsubsection{Stable and unstable foliations and leafwise measures}

    Recall that we fixed a basepoint $o\in \H^{d+1}$.
    In what follows, we use the following notation for pullbacks of the Patterson-Sullivan measures to orbits of $N^+$ under the visual map: for $x\in G/\G$,
    \begin{equation}\label{eq:unstable conditionals}
        d\mu_x^u(n) = e^{\d_\G \beta_{(nx)^+}(o,nx)}d\ps_o((nx)^+).
    \end{equation}
    These measures have simpler transformation formulas under the action of the geodesic flow and $N^+$ which makes them relatively easier to analyze than the Patterson-Sullivan measures directly.
    In particular, they satisfy the following equivariance property under the geodesic flow:
    \begin{equation}\label{eq:g_t equivariance}
        \mu_{g_tx}^u = e^{\d t} \mrm{Ad}(g_t)_\ast \mu_{x}^u.
    \end{equation}
    Moreover, it follows readily from the definitions that for all $n\in N^+$,
    \begin{align}\label{eq:N equivariance}
       (n)_\ast \mu_{nx}^u =  \mu_x^u,
    \end{align}
    where $(n)_\ast \mu_{nz}^u$ is the pushforward of $\mu_{nz}^u$ under the map $u\mapsto un$ from $N^+$ to itself.
    Finally, since $M$ normalizes $N^+$, these conditionals are $\Ad(M)$-invariant in the sense that for all $m\in M$,
    \begin{align}\label{eq:M equivariance}
        \mu^u_{mx}  = \Ad(m)_\ast\mu_x^u.
    \end{align}

\subsubsection{Norms, metrics, and Lie algebras}

    We denote by $\mf{n}^+$ and $\mf{n}^-$ the Lie algebras of $N^+$ and $N^-$ respectively. 
    We fix an isomorphism of $\mf{n}^+$ and $\mf{n}^-$ using a Cartan involution sending $g_{t}$ to $g_{-t}$. 
    Moreover, we fix an isomorphism of $\mf{n}^+$ (and hence of $\mf{n}^-$) with $\R^d$.
    Finally, we fix a Euclidean inner product on $\R^d\cong\mf{n}^+\cong \mf{n}^-$ denoted with $\langle \cdot,\cdot\rangle$ which is invariant by the Adjoint action of the group $M\cong \mrm{SO}_d(\R)$ and induces the respective metrics on the groups $N^+$ and $N^-$.
    Given $r>0$, we let $N^+_r$ (resp.~$P^-_r$) the neighborhood of identity of radius $r$ inside $N^+$ (resp.~$P^-:=MAN^-$).

\subsubsection{Local stable holonomy}\label{sec:holonomy}

    In this section, we recall the definition of (stable) holonomy maps which are essential for our arguments. We give a simplified discussion of this topic which is sufficient in our homogeneous setting.
    Let $x=u^-y$ for some $y\in \Omega$ and $u^-\in N^-_2$. Since the product map $N^-\times A \times M \times N^+\to G$ is a diffeomorphism near the identity, we can choose the norm on the Lie algebra so that the following holds. We can find maps $p^-:N_1^+\to P^-=N^-AM$ and $u^+:N_2^+\to N^+$ so that
\begin{align}\label{eq:switching order of N- and N+}
    nu^- = p^-(n)u^+(n), \qquad \forall n\in N_2^+.
\end{align}
Then, it follows by~\eqref{eq:unstable conditionals} that for all $n\in N_2^+$, we have
\begin{align*}
    d\mu_y^u(u^+(n)) = e^{\d \beta_{(nx)^+}(u^+(n)y,nx)}d\mu_x^u(n).
\end{align*}
Moreover, by further scaling the metrics if necessary, we can ensure that these maps are diffeomorphisms onto their images.
In particular, writing $\Phi(nx)=u^+(n)y$, we obtain the following change of variables formula: for all $f\in C(N_2^+)$,
\begin{align}\label{eq:stable equivariance}
    \int f(n) \;d\mu_x^u(n) = 
    \int f((u^+)^{-1}(n)) e^{-\d \beta_{\Phi^{-1}(ny)^+}(ny,\Phi^{-1}(ny))}\;d\mu_y^u(n).
\end{align}

\begin{remark}\label{rem:commutation of stable and unstable}
To avoid cluttering the notation with auxiliary constants, we shall assume that the $N^-$ component of $p^-(n)$ belongs to $N_2^-$ for all $n\in N_2^+$ whenever $u^-$ belongs to $N_1^-$.
\end{remark}


\subsection{Notational convention}

Throughout this article, given two quantities $A$ and $B$ we will write $A\ll B$ if there exists a constant $C>0$ such that $A\leq CB.$ When we want to emphasize that $C$ depends on another quantity, say $x$, we write $A\ll_{x} B.$ We also write $A=\cO(B)$ to mean $A\ll B$ and $A=\cO_{x}(B)$ to mean $A\ll_{x}B.$ When $A\ll B$ and $B\ll A$ then we write $A\asymp B.$


\section{Self-similar measures: Proof of Theorems \ref{thm:selfsim} and \ref{thm:selfsim rotations}}\label{sec:selfsim}

\subsection{Sketch of the proof}
Before giving the proof of Theorems \ref{thm:selfsim} and \ref{thm:selfsim rotations}, we provide a sketch of the argument. The ideas underpinning this argument are also used in our later proofs. Suppose that our IFS is acting on $\mathbb{R}$ and consists of two maps $\set{f_{a}(x)=r_{a}x+t_{a},f_{b}(x)=r_{b}x+t_{b}}.$ Moreover suppose that $C,l>0$ are such that 
$$\left|\frac{\log |r_{a}|}{\log |r_{b}|}-\frac{p}{q}\right|\geq \frac{C}{q^{l}}$$
for all $(p,q)\in \mathbb{Z}\times \mathbb{N}$. Let $\xi\in\mathbb{R}$ be given. We define the following cut-off set $$\cA_{\xi}:=\set{\a\in\cA^{*}:\prod_{j=1}^{|\a|}|r_{a_j}|< \frac{(\log |\xi|)^{3l}}{|\xi|}\textrm{ and }\prod_{j=1}^{|\a|-1}|r_{a_j}|\geq \frac{(\log \|\xi\|)^{3l}}{\|\xi\|}}.$$ Appealing to the self-similarity of $\mu$ it can be shown that
\begin{equation}
\label{Eq:sketchproofaveraging}
|\widehat{\mu}(\xi)|\leq \sum_{\a\in \cA_{\xi}}p_{\a}|\widehat{\mu}(r_{\a}\xi)|
\end{equation}
The significance of \eqref{Eq:sketchproofaveraging} is that it bounds $|\widehat{\mu}(\xi)|$ from above by an average of $|\widehat{\mu}(\cdot)|$ evaluated at different frequencies. 

At this point, we can employ our Diophantine assumption to show that $|r_{\a}\xi-r_{\b}\xi|\leq 1$ if and only if $r_{\a}=r_{\b}.$ This implies that the frequencies we are averaging over in \eqref{Eq:sketchproofaveraging} form a large well separated set. It is a consequence of this well-separated property and Theorem~\ref{thm:Banaji and Yu} below that the term $|\widehat{\mu}(r_{\a}\xi)|$ decays polylogarithmically in $|\xi|$. The contribution to $\sum_{\a\in \cA_{\xi}}p_{\a}|\widehat{\mu}(r_{\a}\xi)|$ coming from those remaining terms can be bounded by a probabilistic argument. 

What makes this argument more challenging in the full generality of Theorem \ref{thm:selfsim} is that to meaningfully apply the Diophantine assumption to bound the distance between  $r_{\a}\|\xi\|$ and $r_{\b}\|\xi\|$ for some $\xi\in\mathbb{R}^{d}$, we require the products determining $r_\a$ and $r_\b$ to have the same number of contractions coming from digits belonging to $\cA\setminus \set{a,b}$. This issue is overcome by a suitable conditioning argument.  

\subsection{Proof of Theorem \ref{thm:selfsim}}
	
We fix $\xi$ such that $\|\xi\|\gg 1$. Let $a_{1},a_{2}\in \cA$ be such that $\frac{\log |r_{a_{1}}|}{\log |r_{a_{2}}|}$ is Diophantine. Let $l\geq 2$ and $C>0 $ be the corresponding parameters such that $$\left|\frac{\log |r_{a_{1}}|}{\log |r_{a_{2}}|}-\frac{p}{q}\right|\geq \frac{C}{q^{l}}$$ for all $(p,q)\in \mathbb{Z}\times\mathbb{N}$. We define $$\tilde{\xi}=\frac{(\log \|\xi\|)^{3l}}{\|\xi\|}.$$ Given a word $\a\in \cA^{*}$ and $a\in \cA$ we let $$|\a|_{a}=\# \set{1\leq j\leq |\a|:a_{j}=a}.$$ At this point we fix $\delta,\epsilon>0$ to both be sufficiently small that 
\begin{equation}
	\label{Eq:Decay gap}
	\epsilon-1/2(\#\cA-1)+(1/2+\delta)(\#A-2)<0.
\end{equation} 
which allow us to define
$$n(\xi) := \left\lfloor\left(1-\frac{1}{(-\log \tilde{\xi})^{1/2-\delta}}\right)\frac{-\log \tilde{\xi}}{-\sum_{a\in \cA}p_{a}\log |r_{a}|}\right\rfloor \quad \text{and}\quad E_a(\xi) := \frac{(-\log \tilde{\xi})^{1/2+\delta}}{-\# \cA \log |r_a|}, \quad a \in \cA.$$
We say that a word $\a$ is \textit{good} if $\a\in \cA^{n(\xi)}$ and for all $a \in \cA$ we have the following strong upper and lower bounds for $|\a|_{a}$:
$$p_{a} n(\xi)-E_a(\xi)\leq |\a|_{a} \leq p_{a} n(\xi)+E_a(\xi).$$
Let $G_{\xi}$ be the set of good words.  The following lemma records some useful properties of good words.

\begin{lem}
\label{Lemma:Good words lemma} Let $\a\in G_{\xi}.$ Then 
$$\tilde{\xi} \leq |r_{\a}| \ll\tilde{\xi}e^{2(\log\|\xi\|)^{1/2+\delta}}$$
and there exists $C_1,C_2>0$ such that $$\sum_{\a\in \cA^{n(\xi) }:\a\notin G_{\xi}}p_{\a}\leq C_{1}e^{-C_{2}(\log\|\xi\|)^{2\delta}}$$
\end{lem}
\begin{proof}
The proof is a direct calculation using the definitions.
Let $\a\in G_{\xi}$. By the definition of $G_{\xi}$ we have
\begin{align*}
\log |r_{\a}|
=\sum_{a\in \cA} \log|r_{a}|^{|\a|_{a}}
&\geq \sum_{a\in \cA} (p_{a}n(\xi)+E_a(\xi)) \log |r_{a}|
\\
	&\geq \sum_{a\in \cA}\log |r_{a}|
 \left( p_{a} \left(1-\frac{1}{(-\log \tilde{\xi})^{1/2-\delta}}\right) \frac{-\log \tilde{\xi}}{-\sum_{a\in \cA}p_{a}\log |r_{a}|}+E_a(\xi)\right)
 \\
	&=\log \tilde{\xi} \left(1-\frac{1}{(-\log \tilde{\xi})^{1/2-\delta}}\right) -(-\log \tilde{\xi})^{1/2+\delta}\\
	&=\log \tilde{\xi}.
\end{align*} 
This establishes the first property. For the second property, we similarly have
\begin{align*}|r_{\a}|=\prod_{a\in \cA}|r_{a}|^{|\a|_{a}}&\leq \prod_{a\in \cA}|r_{a}|^{p_{a}n(\xi)-E_a(\xi)}\\
	&\ll \exp \left({\sum_{a\in \cA}\log |r_{a}|\left(p_{a} \left(1-\frac{1}{(-\log \tilde{\xi})^{1/2-\delta}}\right) \frac{-\log \tilde{\xi}}{-\sum_{a\in \cA}p_{a}\log |r_{a}|}-E_a(\xi)\right)}\right)\\
	&=\exp\left({\log \tilde{\xi} \left(1-\frac{1}{(-\log \tilde{\xi})^{1/2-\delta}}\right)+(-\log \tilde{\xi})^{1/2+\delta}}\right)\\
	&=\exp\left({\log \tilde{\xi}+2(-\log \tilde{\xi})^{1/2+\delta}}\right).
\end{align*}
Recalling the definition of $\tilde{\xi}$, we have that $-\log \tilde{\xi}\leq \log \|\xi\|$. Combining this with the above implies $$|r_{\a}|\ll \tilde{\xi}e^{2(\log\|\xi\|)^{1/2+\delta}}.$$
The final statement in this lemma follows from an application of Hoeffding's inequality \cite{Hoe}.

\end{proof}
We consider the following cut off set 
\begin{align}\label{eq:selfsimilar - A_xi}
    \cA_{\xi}:=\set{\a\in \cA^{*}:|r_{\a}|< \tilde{\xi}\textrm{ and }|r_{\a^{-}}|\geq \tilde{\xi}},
\end{align}
and its subset 
\begin{align}\label{eq:self-similar tilde A_xi}
    \tilde{\cA_{\xi}}:=\set{\a\in \cA_{\xi}:a|_{1}^{n(\xi)}\in G_{\xi}}.    
\end{align}
 
\begin{lem}
	\label{Lemma:Good cutoff lemma}
The following properties hold:
\begin{itemize}
	\item If $\a\in \tilde{\cA_{\xi}}$ then for all $a\in \cA$ we have
	$$|\a|_{a}\geq p_{a}n(\xi)-E_a(\xi)$$ and there exists $C>0$ such that for all $a\in \cA$ we have
	$$|\a|_{a}\leq p_{a}n(\xi)+E_a(\xi)+C(\log \|\xi\|)^{1/2+\delta}.$$
	\item There exists $C_1,C_2>0$ such that $$\sum_{\a\in \cA_{\xi}:\a\notin \tilde{\cA_{\xi}}}p_{\a}\leq  C_{1}e^{-C_{2}(\log\|\xi\|)^{2\delta}}.$$
\end{itemize}
\end{lem} 
 \begin{proof}
Let $\a\in \tilde{\cA_{\xi}}.$ Then, the lower bound for $|\a|_{a}$ follows from the corresponding lower bound in the definition of a good word. The upper bound for $|\a|_{a}$ follows from the corresponding upper bound in the definition of a good word together with the fact that if $\a\in\tilde{\cA_{\xi}}$ then \begin{equation}
	\label{Eq:length bound}
	|\a|-|\a|_{1}^{n(\xi)}\ll (\log \|\xi\|)^{1/2+\delta}.
\end{equation} \eqref{Eq:length bound} follows since if $\b\in G_{\xi}$, then for $C$ sufficiently large we have 
$$|r_{\b}|\left(\max_{a\in \cA}\set{|r_{a}|}\right)^{C(\log \|\xi\|)^{1/2+\delta}}\ll \tilde{\xi}e^{2(\log\|\xi\|)^{1/2+\delta}}\left(\max_{a\in \cA}\set{|r_{a}|}\right)^{C(\log \|\xi\|)^{1/2+\delta}}<\tilde{\xi},$$
where in the first inequality we have used the second property in Lemma \ref{Lemma:Good words lemma}.
 
To prove the final part of this lemma we begin by remarking that by Lemma \ref{Lemma:Good words lemma} we know that if $\a\in G_{\xi}$ we have $r_{\a}\geq \tilde{\xi}$. Thus all of the descendants of $\a$ belonging to $\cA_{\xi}$ are elements of $\tilde{\cA_{\xi}}$. Therefore 

$$\sum_{\a\in \cA_{\xi}:\a\notin \tilde{\cA_{\xi}}}p_{\a}=\sum_{\a\in \cA^{n(\xi) }:\a\notin G_{\xi}}p_{\a}.$$ 
Our upper bound for $\sum_{\a\in \cA_{\xi}:\a\notin \tilde{\cA_{\xi}}}p_{\a}$ therefore follows from the corresponding bound in Lemma \ref{Lemma:Good words lemma}.
 \end{proof}

We are now at a point where we want to use our Diophantine assumption. To use this assumption effectively we need to condition on the number of occurrences of a digit $a\in \cA\setminus\set{a_{1},a_{2}}$. With this goal in mind, we let
\begin{align}\label{eq:selfsimilar K def}
  K=\set{(|\a|_{a})_{a\in \cA\setminus\set{a_{1},a_{2}}} :\a\in \tilde{\cA_{\xi}} } \subset \N^{\# \Acal -2}.  
\end{align}
It follows from Lemma \ref{Lemma:Good cutoff lemma} and the definition of $\tilde{\xi}$ that if $\a\in \tilde{\cA_{\xi}},$ then $|\a|_{a}$ belongs to an interval of length at most a constant multiple of $(\log \|\xi\|)^{1/2+\delta}$ for any $a\in \cA$. Therefore we have the important bound
\begin{equation}
	\label{eq:K cardinality}
\#K\ll (\log \|\xi\|)^{(\#A-2)(1/2+\delta)}.
\end{equation}
The following lemma shows that if we condition on an element of $K$ then the contraction ratios will be well separated. This is the only place where we use our Diophantine assumption.

\begin{lem}
	\label{Lemma:Diophantine separation}
	Let $\a,\a'\in \tilde{\cA_{\xi}}$ be such that $|\a|_{a}=|\a'|_{a}$ for all $a\in \cA\setminus \set{a_{1},a_{2}},$ and suppose that either $|\a|_{a_{1}}\neq|\a'|_{a_{1}}$ or $|\a|_{a_{2}}\neq|\a'|_{a_{2}},$ then $$\left||r_{\a}|\|\xi\|-|r_{\a'}|\|\xi\|\right|>1$$ for $\norm{\xi}$ sufficiently large.
\end{lem}
\begin{proof}
For convenience, we let
\begin{align*}
    \varrho = (|\a'|_{a_{1}}-|\a|_{a_{1}})\log |r_{1}|+(|\a'|_{a_{2}}-|\a|_{a_{2}})\log |r_{2}|.
\end{align*}
Then, we have the following 
\begin{equation}
	\label{Eq:First step}
	\left||r_{\a}|\|\xi\|-|r_{\a'}|\|\xi\|\right|=|r_{\a}|\|\xi\|\left|1-\frac{|r_{\a'}|}{|r_{\a}|}\right|=|r_{\a}|\|\xi\|\left|1-e^{\varrho}\right|.
\end{equation} 
It also follows from the definition of $\cA_{\xi}$ and $\tilde{\xi}$ that
\begin{equation}
	\label{Eq:First contraction bound}
	|r_{\a}|\|\xi\|\gg (\log \|\xi\|)^{3l}. 
\end{equation} Using our Diophantine assumption we have 
\begin{align}
	\label{Eq:Diophantine lower bound}
	\left|1-e^{\varrho}\right|
    \geq \left|\varrho\right| 
    & =(|\a'|_{a_{1}}-|\a|_{a_{1}})\log |r_{2}|\left|\frac{\log |r_{1}|}{\log |r_{2}|}-\frac{(|\a'|_{a_{2}}-|\a|_{a_{2}})}{(|\a'|_{a_{1}}-|\a|_{a_{1}})}\right|\nonumber \\
	&\gg \frac{1}{(|\a'|_{a_{1}}-|\a|_{a_{1}})^{l}}\nonumber\\
	&\gg \frac{1}{(\log \|\xi\|)^{l}}.
\end{align} In the final line we used that if $\a\in \cA_{\xi}$, then we must have $|\a|\ll \log\|\xi\|.$ Substituting \eqref{Eq:First contraction bound} and \eqref{Eq:Diophantine lower bound} into \eqref{Eq:First step}, we obtain 
 $$\left||r_{\a}|\|\xi\|-|r_{\a'}|\|\xi\|\right|\gg \frac{(\log \|\xi\|)^{3l}}{(\log \|\xi\|)^{l}}.$$ Our result follows. 
\end{proof}
Before proceeding with the proof of our theorem we need the following upper bound for probabilities coming from a multinomial distribution.

\begin{lem}
	\label{Lemma:Maximising multinomial}
	Let $(p_{a})_{a\in \cA}$ be a probability vector and $n\in \mathbb{N}$, then 
	$$\frac{n!}{\prod_{a\in \cA} k_{a}!}\prod_{a\in \cA}p_{a}^{k_{a}}\ll n^{-1/2(\#\cA-1)}$$ for any $(k_{a})_{a\in \cA}\in \mathbb{N}^{\#\cA}$ satisfying $\sum_{a\in \cA}k_{a}=n.$ Here the underlying constants only depend upon the probability vector.
\end{lem}
The proof of Lemma \ref{Lemma:Maximising multinomial} is straightforward but somewhat lengthy. Thus for the purpose of our exposition, we have deferred this argument to Appendix \ref{appendix:Multinomial}.

We also need the following result concerning polynomial decay on average. It is shown in~\cite{BaYu} that self-similar measures satisfy the non-concentration inequality~\eqref{eq:expanded ball affine non-conc}.
The authors also observe that the proof of Theorem~\ref{thm:flattening} goes through under this weaker hypothesis.
Hence, combined with Theorem~\ref{thm:flattening}, this implies the following higher dimensional generalization of a result of Tsujii \cite{Tsujii-selfsimilar}.

\begin{thm}
\label{thm:Banaji and Yu}
Let $\Phi$ be a self-similar IFS. Then every self-similar measure is polynomially decaying on average if and only if $\Phi$ is affinely irreducible.
\end{thm}

Equipped with Lemma \ref{Lemma:Maximising multinomial} and Theorem \ref{thm:Banaji and Yu} we can now complete our proof of Theorem \ref{thm:selfsim}.

\begin{proof}[Proof of Theorem \ref{thm:selfsim}]
Using the self-similarity of the measure and Lemma \ref{Lemma:Good cutoff lemma} we have
\begin{align*}
	|\widehat{\mu}(\xi)|=\left|\int e^{2\pi i \langle\xi,x\rangle}\, d\mu\right|&=\left|\sum_{\a\in \cA_{\xi}}p_{\a}\int e^{2\pi i\langle\xi, r_{\a}O_{\a}x+t_{a}\rangle}\, d\mu\right|\\
	&\leq \sum_{\a\in \tilde{\cA_{\xi}}}p_{\a}|\widehat{\mu}(r_{\a}O_{\a}^{T}\xi)|+C_{1}e^{-C_{2}(\log\|\xi\|)^{2\delta}}.
\end{align*}
Thus to prove our theorem it suffices to show that 
\begin{equation}
	\label{eq:selfsimilar WTS}
	\sum_{\a\in \tilde{\cA_{\xi}}}p_{\a}|\widehat{\mu}(r_{\a}O_{\a}^{T}\xi)|\ll (\log\|\xi\|)^{-\nu},
\end{equation}
for some $\nu>0$. By Theorem \ref{thm:Banaji and Yu} we know that our self-similar measure is polynomially decaying on average. We use this property for $T=(\log\|\xi\|)^{3l}.$ In particular, for our fixed value of $\epsilon$ satisfying \eqref{Eq:Decay gap}, there exists $\tau>0$ such that the set
$$\set{\zeta\in \mathbb{R}^{d}:\|\zeta\|\leq (\log\|\xi\|)^{3l} \textrm{ and }|\widehat{\mu}(\zeta)|\geq (\log\|\xi\|)^{-\tau}}$$ can be covered by $\cO_{\epsilon}((\log\|\xi\|)^{\epsilon})$ balls of length $1$. Taking the modulus of those frequencies $\zeta$ satisfying $|\widehat{\mu}(\zeta)|\geq (\log\|\xi\|)^{-\tau}$ gives us the following more useful bound. Let 
\begin{align}\label{eq:selfsimilar Bad}
    \mrm{Bad}:=\set{n\in\mathbb{N}:\exists \zeta\in\mathbb{R}^{d}\, s.t.\, \|\zeta\|\leq (\log\|\xi\|)^{3l},\,|\widehat{\mu}(\zeta)|\geq (\log\|\xi\|)^{-\tau}\, \textrm{ and }\|\zeta\|\in [n,n+1]}.
\end{align}
Then, $\#\mrm{Bad}=\cO_{\epsilon}((\log\|\xi\|)^{\epsilon})$. 

To proceed, we use a conditioning argument on elements of the set $K$ introduced in~\eqref{eq:selfsimilar K def}.
To this end, given $\overline{k}\in K$ and $a\in \Acal\setminus \set{a_1,a_2}$, we let $k_a$ denote the component of $\overline{k}$ corresponding to the letter $a$.
For simplicity, let us use $\mrm{Good}$ to denote the set $(\cup_{n\in \mrm{Bad}}[n,n+1])^{c}$.
Hence, using the definition of $\mrm{Bad}$, we obtain
\begin{align*}
	\sum_{\a\in \tilde{\cA_{\xi}}}p_{\a}|\widehat{\mu}(r_{\a}O_{\a}^{T}\xi)|&=\sum_{\overline{k}\in K}\sum_{\stackrel{\a\in \tilde{\cA_{\xi}}}{|\a|_{a}=k_{a} \forall a\in \cA\setminus\set{a_1,a_2}}}	p_{\a}|\widehat{\mu}(r_{\a}O_{\a}^{T}\xi)|\\
	&=\sum_{n\in \mrm{Bad}}\sum_{\overline{k}\in K}\sum_{\stackrel{\a\in \tilde{\cA_{\xi}}}{|\a|_{a}=k_{a} \forall a\in \cA\setminus\set{a_1,a_2}}}	p_{\a}\chi_{[n,n+1]}(|r_{\a}|\|\xi\|)|\widehat{\mu}(r_{\a}O_{\a}^{T}\xi)|\\
	&+\sum_{\overline{k}\in K}\sum_{\stackrel{\a\in \tilde{\cA_{\xi}}}{|\a|_{a}=k_{a} \forall a\in \cA\setminus\set{a_1,a_2}}}	p_{\a}\chi_{\mrm{Good}}(|r_{\a}|\|\xi\|)|\widehat{\mu}(r_{\a}O_{\a}^{T}\xi)|\\
		&\leq \sum_{n\in \mrm{Bad}}\sum_{\overline{k}\in K}\sum_{\stackrel{\a\in \tilde{\cA_{\xi}}}{|\a|_{a}=k_{a} \forall a\in \cA\setminus\set{a_1,a_2}}}	p_{\a}\chi_{[n,n+1]}(|r_{\a}|\|\xi\|)|\widehat{\mu}(r_{\a}O_{\a}^{T}\xi)|\\
		&+\sum_{\overline{k}\in K}\sum_{\stackrel{\a\in \tilde{\cA_{\xi}}}{|\a|_{a}=k_{a} \forall a\in \cA\setminus\set{a_1,a_2}}}	p_{\a}(\log\|\xi\|)^{-\tau}\\
		&\leq \sum_{n\in \mrm{Bad}}\sum_{\overline{k}\in K}\sum_{\stackrel{\a\in \tilde{\cA_{\xi}}}{|\a|_{a}=k_{a} \forall a\in \cA\setminus\set{a_1,a_2}}}	p_{\a}\chi_{[n,n+1]}(|r_{\a}|\|\xi\|)|\widehat{\mu}(r_{\a}O_{\a}^{T}\xi)|+(\log\|\xi\|)^{-\tau}.
\end{align*}
Thus, our proof is complete if we can show that 
\begin{equation}
	\label{Eq:Selfsimilar WTS2}
	\sum_{n\in \mrm{Bad}}\sum_{\overline{k}\in K}\sum_{\stackrel{\a\in \tilde{\cA_{\xi}}}{|\a|_{a}=k_{a} \forall a\in \cA\setminus\set{a_1,a_2}}}	p_{\a}\chi_{[n,n+1]}(|r_{\a}|\|\xi\|)|\widehat{\mu}(r_{\a}O_{\a}^{T}\xi)|\ll (\log(\|\xi\|)^{-\nu},
\end{equation}for some $\nu>0$. 

The crucial step towards verifying \eqref{Eq:Selfsimilar WTS2} is to note that by Lemma \ref{Lemma:Diophantine separation}, for any $n\in \N$ and $\overline{k}\in K,$ if $\a,\b\in \tilde{\cA_{\xi}}$ satisfy $|\a|_{a}=|\b|_{a}$ for all $a\in \cA\setminus\set{a_1,a_2}$ and $|r_{\a}|\|\xi\|,|r_{\b}|\|\xi\|\in [n,n+1],$ then $|\a|_{a_1}=|\b|_{a_1}$ and $|\a|_{a_2}=|\b|_{a_2}$. Assuming such words exist, we denote the unique values for $|\a|_{a_1}$ and $|\a|_{a_{2}}$ by $j_{1}(\overline{k},n)$ and $j_{2}(\overline{k},n)$ respectively. Using this observation, we have
\begin{align*}
		&\sum_{n\in \mrm{Bad}}\sum_{\overline{k}\in K}\sum_{\stackrel{\a\in \tilde{\cA_{\xi}}}{|\a|_{a}=k_{a} \forall a\in \cA\setminus\set{a_1,a_2}}}	p_{\a}\chi_{[n,n+1]}(|r_{\a}|\|\xi\|)|\widehat{\mu}(r_{\a}O_{\a}^{T}\xi)|\\
		\leq &\sum_{n\in \mrm{Bad}}\sum_{\overline{k}\in K}\sum_{\substack{\a\in \tilde{\cA_{\xi}}\\|\a|_{a}=k_{a} \forall a\in \cA\setminus\set{a_1,a_2}\\ |\a|_{a_{1}}=j_{1}(\overline{k},n),\, |\a|_{a_{2}}=j_{2}(\overline{k},n) }}	p_{\a}\\
		=&\sum_{n\in \mrm{Bad}}\sum_{\overline{k}\in K} \frac{\left(\sum_{a\in \cA\setminus\set{a_1,a_2}}k_{a}+j_{1}(\overline{k},n)+j_{2}(\overline{k},n)\right)!}{\prod_{a\in \cA\setminus\set{a_1,a_2}}k_{a}!j_{1}(\overline{k},n)!j_{2}(\overline{k},n)!}p_{a_1}^{j_{1}(\overline{k},n)}p_{a_2}^{j_{2}(\overline{k},n)}\prod_{a\in \cA\setminus\set{a_1,a_2}}p_{a}^{k_{a}}.
\end{align*} Since $\sum_{a\in \cA\setminus\set{a_1,a_2}}k_{a}+j_{1}(\overline{k},n)+j_{2}(\overline{k},n)$ is the length of a word in $\tilde{\cA_{\xi}},$ it follows that $$ \sum_{a\in \cA\setminus\set{a_1,a_2}}k_{a}+j_{1}(\overline{k},n)+j_{2}(\overline{k},n)\asymp \log\|\xi\|.$$ Using this fact together with \eqref{eq:K cardinality}, Lemma \ref{Lemma:Maximising multinomial}, and the fact $\#\mrm{Bad}=\cO_{\epsilon}((\log \|\xi\|)^{\epsilon})$ we have 
\begin{align*}
	&\sum_{n\in \mrm{Bad}}\sum_{\overline{k}\in K}\sum_{\stackrel{\a\in \tilde{\cA_{\xi}}}{|\a|_{a}=k_{a} \forall a\in \cA\setminus\set{a_1,a_2}}}	p_{\a}\chi_{[n,n+1]}(r_{a}\|\xi\|)|\widehat{\mu}(r_{\a}O_{\a}^{T}\xi)|\\
	\ll& \sum_{n\in \mrm{Bad}}\sum_{\overline{k}\in K}(\log\|\xi\|)^{-1/2(\#\cA-1)}\\
	\ll& \sum_{n\in \mrm{Bad}}(\log\|\xi\|)^{-1/2(\#\cA-1)+(1/2+\delta)(\#A-2)}\\
	\ll& (\log\|\xi\|)^{\epsilon-1/2(\#\cA-1)+(1/2+\delta)(\#A-2)}
\end{align*} Recall now that by \eqref{Eq:Decay gap} we fixed $\epsilon,\delta>0$ to be sufficiently small so that the exponent appearing in the last line in the above is negative. Thus \eqref{Eq:Selfsimilar WTS2} holds and our proof is complete.
\end{proof}

\subsection{Proof of Theorem \ref{thm:selfsim rotations}}
The proof of Theorem \ref{thm:selfsim rotations} is almost identical to the proof of Theorem \ref{thm:selfsim} so we only indicate the small change required.

Let $a_{1},a_{2}\in \cA$ be such that $(\theta_{a_1},\theta_{a_2})$ is Diophantine. We then fix $\xi\in \mathbb{R}^{d}$. We define $\tilde{\xi}$, $G_{\xi}$, $\cA_{\xi}$ and $\tilde{\cA_{\xi}}$ as above. The proof of Theorem \ref{thm:selfsim rotations} proceeds in an identical way to the proof of Theorem \ref{thm:selfsim} until we get to Lemma \ref{Lemma:Diophantine separation}. We have the following analogue of this lemma.

\begin{lem}
	\label{Lemma:Rotation separation}
Let $\set{f_{a}(x)=r_{a}O_{a}x+t_{a}}_{a\in \cA}$ be an IFS acting on $\mathbb{R}^{2}$ satisfying the assumptions of Theorem \ref{thm:selfsim rotations}. Let $\a,\a'\in \tilde{\cA_{\xi}}$ be such that $|\a|_{a}=|\a'|_{a}$ for all $a\in \cA\setminus \set{a_{1},a_{2}}$ and suppose that either $|\a|_{a_{1}}\neq|\a'|_{a_{1}}$ or $|\a|_{a_{2}}\neq|\a'|_{a_{2}}$ then $$\|r_{\a}O_{\a}^{T}\xi-r_{\a'}O_{\a'}^{T}\xi\|>1$$ for $\|\xi\|$ sufficiently large.
\end{lem}
\begin{proof}
	Let $\a,\a'$ satisfy the assumptions of the lemma. Iterating our self-similar IFS if necessary, we can assume without loss of generality that $r_{a_1}>0$ and $r_{a_{2}}>0$. Using the fact that elements of $SO(2)$ commute and are distance preserving we have 
$$\left\|O_{\a}^{T}\frac{\xi}{\|\xi\|}-O_{\a'}^{T}\frac{\xi}{\|\xi\|}\right\|=\left\|(O_{a_1}^{T})^{|\a|_{a_1}}(O_{a_2}^{T})^{|\a|_{a_2}}e-(O_{a_1}^{T})^{|\a'|_{a_1}}(O_{a_2}^{T})^{|\a'|_{a_2}}e\right\|$$ where $e=(1,0)$. Focusing on this latter expression, it follows from our Diophantine assumption and the fact $|\a|,|\a'|\ll \log\|\xi\|$ that
\begin{align*}
	\left\|(O_{a_1}^{T})^{|\a|_{a_1}}(O_{a_2}^{T})^{|\a|_{a_2}}e-(O_{a_1}^{T})^{|\a'|_{a_1}}(O_{a_2}^{T})^{|\a'|_{a_2}}e\right\|&\gg d((|\a|_{a_1}-|\a'|_{a_1})\theta_{1}+(|\a|_{a_2}-|\a'|_{a_2})\theta_{2},\mathbb{Z})\\
	&\gg \frac{1}{(\log \|\xi\|)^{l}}.
	\end{align*} Combining the bound above with the bounds $|r_{\a}|\|\xi\|\gg  (\log\|\xi\|)^{3l}$ and  $|r_{\a'}|\|\xi\|\gg  (\log\|\xi\|)^{3l}$ we have
\begin{align*}
	\|r_{\a}O_{\a}^{T}\xi-r_{\a'}O_{\a'}^{T}\xi\|&=\left\|r_{\a}\|\xi\|O_{\a}^{T}\frac{\xi}{\|\xi\|}-r_{\a'}\|\xi\|O_{\a'}^{T}\frac{\xi}{\|\xi\|}\right\|\\
	&\geq \|\xi\|\prod_{a\in \cA\setminus\set{a_1,a_2}}|r_{a}|^{|\a|_{a}}\left\|r_{a_1}^{|\a|_{a_{1}}}r_{a_2}^{|\a|_{a_{2}}}O_{\a}^{T}\frac{\xi}{\|\xi\|}-r_{a_1}^{|\a'|_{a_{1}}}r_{a_2}^{|\a'|_{a_{2}}}O_{\a'}^{T}\frac{\xi}{\|\xi\|}\right\|\\
	&\geq \min\set{|r_{\a}|\|\xi\|,|r_{\a'}|\|\xi\|}\left\|\left(O_{\a}^{T}\frac{\xi}{\|\xi\|}-O_{\a'}^{T}\frac{\xi}{\|\xi\|}\right)\right\|\\
	&\gg\frac{(\log\|\xi\|)^{3l}}{(\log \|\xi\|)^{l}}.
\end{align*}In the penultimate line we have used the fact that if $u,v\in\mathbb{R}^{2}\setminus\set{0}$ have unit length and $r_{1},r_{2}\geq 0$ satisfy $r_{1}\leq r_{2}$ then $\|r_{1}u-r_{1}v\|\leq \|r_{1}u-r_{2}v\|.$ The final bound obtained above certainly exceeds $1$ for $\|\xi\|$ sufficiently large. Thus our result follows.
\end{proof}
The Diophantine assumptions on our IFS imply that it is affinely irreducible and Theorem \ref{thm:Banaji and Yu} can be applied. What remains of our proof is analogous to the proof of Theorem \ref{thm:selfsim}, the only difference being that the role played by Lemma \ref{Lemma:Diophantine separation} is now played by Lemma \ref{Lemma:Rotation separation}.

\section{Patterson-Sullivan measures: Proof of Theorem \ref{thm:PS}}\label{sec:proofps}

The goal of this section is to provide the proof of Theorem~\ref{thm:PS}.
	Throughout this section, we fix a discrete, Zariski-dense, convex cocompact group $\G$ of isometries of $\H^{d+1}$, $d\geq 1$.

\subsection{Reduction to linear phases}
Since PS measures are absolutely continuous with respect to the conditional measures $\mu_x^u$ with smooth Radon-Nikodym derivatives, cf.~\eqref{eq:unstable conditionals}, it suffices to prove the conclusion of Theorem~\ref{thm:PS} for $\mu_x^u$ for some $x\in G/\G$ in place of the PS measure $\mu$.
In fact, we prove a stronger statement which establishes uniform bounds for the measures $\mu_x^u$ as $x$ varies in the non-wandering set $\Omega$.

We begin with the following elementary lemma which reduces the proof to the study of linear phase functions. 

\begin{lem}\label{lem:reduce to linear}
To prove Theorem~\ref{thm:PS}, it suffices to show that there exists $\kappa>0$ so that for all $0\neq \xi\in  \R^{d}$, $x\in \Omega$, and $\psi\in C_c^1(N_1^+)$, we have
\begin{align}\label{eq:linear decay}
    \int_{N_1^+} e^{ i \langle \xi, n\rangle} \psi(n)\;d\mu^u_x(n)
    \ll_\G \norm{\psi}_{C^1} \norm{\xi}^{-\kappa},
\end{align}
where, by abuse of notation, if $n=\exp(v)$ for some $v\in \mf{n}^+\cong \R^d$, we let $\langle \xi,n\rangle := \langle \xi, v\rangle$.
\end{lem}

\begin{remark}
    In the remainder of this section, we fix $\xi\in \R^d$ and $\psi\in C_c^1(N_1^+)$. Our goal is to prove the estimate~\eqref{eq:linear decay}.

\end{remark}

\begin{proof}[Proof of Lemma~\ref{lem:reduce to linear}]
    The proof is based on the uniformity of the estimate~\eqref{eq:linear decay} as the basepoint $x$ varies in $\Omega$ which roughly translates to uniform Fourier decay (with linear phases) over pieces of the measure of size $|\l|^{-1/2-\e}$.
    We include a sketch of the argument for completeness.
    
    Recall the notation of Theorem~\ref{thm:PS}.
    Let $\set{\rho_j}$ be a partition of unity of $\mrm{supp}(\mu_x^u\left|_{N^+}\right.)$ with bounded multiplicity and such that each $\rho_j$ is supported in a ball $B_j$ of radius
    \begin{align*}
        r = |\l|^{-(1+\kappa)/(2+\kappa)}
    \end{align*}
    around a $n_j\in N_1^+$.
    Here, $\kappa$ is the exponent in~\eqref{eq:linear decay}.
    In view of~\cite[Prop.~9.9]{Khalil-Mixing}, we can choose such partition of unity so that each $\rho_j$ has first derivatives with norm $\cO(r^{-1})$ and
    \begin{align}\label{eq:bdd multiplicity}
        \sum_j \mu_x^u(B_j) \ll \mu_x^u(N_1^+).
    \end{align}

    Then, Taylor expanding $\vp$ to the second order around each $n_j$, we obtain
    \begin{align*}
        \int_{N_1^+} e^{ i \l \vp(n) } \psi(n)\;d\mu_x^u
        &\leq \sum_j \left|\int_{B_j} \exp(i \langle \l \nabla \vp(n_j), nn_j^{-1} \rangle)
        (\psi \rho_j)(n)\;d\mu_x^u \right|
        \nonumber\\
        &+ \cO(\norm{\psi}_{C^0} \norm{\vp}_{C^2} |\l| r^2 \mu^u_x(N_1^+)).
    \end{align*}
    Next, we use a change of variables sending $B_j$ to $N_1^+$ and apply~\eqref{eq:g_t equivariance} and~\eqref{eq:N equivariance}.
    More precisely, let $t = -\log r, x_j=g_t n_j x$, $\xi_j=r\l \nabla\vp(n_j)$, and $\psi_j(n):= (\psi\rho_j)(\Ad(g_{-t})(n)n_j)$.
    Then, the $j^{th}$ term in the above sum can be rewritten as follows:
    \begin{align*}
        \int_{B_j} \exp(i \langle \l \nabla \vp(n_j), nn_j^{-1} \rangle)
        (\psi \rho_j)(n)\;d\mu_x^u
        = r^\d \int_{N_1^+} \exp(i \langle \xi_j, n \rangle) 
        \psi_j(n) \;d\mu^u_{x_j}.
    \end{align*}
    Note that, since the geodesic flow scales the first derivatives of $\rho_j$ by a factor of $e^{-t}$, each $\psi_j$ has $C^1$-norm $\cO(\norm{\psi}_{C^1})$.
    Hence, since each $x_j$ belongs to $\Omega$,~\eqref{eq:linear decay} implies that
    \begin{align*}
        \int_{N_1^+} e^{ i \l \vp(n) } \psi(n)\;d\mu_x^u
        \ll_{A}  |r\l|^{-\kappa} a^{-\kappa}  \sum_j r^\d + |\l| r^2.
    \end{align*}
    
    Finally, in view of Sullivan's shadow lemma (cf.~\cite[Theorem 3.4]{Khalil-Mixing}), we have that $\mu_x^{u}(B_j) \asymp r^\d$.
    This concludes the proof in light of~\eqref{eq:bdd multiplicity}.
    The decay exponent obtained in this manner is $\kappa/(2+\kappa)$, with $\kappa$ as in~\eqref{eq:linear decay}.
\end{proof}

\subsection{Polynomial non-concentration estimates}
We recall here well-known non-concentration estimates for PS measures.
The first estimate is a direct consequence of Sullivan's shadow lemma.

\begin{prop}[{Sullivan's Shadow Lemma, cf.\cite{Sullivan}}]
\label{prop:shadow lem}
    For all $y\in \Omega$ and all $r>0$, we have
    \begin{align*}
        \mu^u_y(N_r^+) \asymp_\G r^\d,
    \end{align*}
    where $\d$ is the critical exponent of $\G$.
\end{prop}

We also recall the following quantitative decay property of the measure of hyperplane neighborhoods with respect to PS measures from~\cite{Dasetal}.
Recall that $N^+$ is an abelian group which we identify with its Lie algebra $\mf{n}^+\cong \R^d$ via the exponential map.
In light of this identification, the following result shows that PS measures (or, more precisely, their shadows $\mu_x^u$) are non-concentrated in the sense of Definition~\ref{def:non-conc}.
This will allow us to apply Theorem~\ref{thm:flattening} in the proof of Theorem~\ref{thm:PS} as well as provide quantitative estimates on separation of frequencies arising over the course of implementing the strategy discussed in the introduction.

\begin{thm}[{\cite[Theorem 12.1]{Khalil-Mixing}}]
\label{thm:friendly}
    There exist constants $C\geq 1$ and $\alpha>0$ such that for all $\e,r>0$, $x\in\Omega$, and all affine hyperplanes $L<N^+$, we have that
    \begin{align*}
        \mu_x^u(N_r^+\cap L^{(\e r)}) \leq C\e^\alpha \mu_x^u(N_r^+),
    \end{align*}
    where $L^{(\e r)}$ denotes the $\e r$-neighborhood of $L$ in $N^+$.
\end{thm}

\begin{remark}
    The reference \cite[Theorem 12.1]{Khalil-Mixing} shows that, in the more general setting of geometrically finite groups, one has $\mu_x^u(N_r^+\cap L^{(\e r)}) \leq t(\e) V(x) \mu_x^u(N_r^+)$, for a function $t(\e) \to 0$ as $\e \to 0$ and for a function $V(x)$ that is uniformly bounded above and below on compact sets.
    The proof is much simpler in the case $\G$ is convex cocompact and in fact yields the apriori stronger bound $t(\e)\leq C\e^\alpha$.
    In fact, when $\G$ is convex cocompact, Theorem~\ref{thm:friendly} can be deduced directly from the fact that PS measures give $0$ mass to proper subvarieties of the boundary (\cite[Corollary 9.4]{EdwardsLeeOh-Torus}) using the argument in~\cite[Section 8]{KleinbockLindenstraussWeiss}.

\end{remark}

\subsection{Proof of Theorem~\ref{thm:PS}}
The remainder of the section is dedicated to the proof of the estimate~\eqref{eq:linear decay}.
We fix $\xi\in\R^d$, $x\in\Omega$, and $\psi\in C_c^1(N_1^+)$.
Our argument is dynamical in nature using the self-similar structure of the measures $\mu_x^u$ under the geodesic flow, cf.~\eqref{eq:g_t equivariance} and~\eqref{eq:N equivariance}, to implement the strategy described in the introduction.

\subsection*{Partitions of unity and flow prisms}

As a first step, we find convenient partitions of the space by flow boxes.
Namely, we refer to sets of the form $P^-_r N^+_s\cdot x$ for $r,s>0$ and $x\in G/\G$ as \textit{flow boxes}.
We say that a collection of sets $\set{S_i}$ has multiplicity bounded by a constant $C\geq 1$ if for all $x$:
        $\sum_{i} \mathbbm{1}_{S_i}(x) \leq C \mathbbm{1}_{\cup_i S_i}(x).$ Let $\iota$ denote the smaller of $1/2$ and the injectivity radius of $G/\G$ and set
\begin{align}\label{eq:iota_xi}
    \iota_\xi := \iota/\norm{\xi}^{1/3}.
\end{align}
The following lemma provides an efficient cover of $\Omega$ by ``thin flow boxes" in the unstable direction.
\begin{lem}\label{lem:count flow boxes}
    The collection $\set{P^-_\iota N^+_{\iota_\xi}\cdot x:x\in \Omega}$ admits a finite subcover $\Bcal_\xi$ such that
      $  \# \Bcal_\xi \ll_\G \norm{\xi}^{\d/3},$
    where $\d$ is the critical exponent of $\G$.
    Moreover, $\Bcal_\xi$ has uniformly bounded multiplicity on $\Omega$; i.e.~for all $x\in \Omega$, $\sum_{B\in\Bcal_\xi} \mathbbm{1}_B(x)\ll_\G 1$.
\end{lem}

\begin{proof}
    
    Let $\Qcal$ denote a cover of $G/\G$ by flow boxes of the form $P^-_\iota N^+_\iota \cdot x$, where $\iota$ is a fixed lower bound on the injectivity radius of $G/\G$ as above.
    With the help of the Vitali covering lemma, such cover can be chosen to have multiplicity $C_G\geq 1$ depending only on the dimension of $G$.
    We will build our collection of boxes $\Bcal_\xi$ by refining this cover as follows.
    
    Let $\Qcal^0$ denote the collection of boxes $Q\in\Qcal$ such that $Q\cap  \Omega\neq \emptyset$.
    By convex cocompactness, we have that $\# \Qcal^0 \asymp_\G 1 $.
    For each $Q\in \Qcal^0$, we fix some $x_Q\in Q\cap \Omega$.
    Then, we can find a finite set of points $\set{u_i:i\in I_Q}\subset N^+_{2\iota}$ such that the points $x_i:=u_i x_Q$ belong to $\Omega$.
    Moreover, these points can be chosen so that the balls $N^+_{\iota_\xi}\cdot x_i$ provide a cover of $\Omega\cap N^+_\iota \cdot x_Q $
    with uniformly bounded multiplicity thanks to the Vitali covering lemma applied to $N^+$.
    
    With this notation, we define $\Bcal_\xi$ as follows:
    \begin{align*}
        \Bcal_\xi :=\set{ P^-_\iota N^+_{\iota_\xi} \cdot u_i x_Q: i\in I_Q, Q \in \Qcal^0 }.
    \end{align*}
    This bounded multiplicity in particular implies that
    \begin{align*}
        \iota_\xi^\d \times \#  I_Q \asymp \sum_{i\in I_Q} \mu^u_{x_i}(N^+_{\iota_\xi}) \asymp \mu^u_{x_Q}(N^+_\iota) \asymp 1.
    \end{align*}
    This estimate implies the desired the bound on the cardinality of $\Bcal_\xi$ in light of~\eqref{eq:iota_xi}.
    To bound the multiplicity of $\Bcal_\xi$, let $x\in \Omega$ be arbitrary, and note that for $A_Q:=\cup_{i\in I_Q} P^-_\iota N^+_{\iota_\xi} \cdot u_i x_Q$, we have
    \begin{align*}
        \sum_{B\in \Bcal_\xi} \mathbbm{1}_B(x) =
        \sum_{Q\in \Qcal^0} \sum_{i\in I_Q} \mathbbm{1}_{P^-_\iota N^+_{\iota_\xi} \cdot u_i x_Q} (x)
        \ll \sum_{Q\in \Qcal^0} \mathbbm{1}_{A_Q} (x)
        \ll \# \Qcal^0 \ll_\G 1.
    \end{align*}
    \qedhere
\end{proof}

Let $\Bcal_\xi$ be the finite cover provided by Lemma~\ref{lem:count flow boxes} and let $\Pcal_\xi$ denote a partition of unity 
subordinate to it. 
For each $\rho\in \Pcal_\xi$, we denote by $B_\rho $ the element of $\Bcal_\xi$ containing the support of $\rho$.
In particular, such partition of unity can be chosen so that for all $\rho\in \Pcal_\xi$, we have
\begin{equation}\label{eq:norm of rho}
    \norm{\rho}_{C^1} \ll \norm{\xi}^{1/3}.
\end{equation}
Moreover, by Lemma~\ref{lem:count flow boxes}, we have 
\begin{align}\label{eq:count Pcal_xi}
    \# \Pcal_\xi \leq \# \Bcal_\xi \ll_\G \norm{\xi}^{\d/3}.
\end{align}

 \subsection*{Transversals}
We fix a system of transversals $\set{T_\rho}$ to the strong unstable foliation inside the boxes $B_\rho$.
Since $B_\rho$ meets $\Omega$ for all $\rho\in \Pcal_\xi$, we fix some $y_\rho$ in the intersection $B_\rho\cap \Omega$. 
In this notation, we write 
\begin{equation}\label{eq:box notation}
    B_\rho =  P^-_\iota N^+_{\iota_\xi}\cdot y_\rho, \qquad T_\rho = P^-_\iota \cdot y_\rho.
\end{equation}
We also let $M_\rho, A_\rho$, and $N^-_\rho$ be neighborhoods of identity in $M, A$ and $N^-$ respectively so that $P^-_\iota=M_\rho A_\rho N^-_\rho$.

\subsection*{Saturation}
Fix $t>0$ to be chosen so that $e^t$ is a small positive power of $\norm{\xi}$; cf.~\eqref{eq:eta and t}.
Using our partition of unity, we can write 
\begin{align}\label{eq:initial partition}
    \int_{N_1^+} e^{ i \langle \xi, n\rangle} \psi(n)\;d\mu^u_x(n)
    =\sum_{\rho\in \Pcal_\xi} \int_{N_1^+}e^{i\langle \xi,n\rangle} \psi(n) \rho(g_{t} nx) \;d\mu_x^u(n).
\end{align}
Here, we are using the fact that, since $x\in \Omega$, then the restriction of the support of $\mu_{x}^u$ to $N_1^+$ consists of points $n\in N_1^+$ with $nx \in \Omega$ (or equivalently, that $g_tnx\in \Omega$) and that $\sum_\rho \rho(y)=1$ for all $y\in\Omega$.

Our first step is to partition the integrals on the right side of~\eqref{eq:initial partition} over $N_1^+$ into pieces according to the flow box they land in under flowing by $g_t$.
To simplify notation, we write
\begin{equation}\label{eq:y_t}
      x_t:= g_t x.
\end{equation}
We denote by $N_1^+(t)$ a neighborhood of $N_1^+$ defined by the property that the intersection 
$$ B_\rho\cap (\mrm{Ad}(g_t)(N_1^+(t))\cdot x_t)$$
consists entirely of full local strong unstable leaves in $B_\rho$.
We note that since $\mrm{Ad}(g_t)$ expands $N^+$ and $B_\rho$ has radius $<1$, $N_1^+(t)$ is contained inside $N_2^+$.
Since $\psi$ is compactly supported inside $N_1^+$, we have
\begin{equation}\label{eq:N_1^+(t)}
    \chi_{N_1^+}(n)\psi(n)  = \chi_{N_1^+(t)}(n)\psi(n), \qquad \forall n \in N^+.
\end{equation}

For simplicity, we set
\begin{equation*}
    \xi_t := e^{-t}\xi, \qquad 
    \psi_t(n) := \psi(\Ad(g_t)^{-1} n ), \qquad \Acal_t :=  \mrm{Ad}(g_t)(N_1^+(t)).
\end{equation*}
For $\rho\in \Pcal_\xi$, we let $\Wcal_{\rho,t}$ denote the collection of connected components of the set
\begin{equation*}
    \set{n\in\Acal_t: nx_t\in B_\rho}.
\end{equation*}
In view of~\eqref{eq:N_1^+(t)}, changing variables using~\eqref{eq:g_t equivariance} yields
\begin{align}\label{eq:localize space}
  \sum_{\rho\in \Pcal_\xi} \int_{N_1^+}e^{i\langle \xi,n\rangle} \psi(n) \rho(g_t nx) \;d\mu_x^u
  =e^{-\d t } 
    \sum_{\rho\in \Pcal_\xi, W\in \Wcal_{\rho,t}}
    \int_{n\in W}  e^{i\langle \xi_t,n\rangle} \psi_t(n) 
        \rho(nx_t)  \; d\mu_{x_t}^u.
\end{align}

\subsection*{Centering the integrals}
It will be convenient to center all the integrals in~\eqref{eq:localize space} so that their basepoints belong to the transversals $T_\rho$ of the respective flow box $B_\rho$; cf.~\eqref{eq:box notation}.

Let $I_{\rho,t}$ denote an index set for $\Wcal_{\rho,t}$.
For $W\in \Wcal_{\rho,t}$ with index $\ell\in I_{\rho,t}$, let $n_{\rho,\ell}\in W$, $m_{\rho,\ell}\in M_\rho$, $n_{\rho,\ell}^-\in N^-_\rho$, and $t_{\rho,\ell}$ with $|t_{\rho,\ell}|\ll \iota$ be such that
\begin{equation}\label{eq:centers}
    x_{\rho,\ell} := n_{\rho,\ell}\cdot x_t = n_{\rho,\ell}^- m_{\rho,\ell}g_{t_{\rho,\ell}}  \cdot y_\rho \in T_\rho.
\end{equation}
Note that since $x$ and $y_\rho$ both belong to $\Omega$, we have that
\begin{equation}\label{eq:x_rho,ell in omega}
    x_{\rho,\ell}\in \Omega, \qquad n_{\rho,\ell}^- y_\rho \in \Omega.
\end{equation}
For each such $\ell$ and $W$, let us denote $W_\ell=Wn_{\rho,\ell}^{-1}$ and set
\begin{equation}\label{eq:phi tilde}
    \tilde{\chi}_{\rho,\ell}(t,n) := 
          \exp(i \langle\xi_t, nn_{\rho,\ell} \rangle).
\end{equation} 
Changing variables using~\eqref{eq:g_t equivariance} and~\eqref{eq:N equivariance}, we can rewrite the right side of~\eqref{eq:localize space} as follows: 
\begin{align}\label{eq:center integrals on transversal}
         e^{-\d t} 
         \sum_{\rho\in \Pcal_\xi, W\in \Wcal_{\rho,t}}
          &\int_{n\in W}  e^{i\langle \xi_t,n\rangle}
          \psi_t( n  )
         \rho(nx_t)  \;d\mu_{x_t}^u(n) 
         \nonumber\\
         &=e^{-\d t}
         \sum_{\rho\in \Pcal_\xi} 
         \sum_{\ell\in I_{\rho,t}} 
           \int_{n\in W_\ell}  
           \widetilde{\chi}_{\rho,\ell}(t,n)
          \psi_t( n n_{\rho,\ell}) \rho(nx_{\rho,\ell})
            \;d\mu_{x_{\rho,\ell}}^u(n) .
\end{align}  

\subsection*{Mass estimates}
We record here certain counting estimates which will allow us to sum error terms in later estimates over $\Pcal_\xi$.
Note that by definition of $N_1^+(j)$, we have $\bigcup_{\rho\in\Pcal_\xi,W\in \Wcal_{\rho,t}}W\subseteq \Acal_t$. 
Thus, it follows that
\begin{align}\label{eq:total mass}
    \sum_{\rho\in\Pcal_\xi,\ell\in I_{\rho,t}}
    \mu_{ x_{\rho,\ell}}^u(W_\ell)
     \ll \mu_{x_t}^u(\Acal_t)
    =
    e^{\d t} \mu_{ x}^u(N_1^+(t))
    \ll  e^{\d t} \mu_{ x}^u(N_2^+) \ll e^{\d t} \mu_x^u(N^+_1), 
\end{align}
where the last inequality follows since $N_1^+(j)\subseteq N_2^+$ using the doubling property of PS measures~\cite[Proposition 3.1]{Khalil-Mixing}.
We also used the fact that the partition of unity $\Pcal_\xi$ has uniformly bounded multiplicity.

\subsection*{Weak-stable holonomy} Fix some $\rho\in \Pcal_\xi$.
Recall the points $y_{\rho}\in T_\rho$ and $n^-_{\rho,\ell}\in N^-_\rho$ satisfying~\eqref{eq:centers}.
Let
\begin{align}\label{eq:p_rho,ell}
    p^-_{\rho,\ell} := 
    n_{\rho,\ell}^- m_{\rho,\ell}g_{t_{\rho,\ell}} .
\end{align}
The product map $ N^-\times A \times M\times N^+\to G$ is a diffeomorphism on a ball of radius $1$ around the identity; cf.~Section~\ref{sec:holonomy}. 
Hence, given $\ell\in I_{\rho,t}$, we can define maps $\phi_\ell$ and $\tilde{p}^-_\ell$ from $ W_\ell$ to $N^+$ and $P^-$ respectively by the following formula
\begin{equation}\label{eq:stable hol map}
    n p_{\rho,\ell}^- 
    =\tilde{p}^-_{\ell}(n)\phi_\ell(n).
\end{equation}
We suppress the dependence on $\rho$ and $t$ to ease notation.
Then, $\phi_\ell$ induces a map between the strong unstable manifolds of $x_{\rho,\ell}$ and $y_\rho$, also denoted $\phi_\ell$, and defined by
\begin{equation*}
    \phi_\ell( nx_{\rho,\ell}) = \phi_\ell(n)y_\rho.
\end{equation*}
In particular, this induced map coincides with the local weak stable holonomy map inside $B_\rho$.

Note that we can find a neighborhood $W_\rho\subset N^+$ of identity of radius $\asymp \iota_\xi$ such that
\begin{equation}\label{eq:W_rho}
    \phi_\ell(W_\ell) \subseteq W_\rho,
\end{equation}
for all $\ell\in I_{\rho,t}$.
Moreover, by shrinking the radius $\iota_\xi$ of the flow boxes by an absolute amount (depending only on the metric on $G$) if necessary, we may assume that all the maps $\phi_\ell$ are invertible on $ W_{\rho}$.
Hence, we can define the following:
\begin{align}\label{eq:def of tau_ell}
     p^-_\ell(n) &:= \tilde{p}_\ell^-(\phi_\ell^{-1}(n)) \in P^-,
     \qquad
     \widetilde{\psi}_{\rho,\ell}(t,n) := J\phi_\ell(n)\times  
     \psi_t( \phi_\ell^{-1}(n) n_{\rho,\ell}),
    \nonumber\\
    \chi_{\rho,\ell}(t,n) &:= \widetilde{\chi}_{\rho,\ell}(t,\phi_\ell^{-1}(n)), \qquad
    \rho_\ell(n) := \rho(p^-_\ell(n)ny_\rho),
\end{align}
where $J\phi_\ell$ denotes the Jacobian of the change of variable $\phi_\ell$; cf.~\eqref{eq:stable equivariance}.

Changing variables in the right side of~\eqref{eq:center integrals on transversal}, we obtain
\begin{align}\label{eq:stable hol}
    \sum_{\ell\in I_{\rho,t}} 
           \int_{n\in W_\ell}  
           \widetilde{\chi}_{\rho,\ell}(t,n)
          \widetilde{\psi}_{\rho,\ell}(t,n) \rho(nx_{\rho,\ell})
            \;d\mu_{x_{\rho,\ell}}^u 
          =  \sum_{\ell\in I_{\rho,t}} \int_{W_\rho}
         \chi_{\rho,\ell}(t,n)
         \widetilde{\psi}_{\rho,\ell}(t,n)
    \rho_\ell(n)    \;d\mu_{y_{\rho}}^u.
\end{align}

\subsection*{Phase formula}
The following lemma provides a formula for the stable holonomy maps $\phi_\ell$ defined above \eqref{eq:W_rho} which are responsible for the oscillation of $\chi_{\rho,\ell}$ along $N^+$.
The elementary proof of this lemma is given in Section~\ref{sec:temporal function}.

\begin{lem}
    \label{lem:phi_ell formula}
    Let $p^-_{\rho,\ell}$ be as in~\eqref{eq:p_rho,ell} and let $w_{\rho,\ell}\in \mf{n}^-$ be such that $n^-_{\rho,\ell}=\exp(w_{\rho,\ell})$. 
    Define vectors $z_{\rho,\ell}\in \mf{n}^-$ by
    \begin{align}\label{eq:z_rho,ell}
        z_{\rho,\ell} := -e^{t_{\rho,\ell}} m_{\rho,\ell}^{-1} \cdot w_{\rho,\ell}.
    \end{align}
    Then, for every $n=\exp(v)\in N_{1/2}^+$, we have
    \begin{align*}
        \log \phi^{-1}_\ell(n) = e^{t_{\rho,\ell} - \tilde{\t}_{\ell}(v)}
        m_{\rho,\ell} \cdot \left(v + \frac{\norm{v}^2}{2} z_{\rho,\ell}\right),
    \end{align*}
    where $\log \phi^{-1}_\ell(n)$ is viewed as an element of $\mf{n}^+$ and $\tilde{\t}_{\ell}:N_{1/2}^+\to \R_+$ is given by
    \begin{align*}
        \tilde{\t}_{\ell}(v) =\log\left( 1+\langle v, z_{\rho,\ell}\rangle + \frac{\norm{v}^2 \norm{z_{\rho,\ell}}^2}{4}\right).
    \end{align*}
    \end{lem}

It will be convenient for our estimates to simplify the expression for $\tilde{\t}_\ell$ by removing the quadratic term.
This is the reason for our choice of flow boxes of width $\asymp\norm{\xi}^{-1/3}$ along the strong unstable manifold.
In what follows, to simplify notation, we set
\begin{align}\label{eq:tau_ell and Gamma_ell}
        \t_\ell(v) 
        =-t_{\rho,\ell}+\log\left( 1+\langle v, z_{\rho,\ell}\rangle \right),
        \qquad
        \G_\ell(v) := e^{-\t_\ell(v)} m_{\rho,\ell} \cdot \left(v + \frac{\norm{v}^2}{2} z_{\rho,\ell}\right).
\end{align}
Recall the centering points $n_{\rho,\ell}$ defined in~\eqref{eq:centers}.
The following corollary provides a first step towards linearizing the phase in the oscillatory functions $\chi_{\rho,\ell}$ by replacing $\tilde{\t}_\ell$ in Lemma~\ref{lem:phi_ell formula} with $\t_\ell$ in~\eqref{eq:tau_ell and Gamma_ell}.
\begin{cor}
\label{cor:linearize phase}
    
    With the same notation as in Lemma~\ref{lem:phi_ell formula}, we have for all $n=\exp(v)\in W_\rho$ that
    \begin{align*}
        \chi_{\rho,\ell}(t,n) = 
        \a_{\rho,\ell}(t,n) + \cO(e^{-t}),
    \end{align*}
    where
    \begin{align}\label{eq:alpha_rho,ell}
        \a_{\rho,\ell}(t,n):= \exp(i \langle\xi_{t}, n_{\rho,\ell}  \rangle)
        \times 
         \exp\left (i \langle \xi_{t}, \G_\ell(v)\rangle \right).
    \end{align}
\end{cor}
\begin{proof}
    
    Recall from~\eqref{eq:def of tau_ell} and~\eqref{eq:phi tilde} that $\chi_{\rho,\ell}(t,n) = \exp(i\langle \xi_t,\phi_\ell^{-1}(n)n_{\rho,\ell}\rangle)$.
    We also recall that $\xi_t = e^{-t}\xi$.
    Then, for all $n=\exp(v)\in W_\rho$, we have that
    \begin{align*}
        |\chi_{\rho,\ell}(t,n)
        - \a_{\rho,\ell}(t,n)|
        \ll \norm{v}^2 \norm{\xi_t} \norm{\G_\ell(v)}
    \end{align*}
    Since $W_\rho$ has radius $\asymp \iota_\xi \asymp \norm{\xi}^{-1/3}$ (cf.~\eqref{eq:iota_xi}), we have that both $v$ and $\Gamma_\ell(v)$ have norm $\ll \norm{\xi}^{-1/3}$.
    In particular, the upper bound above is $\cO(e^{-t})$ as desired.
    \qedhere
\end{proof}

Let us summarize our progress so far.
To simplify notation, set
\begin{align}\label{eq:psi_rho,ell}
    \psi_{\rho,\ell}(t,n) :=
    \widetilde{\psi}_{\rho,\ell}(t ,n)
    \times \rho_\ell(n).
\end{align}
Then, in light of~\eqref{eq:initial partition},~\eqref{eq:localize space},~\eqref{eq:center integrals on transversal},~\eqref{eq:stable hol}, and Corollary~\ref{cor:linearize phase}, we find that
\begin{align}\label{eq:pre-Cauchy-Shwarz}
    \int_{N_1^+} e^{ i \langle \xi, n\rangle} \psi(n)\;d\mu^u_x
    = e^{-\d t}  \sum_{\rho\in \Pcal_\xi}
    \int_{W_\rho}
    \sum_{\ell\in I_{\rho,t}} 
         \a_{\rho,\ell}(t ,n)
         \psi_{\rho,\ell}(t,n)
     \;d\mu_{y_{\rho}}^u 
    + \cO\left(e^{-t}\right).
\end{align}

\subsection*{Cauchy-Schwarz}
\label{sec:CauchySchwarz}
We are left with estimating integrals of the form:
\begin{align}\label{eq:Psi_rho}
       \int_{ W_\rho} \Psi_{\rho}(t,n)
      \;d\mu^u_{y_\rho} ,
    \qquad
    \Psi_{\rho}(t,n):=\sum_{\ell\in I_{\rho,t}} 
    \a_{\rho,\ell}(t,n)
    \psi_{\rho,\ell}(t,n) . 
\end{align}
 By Cauchy-Schwarz, we get
\begin{align}\label{eq:CauchySchwarz}
    \left|\int_{W_\rho}\Psi_\rho(t,n)\;d\mu^u_{y_\rho}\right|^2
    \leq  \mu^u_{y_\rho}(W_\rho) \int_{W_\rho}\left|\Psi_\rho(t,n)\right|^2\;d\mu^u_{y_\rho}
\end{align}

We begin by noting the following apriori bounds on $\Psi_\rho$:
 \begin{align}\label{eq:bound psi without tilde}
     \norm{\psi_{\rho,\ell}}_{L^\infty( W_\rho)}
     \ll  1, \qquad \norm{\Psi_\rho}_{L^\infty(W_\rho)} 
     \ll \# I_{\rho,t}.
 \end{align}

\subsection*{Partitioning the support}

Using~\cite[Proposition 9.9]{Khalil-Mixing}, we can find a cover $\set{A_j}$ of $W_\rho$ with balls of radius
\begin{align}\label{eq:r}
    r= \norm{\xi}^{-1/2}
\end{align}
centered around $u_j\in W_\rho\cap \supp(\mu^u_{y_\rho})$ and satisfying $\sum_j\mu^u_{y_\rho}(A_j)\ll \mu_{y_\rho}^u(W_\rho)$.
By the triangle inequality\footnote{Cauchy-Schwarz allows us to have a non-negative integrand which in turn enables this step.} we have
\begin{align}\label{eq:partition}
\int_{ W_\rho}\left|\Psi_\rho(t,n)\right|^2\;d\mu^u_{y_\rho} 
\leq \sum_j \int_{A_j}\left|\Psi_\rho(t,n)\right|^2\;d\mu^u_{y_\rho}.
\end{align}

For $k,\ell\in I_{\rho,t}$, we let
\begin{equation*}
    \psi_{k,\ell}(t,n) := \psi_{\rho,k}(t,n)
    \overline{\psi_{\rho,\ell}(t,n) }, 
    \qquad
    \a_{k,\ell}(t,n) := \a_{\rho,k}(t,n) \overline{\a_{\rho,\ell}(t,n)}.
\end{equation*}
Expanding the square, we get
\begin{align*}
    \sum_j \int_{ A_j} |\Psi_\rho(t,n)|^2\;d\mu_{y_\rho}^u
    \leq 
     \sum_j \sum_{k,\ell\in I_{\rho,t}}
     \left|\int_{A_j} \a_{k,\ell}(t,n) \psi_{k,\ell}(t,n)  \;d\mu^u_{y_{\rho}}\right| .
\end{align*}
Using~\eqref{eq:g_t equivariance} and~\eqref{eq:N equivariance}, we change variables in the integrals using the maps taking each $A_j$ onto $N_{1}^+$.
More precisely, recall that $A_j$ is a ball of radius $r$ around $u_j$ such that $u_j \yrho \in \Omega$.
Letting  
\begin{align}\label{eq:after moving to cpt}
    \t = -\log r, \qquad
    \yrho^j &= g_{\t}u_j \yrho, 
    \qquad
    \a^j_{k,\ell}(t,n) = \a_{k,\ell}(t,\Ad(g_{-\t})(n)u_j), 
    \nonumber\\
    \psi^j_{k,\ell}(t,n) &= \psi_{k,\ell} (t,\Ad(g_{-\t})(n)u_j),
\end{align}
we can rewrite the above sum as
\begin{align}\label{eq:from A_i to N_1}
    \sum_j \sum_{k,\ell\in I_{\rho,t}} \left|
     \int_{A_j} \a_{k,\ell}(t,n) \psi_{k,\ell}(t,n)  \;d\mu^u_{y_{\rho}}\right|
     \leq r^\d \sum_{j} 
    \sum_{k,\ell\in I_{\rho,t}}
       \left| \int_{N_{1}^+} 
       \a^j_{k,\ell}(t,n) \psi^j_{k,\ell}(t,n) 
       d\mu^u_{y^j_\rho} 
       \right|.
\end{align}

One advantage of flowing forward by $g_\t$ is that it provides smoothing of the amplitude functions $\psi_{k,\ell}$.
In particular, it follows by~\eqref{eq:norm of rho} that
\begin{align*}
    \norm{\psi_{k,\ell}^j}_{C^1}  \ll \norm{\psi}_{C^1} \times r\times \norm{\xi}^{1/3} \ll \norm{\psi}_{C^1} \norm{\xi}^{-1/6}.
\end{align*}
Applied to the right side of~\eqref{eq:from A_i to N_1}, we obtain
\begin{align}\label{eq:remove amp}
    \int_{W_\rho}\left|\Psi_\rho(t,n)\right|^2\;d\mu^u_{y_\rho}
    =r^\d \sum_j \sum_{k,\ell\in I_{\rho,t}}
       \left| \int_{N_{1}^+} 
       \a^j_{k,\ell}(t,n)
       d\mu^u_{y^j_\rho} 
       \right|
       + \cO(\norm{\psi}_{C^1} \norm{\xi}^{-1/6} \# I_{\rho,t}^2 \mu^u_{y_\rho}(W_\rho)).
\end{align}

\subsection*{Linearizing the phase}

We now turn to estimating the sum of oscillatory integrals in~\eqref{eq:remove amp}.
Recall that $u_j$ denotes the center of the ball $A_j$ for each $j$
and let $v_j\in \mf{n}^+$ be such that
\begin{align*}
    u_j = \exp(v_j).
\end{align*}
Then, given $n=\exp(v)\in A_j$, and recalling the maps $\G_\ell$ in~\eqref{eq:tau_ell and Gamma_ell}, we get
\begin{align*}
    \G_\ell(v) = \G_\ell(v_j) 
    + D(\G_\ell(v_j))(v-v_j) + \cO(r^2),
\end{align*}
where $D(\G_\ell)$ denotes the derivative of $\G_\ell$.

The following elementary lemma uses the explicit expression for $\G_\ell$ in~\eqref{eq:tau_ell and Gamma_ell} to simplify the form of $D\G_\ell(v_j)$.

\begin{lem}
    \label{lem:simplify Gamma_ell}
    For all $\ell$ and $j$, we have
    \begin{align*}
        D\G_\ell(v_j) = e^{-\t_\ell(v_j)} m_{\rho,\ell} + \cO(\norm{\xi}^{-2/3}).
    \end{align*}
\end{lem}

Let
\begin{align}\label{eq:beta_k,ell}
    \b^j_{k,\ell} := r  \xi_t \cdot \left(e^{-\t_k(v_j)}m_{\rho,k}- e^{-\t_\ell(v_j)}m_{\rho,\ell} \right).
\end{align}
Recall that $\xi_t = e^{-t}\xi$ so that $\norm{\xi_t}r^2 = e^{-t}$.
Hence, by absorbing the constant terms into the absolute value, we obtain from~\eqref{eq:remove amp} and Lemma~\ref{lem:simplify Gamma_ell} that
\begin{align}\label{eq:linearize phase}
     &\int_{W_\rho}\left|\Psi_\rho(t,n)\right|^2 \;d\mu^u_{y_\rho}
    \nonumber\\
    &=r^\d \sum_j \sum_{k,\ell\in I_{\rho,t}}
       \left| \int_{N_{1}^+} 
       \exp(i\langle \b^j_{k,\ell}, v\rangle)
       d\mu^u_{y^j_\rho} 
       \right|
       +  \cO((e^{-t}
         + e^{-t}\norm{\xi}^{-1/6}
       +\norm{\psi}_{C^1} \norm{\xi}^{-1/6} )\# I_{\rho,t}^2 \mu^u_{y_\rho}(W_\rho)).
\end{align}

\begin{proof}
    [Proof of Lemma~\ref{lem:simplify Gamma_ell}]
    Recall the definition of the vectors $z_{\rho,\ell}\in \mf{n}^-$.
    To simplify notation, set
    \begin{align*}
        \l_\ell(v_j) = \frac{1}{1+\langle v_j,z_{\rho,\ell}\rangle}.
    \end{align*}
    In particular, $e^{-\tau_\ell(v_j)}= e^{t_{\rho,\ell}}\l_\ell(v_j)$.
    Then, using the formula for $\G_\ell$ in~\eqref{eq:tau_ell and Gamma_ell}, we obtain
    \begin{align*}
        D\G_\ell(v_j) 
        =e^{t_{\rho,\ell}}\l_\ell(v_j) m_{\rho,\ell} \left[ -\l_\ell(v_j)
        \left(
        v_j \cdot  z_{\rho,\ell}^t 
        + \frac{\norm{v_j}^2}{2} z_{\rho,\ell}\cdot  z_{\rho,\ell}^t
        \right)
        + \mrm{Id} + z_{\rho,\ell} \cdot v_j^t
        \right] .
    \end{align*}
Here, we are viewing $v_j$ and $z_{\rho,\ell}$ as $(d\times 1)$-column vectors and use $v_j^t$ and $z_{\rho,\ell}^t$ to denote the transpose of $v_j$ and $z_{\rho,\ell}$ respectively.
Now, observe that
\begin{align*}
    \l_\ell(v_j) v_j = v_j -\frac{\langle v_j,z_{\rho,\ell}\rangle}{1+\langle v_j,z_{\rho,\ell}\rangle} v_j = v_j + \cO(\norm{v_j}^2).
\end{align*}
The lemma now follows upon recalling that $\exp(v_j)$ belongs to $W_\rho$ so that $\norm{v_j} \ll \norm{\xi}^{-1/3}$ in view of our choice of flow boxes; cf.~\eqref{eq:iota_xi} and the discussion around it.
\end{proof}

\subsection*{Separation of frequencies}
To apply Theorem \ref{thm:flattening}, it will be important to understand the distribution of the frequencies $\b^j_{k,\ell}$.
To this end, we have the following lemma which allows us to avoid studying the separation of the holonomy matrices $m_{\rho,\ell}$.
\begin{lem}\label{lem:tau and freq}
   For all $j,k,\ell$, we have
   \begin{align*}
       \norm{\b^j_{k,\ell}} \gg \norm{r\xi_{t}} | e^{-\t_\ell(v_j)} - e^{-\t_k(v_j)}|,
   \end{align*}
   where $\t_\ell(v_j)$ and $\t_k(v_j)$ are defined in~\eqref{eq:tau_ell and Gamma_ell}.
\end{lem}
\begin{proof}
    In what follows, to simplify notation, we let
    \begin{align*}
        m_k:= m_{\rho,k},\qquad  c_k:=e^{-\t_k(v_j)},
        \qquad Q_k:= c_k m_{k},
    \end{align*}
    with the similar notation for the index $\ell$ in place of $k$ defined analogously.
    The lemma is evident when $c_k=c_\ell$.
    Hence, we may assume without loss of generality that $c_k>c_\ell$, and recall that these functions are non-negative by definition; cf.~\eqref{eq:alpha_rho,ell}.

    Recall the elementary estimate $\norm{g\cdot v} \geq \norm{v}/\norm{g^{-1}}$ for any invertible linear map $g$ and any vector $v\in \R^d$. This estimate implies the following lower bound:
    \begin{align*}
        \norm{\b^j_{k,\ell}} \geq 
        \frac{r \norm{\xi_{t}}}{\norm{(Q_k-Q_\ell)^{-1}}} = \frac{r\norm{\xi_{t}} c_k}{ \norm{(\id - \frac{c_\ell}{c_k}m_{\ell} m_{k}^{-1})^{-1}}}.
    \end{align*}
    
    That $\id - \frac{c_\ell}{c_k}m_\ell m_k^{-1}$ (and hence $Q_k-Q_\ell$) is invertible follows at once from the following estimate on the norm of its inverse.
    Recall that the rotation matrices $m_k$ and $m_\ell$ have spectral radius $1$. In particular, since $c_\ell<c_k$, we may use the power series expansion of $\id - Q$, for matrices $Q$ with spectral radius $<1$, to get that
    \begin{align*}
        \norm{(\id - \frac{c_\ell}{c_k}m_\ell m_k^{-1})^{-1}}
        \ll \sum_{n\geq 0} \left(\frac{c_\ell}{c_k}\right)^n 
        = \frac{c_k}{c_k-c_\ell}.
    \end{align*}
    The lemma follows by combining the above two estimates.\qedhere
\end{proof}

To proceed, we recall that $t_{\rho,\ell}, m_{\rho,\ell},$ and $n^-_{\rho,\ell}=\exp(w_{\rho,\ell})$ parametrize respectively the geodesic flow, $M$, and strong stable coordinates of the transverse intersections of the expanded horospherical disk $g_t N_1^+x$ with a fixed transversal $T_\rho$ of the flow box $B_\rho$.
We also recall that $z_{\rho,\ell} = e^{t_{\rho,\ell}}m_{\rho,\ell}^{-1} w_{\rho,\ell}$ and $\t_\ell(v_j) = e^{t_{\rho,\ell}}/(1+\langle v_j,z_{\rho,\ell}\rangle$.

Lemma~\ref{lem:tau and freq} motivates the definition of the following subset of $I_{\rho,t}^2$ parametrizing pairs $(k,\ell)$ for which the frequencies $\b^j_{k,\ell}$ are too small. 
Namely, we set
\begin{align}\label{eq:close freqs}
    \mrm{Small}
    := \set{(k,\ell,j): 
    \norm{\b^j_{k,\ell}} < 1 }.
\end{align}
Roughly speaking, elements of $C_{k,\ell}$ correspond to points $v_j$ lying in a small neighborhood of a hyperplane orthogonal to $m_{\rho,\ell}^{-1}w_{\rho,\ell} - m_{\rho,k}^{-1}w_{\rho,k}$.
Theorem~\ref{thm:friendly} will then provide us with a counting estimate on $C_{k,\ell}$.
This is done in the following lemma.
\begin{lem}
\label{lem:count close freqs}
    Let $\alpha>0$ be the exponent provided by Theorem~\ref{thm:friendly}.
    Then, for every fixed $k,\ell\in I_{\rho,t}$, we have
    \begin{align*}
        \sum_{j:(k,\ell,j)\in \mrm{Small}} r^\d \ll 
        \left(
        \frac{\norm{\xi}^{-1/6}e^t}{\norm{m_{\rho,\ell}^{-1}w_{\rho,\ell} - m_{\rho,k}^{-1}w_{\rho,k} }}
        \right)^\alpha
        \mu^u_{y_\rho}(W_\rho).
    \end{align*}
\end{lem}
\begin{proof}
Let $j$ be such that $(k,\ell,j)\in\mrm{Small}$ and recall that $\xi_t=e^{-t}\xi$ and $r=\norm{\xi}^{-1/2}$.
To simplify notation, we also let
\begin{align*}
    u_{k,\ell}:= m_{\rho,\ell}^{-1}w_{\rho,\ell} - m_{\rho,k}^{-1}w_{\rho,k}.
\end{align*}
Then, Lemma~\ref{lem:tau and freq} and a direct calculation show that
\begin{align*}
    \left| e^{t_{\rho,\ell}} - e^{t_{\rho,k}} + e^{t_{\rho,\ell}+t_{\rho,k}}\langle v_j,  u_{k,\ell}\rangle \right| \ll \norm{\xi}^{ -1/2} e^t.
\end{align*}
Let $\epsilon_1 = \norm{\xi}^{-1/2} e^t/\norm{u_{k,\ell}}$.
It follows that $v_j$ belongs to a neighborhood of radius $\cO(\epsilon_1)$ around an affine hyperplane $L$ parallel to the kernel of the linear functional $v\mapsto \langle v,u_{k,\ell}\rangle$.

Recall that $A_j$ denotes the ball of radius $r$ around $\exp(v_j)\in W_\rho$ and that $W_\rho$ has radius $ \asymp \norm{\xi}^{-1/3}$.
It follows we can find a radius $\epsilon_2 \asymp \norm{\xi}^{-1/3}$ such that
\begin{align*}
    \bigcup_{j:(k,\ell,j)\in \mrm{Small}} A_j \subseteq L^{(\epsilon_1 + r)} \cap N^+_{\epsilon_2},
\end{align*}
where $L^{(\epsilon_1 +r)}$ denotes the $(\epsilon_1+r)$-neighborhood of $L$.
Furthermore, by the bounded multiplicity of the cover $\set{A_j}$ of $W_\rho$ and the fact that each $A_j$ has measure $\asymp r^\d$ (cf.~Proposition~\ref{prop:shadow lem}), we get that
\begin{align*}
    \sum_{j:(k,\ell,j)\in \mrm{Small}} r^\d \ll 
    \mu^u_{y_\rho}\left(\bigcup_{j:(k,\ell,j)\in \mrm{Small}} A_j \right).
\end{align*}
Hence, Theorem~\ref{thm:friendly} implies that the above sum is $\cO(\mu^u_{y_\rho}(W_\rho)(\epsilon_1+r)^\alpha/\epsilon_2^\alpha)$, which concludes the proof since $r\ll \epsilon_1$.
\end{proof}

To apply Lemma~\ref{lem:count close freqs}, we need the following counting estimate on close by vectors of the form $m_{\rho,\ell}^{-1}w_{\rho,\ell}$.
It is a consequence of Theorem~\ref{thm:friendly}.
\begin{lem}\label{lem:stable separation}
    For every $k\in I_{\rho,t}$ and $\eta>0$, we have
    \begin{align}\label{eq:sphere around k}
        \#\set{\ell\in I_{\rho,t}: \norm{m_{\rho,\ell}^{-1}w_{\rho,\ell} - m_{\rho,k}^{-1}w_{\rho,k}} <\norm{\xi}^{-\eta}}
        \ll \left(e^{-t}+ \norm{\xi}^{-\eta}\right)^\alpha e^{\d t},
    \end{align}
    where $\alpha>0$ is the exponent provided by Theorem~\ref{thm:friendly}.
\end{lem}
\begin{proof}
    Let $B_k$ denote the set on the left side of~\eqref{eq:sphere around k} and let $\ell$ be some element of $B_k$.
    Then, by $M$-invariance of the norm, we have\footnote{This estimate is again done to bypass studying the separation of the rotation matrices $w_{\rho,\ell}$.}
    \begin{align*}
        |\norm{w_{\rho,\ell}} - \norm{w_{\rho,k}}| \ll \norm{\xi}^{-\eta}.
    \end{align*}
    In particular, the vectors $w_{\rho,\ell}$ with $\ell\in B_k$ all belong to a neighborhood of width $\ll \norm{\xi}^{-\eta}$ of the sphere $S$ of radius $\norm{w_{\rho,k}}_\infty$ around the origin in the norm metric.

    The next ingredient is to note that the points $w_{\rho,\ell}$ are separated by an amount $\gg e^{-t}$.
    This follows by a similar argument to the proof of~\cite[Proposition 9.13]{Khalil-Mixing}\footnote{This proof is based on injectivity radius considerations along with the fact that $g_t$ expands the stable manifold by $e^t$ in backward time.}.
    In particular, there is $\epsilon_1 \asymp (\norm{\xi}^{-\eta} + e^{-t})$ and $\epsilon_2 \asymp e^{-t}$ such that
    \begin{align*}
        \bigsqcup_{\ell \in B_k} N^-_{\epsilon_2}\cdot \exp(w_{\rho,\ell}) y_\rho  \subseteq N^-_{\epsilon_1}\cdot S,
    \end{align*}
    where $N^-_{\epsilon_1}\cdot S$ is the $\epsilon_1$-neighborhood of $S$.
    
    To conclude the proof, let $\mu^s_{y_\rho}$ denote the shadow of the PS measure on $N^-\cdot y_\rho$ defined analogously to the measures $\mu^u_{y_\rho}$ in~\eqref{eq:unstable conditionals}.
    The above discussion implies that
    \begin{align*}
        \# B_k
        \ll \frac{ \mu^s_{y_\rho}\left(N^-_{\epsilon_1}\cdot S \right)  }{ \min \mu^s_{y_\rho}(N^-_{\epsilon_2}\cdot \exp(w_{\rho,\ell}) y_\rho )}.
    \end{align*}
    By~\eqref{eq:x_rho,ell in omega}, the points $\exp(w_{\rho,\ell})\cdot y_\rho$ all belong to $ \Omega $.
    In particular, by Proposition~\ref{prop:shadow lem}, we have
    \begin{align*}
        \mu^s_{y_\rho}\left(N^-_{e^{-t}} \cdot \exp(w_{\rho,\ell}) y_\rho \right) \asymp e^{-\d t}.
    \end{align*}
    On the other hand, by~\cite[Lemma 3.8]{Dasetal}, we have that $N^-_{\epsilon_1}\cdot S$ has measure $\cO(\epsilon_1^\alpha)$\footnote{Note that, similarly to the case of affine subspaces in Theorem~\ref{thm:friendly}, this estimate can be deduced from the fact that PS measures give $0$ mass to proper subvarieties of the boundary using the argument in~\cite[Section 8]{KleinbockLindenstraussWeiss}.}.
    The lemma now follows.
\end{proof}

\subsection*{Reduction to $L^2$-flattening}
To simplify our error terms, we make the following choices:
\begin{align}\label{eq:eta and t}
    \eta = 1/12, \qquad e^t = \norm{\xi}^{1/24}.
\end{align}
In view of Lemmas~\ref{lem:count close freqs} and~\ref{lem:stable separation}, we introduce the following notation:
\begin{align}\label{eq:C_rho,t}
    \mrm{Close}_\eta := \set{(k,\ell)\in I_{\rho,t}^2: \norm{m_{\rho,\ell}^{-1}w_{\rho,\ell} - m_{\rho,k}^{-1}w_{\rho,k}} <\norm{\xi}^{-\eta}} .
\end{align}
We also define the following set of indices parametrizing measures $\mu^u_{y_\rho^j}$ for which many of the frequencies $\b^j_{k,\ell}$ are close together. 
Let $\Jcal$ denote the index set for the indices $j$ of the measures $\mu^u_{y^j_\rho}$ and set
\begin{align}
    \mrm{Bad}_\eta := \set{j\in \Jcal: \#\set{(k,\ell) \in I_{\rho,t}^2: (k,\ell,j)\in \mrm{Small} }> \norm{\xi}^{-\alpha/48} \times  (e^{\d t}+\#I_{\rho,t})  \# I_{\rho,t} } .
\end{align}
The following corollary allows us to estimate estimate the part of the sum corresponding to $\mrm{Bad}_\eta$.
\begin{cor}\label{cor:Markov}
We have the following counting estimate on $\mrm{Bad}_\eta$:
   \begin{align*}
       \sum_{j\in \mrm{Bad}_\eta} r^\d
       \ll \norm{\xi}^{-\alpha/48}  \mu^u_{y_\rho}(W_\rho).
   \end{align*}
\end{cor}

\begin{proof}

The corollary will follow from an application of Markov's inequality to the estimates in Lemmas~\ref{lem:count close freqs} and~\ref{lem:stable separation} as follows.
First, we note that
    \begin{align*}
        &\sum_{j\in \Jcal} \sum_{(k,\ell)\in I_{\rho,t}^2}
        r^\d \mathbbm{1}_{\mrm{Small}}(k,\ell,j)
        \nonumber\\
        &= \underbrace{\sum_{(k,\ell)\in \mrm{Close}_\eta} 
        \sum_{j\in \Jcal} r^\d \mathbbm{1}_{\mrm{Small}}(k,\ell,j)}_{(\mrm{I})}
        +\underbrace{ \sum_{(k,\ell)\notin \mrm{Close}_\eta} 
        \sum_{j\in \Jcal} r^\d \mathbbm{1}_{\mrm{Small}}(k,\ell,j)}_{(\mrm{II})}.
    \end{align*}
    Then, by Lemma~\ref{lem:stable separation} and our choices in~\eqref{eq:eta and t}, the first sum is estimated as follows:
    \begin{align*}
        (\mrm{I}) \ll \left(e^{-t} + \norm{\xi}^{-\eta}\right)^\alpha  \times e^{\d t}  \# I_{\rho,t} \mu^u_{y_\rho}(W_\rho)
        \ll \norm{\xi}^{-\alpha/24} \times e^{\d t}  \# I_{\rho,t} \mu^u_{y_\rho}(W_\rho).
    \end{align*}
    For the second sum, we use Lemma~\ref{lem:count close freqs} and the definition of $\mrm{Close}_\eta$ to get
    \begin{align*}
        (\mrm{II}) \ll \norm{\xi}^{(\eta -1/6)\alpha } e^{\alpha t} \times \# I_{\rho,t}^2 \mu^u_{y_\rho}(W_\rho) 
        \ll \norm{\xi}^{-\alpha/24} \times  \# I_{\rho,t}^2 \mu^u_{y_\rho}(W_\rho).
    \end{align*}
    Hence, the corollary follows by Markov's inequality.
\end{proof}

To simplify notation, we set
\begin{align}
    E_1:= \max\set{\norm{\xi}^{-1/24}, \norm{\xi}^{-\alpha/48}}.
\end{align}
For $w\in \R^d$, we let
\begin{align}\label{eq:nu_j}
    \nu_j :=  \mu^u_{y_\rho^j}\left|_{N_1^+} \right. 
    , \qquad
    \widehat{\nu}_j(w):= 
    \int_{N^+} e^{i \langle w, n\rangle} \;d\nu_j(n).
\end{align}
Then, by~\eqref{eq:remove amp} and Corollary~\ref{cor:Markov}, we obtain
\begin{align}\label{eq:before flattening}
    \int_{W_\rho} |\Psi_\rho(t,n)|^2\;d\mu^u_{y_\rho}
     = r^\d \sum_{j\notin \mrm{Bad}_\eta} \sum_{k,\ell\in I_{\rho,t}} 
    |\widehat{\nu}_j(\b^j_{k,\ell})|
     + O \left( ( (\norm{\psi}_{C^1}+1) \times E_1 \times \# I_{\rho,t}^2 
     \mu^u_{y_\rho}(W_\rho) \right) .
\end{align}

For each $j$, the sum on the right side of the above estimate can be viewed as an average, when properly normalized, over Fourier coefficients of the measure $\nu_j$.
Moreover, Corollary~\ref{cor:Markov} guarantees that the frequencies $\b^j_{k,\ell}$ are sampled from a well-separated set.
Hence, this average can be estimated using the $L^2$-Flattening Theorem, Theorem~\ref{thm:flattening}.

\subsection*{The role of $L^2$-flattening}

Let $\eta_2>0$ be a small parameter to be chosen using Lemma~\ref{lem:apply flattening} below.
Note that the total mass of $\nu_j$, denoted $|\nu_j|$, is $\murhoj(N_1^+)$.
For each $k\in I_{\rho,t}$, define the following set, which roughly speaking, consists of frequencies where $\widehat{\nu}_j$ is large:
\begin{align}
    B(j,k,\eta_2):= \set{\ell\in I_{\rho,t}:  |\widehat{\nu}_j(\b_{k,\ell}^j)| > \norm{\xi}^{-\eta_2} |\nu_j|}.
\end{align}
Then, splitting the sum over frequencies according to the size of the Fourier transform $\widehat{\nu}_j$ and reversing our change variables to go back to integrating over $A_j$, we obtain
\begin{align}\label{eq:bound by count}
    r^\d
    \sum_{j\notin \mrm{Bad}_\eta} \sum_{k,\ell\in I_{\rho,t}}
  |\widehat{\nu}_j(\b^j_{k,\ell})|
    \ll
     \left( 
     \max_{j\notin \mrm{Bad}_\eta, k\in I_{\rho,t} } \# B(j,  k,\eta_2) 
    + \norm{\xi}^{-\eta_2} \#I_{\rho,t} 
    \right)  \# I_{\rho,t}\mu^u_{y_\rho}(W_\rho),
 \end{align}

The following key counting estimate for $B(j,k,\eta_2)$ is a consequence of the $L^2$-flattening theorem, Theorem~\ref{thm:flattening}.
\begin{lem}
\label{lem:apply flattening}
    For every $\e>0$, there is $\eta_2>0$ such that for all $j\notin \mrm{Bad}_\eta$ and $k\in I_{\rho,t}$, we have
    \begin{align*}
        \# B(j,k,\eta_2) \ll_\e \norm{\xi}^{\e -\alpha/96}  \times \sqrt{ (e^{\d t} + \# I_{\rho,t})\#I_{\rho,t}},
    \end{align*}
    where $\alpha>0$ is the exponent provided by Theorem~\ref{thm:friendly}.
    Here, $\eta_2$ is the constant provided by Theorem~\ref{thm:flattening} (denoted by $\d$ in the notation of the theorem).
\end{lem}

\begin{proof}
    
    Recall the definition of the frequencies $\b^j_{k,\ell}$ in~\eqref{eq:beta_k,ell}.
    The rough idea behind the proof is that. since $j\notin \mrm{Bad}_\eta$, the frequencies $\b^j_{k,\ell}$ are well-separated. 
    This allows us to apply Theorem~\ref{thm:flattening} to the set of frequencies where the Fourier transform is large to conclude that the sets $B(j,k,\eta_2)$ are relatively small in size.

    More precisely, 
    Theorem~\ref{thm:friendly}
    and Theorem~\ref{thm:flattening} imply that there exists $\eta_2>0$, depending on $\e$ (but not on the index $j$), such that the set 
    \begin{align*}
       Q:= \set{\b^j_{k,\ell}: \ell \in B(j,k,\eta_2)}
    \end{align*}
     can be covered by $\cO_\e(\norm{\xi}^\e)$ balls $B_i$ of radius $1/2$.
     
     Let $\tilde{B}_i$ denote the set of indices $\ell\in B(j,k,\eta_2)$ such that $\b^j_{k,\ell}\in B_i$.
     In particular, we have
     \begin{align}\label{eq:bound by balls}
         \# B(j,k,\eta_2) \leq \sum_i \#\tilde{B}_i.
     \end{align}
    Moreover, we note that for $\ell_1,\ell_2\in \tilde{B}_i$, we have that $\b^j_{k,\ell_1}-\b^j_{k,\ell_2} = \b^j_{\ell_2,\ell_1}$.
    Since $B_i$ has radius $1/2$, we get that $\norm{\b^j_{\ell_2,\ell_1}}<1$.
    Hence, recalling the definition of the sets $\mrm{Small}$ in~\eqref{eq:close freqs}, and letting $\mrm{Small}_j$ denote the set of pairs $(p,q)\in I_{\rho,t}^2$ with $(p,q,j)\in \mrm{Small}$, we obtain
    \begin{align*}
        \# \tilde{B}_i^2 \leq \# \mrm{Small}_j.
    \end{align*}
    On the other hand, since $j\notin \mrm{Bad}_\eta$, then by definition, we have that 
    \begin{align*}
        \# \mrm{Small}_j \leq \norm{\xi}^{-\alpha/48} \times (e^{\d t} + \# I_{\rho,t})\#I_{\rho,t}.
    \end{align*}
    Since the sum in~\eqref{eq:bound by balls} has at most $\cO_\e(\norm{\xi}^\e)$ terms, this estimate completes the proof.
    \qedhere

\end{proof}

\subsection*{Combining estimates and concluding the proof}

Recall that $\alpha$ is the exponent provided by Theorem~\ref{thm:friendly}.
Let $\eta_2>0$ be the exponent provided by Lemma~\ref{lem:apply flattening} when applied with $\e= \alpha/200$ and let $\kappa$ be defined as follows:
\begin{align}\label{eq:kappa}
    \kappa = \min \set{1/24, \alpha/200, \eta_2}.
\end{align}
Then, by combining~\eqref{eq:pre-Cauchy-Shwarz},~\eqref{eq:CauchySchwarz},~\eqref{eq:before flattening},~\eqref{eq:bound by count}, and Lemma~\ref{lem:apply flattening}, we obtain the following bound:
\begin{align*}
    \int_{N_1^+} e^{i\langle \xi, n\rangle} \psi(n)\;d\mu^u_x(n)
    \ll_\G & \;(\norm{\psi}_{C^1} +1) \times \norm{\xi}^{-\kappa}
    \nonumber\\
    &\times
    \left( e^{-\d t} \sum_\rho \mu^u_{y_\rho}(W_\rho) \# I_{\rho,t}
    + e^{-3\d t/4} \sum_\rho \mu^u_{y_\rho}(W_\rho) (\# I_{\rho,t})^{3/4}
    \right).
\end{align*}
The first sum on the right side is $\cO_\G(1)$ in light of~\eqref{eq:total mass} and the fact that $\mu^u_{y_\rho}(W_\rho)\asymp \mu^u_{x_{\rho,\ell}}(W_\ell)$ for all $\ell\in I_{\rho,t}$.
That the second sum is also $\cO_\G(1)$ is proved in the following lemma.
This concludes the proof of Theorem~\ref{thm:PS} apart from Lemma~\ref{lem:phi_ell formula} which is proved in the next section.
\begin{lem}\label{lem:weighted box counting}
    For every $p\in (1,\infty)$, we have that
    \begin{align*}
        e^{-\d t/p} \sum_{\rho\in\Pcal_\xi} \mu^u_{y_\rho}(W_\rho) 
        \left(\# I_{\rho,t} \right)^{1/p} \ll_\G 1.
    \end{align*}
\end{lem}
\begin{proof}
    Indeed, letting $q$ be such that $1/p + 1/q=1$, we obtain by H\"older's inequality that
    \begin{align*}
         e^{-\d t/p} \sum_{\rho\in\Pcal_\xi} \mu^u_{y_\rho}(W_\rho) 
        \left(\# I_{\rho,t} \right)^{1/p}
        \leq  e^{-\d t/p} \left(\sum_{\rho\in\Pcal_\xi} \mu^u_{y_\rho}(W_\rho) \right)^{1/q} \times
        \left(\sum_{\rho\in\Pcal_\xi}\mu^u_{y_\rho}(W_\rho)\# I_{\rho,t} \right)^{1/p}.
    \end{align*}
    Since $\mu^u_{y_\rho}(W_\rho)\asymp \mu^u_{x_{\rho,\ell}}(W_\ell)$ for all $\ell\in I_{\rho,t}$, it follows by~\eqref{eq:total mass} that
    \begin{align*}
        e^{-\d t}\sum_{\rho\in\Pcal_\xi}\mu^u_{y_\rho}(W_\rho)\# I_{\rho,t} \ll_\G 1.
    \end{align*}
    Moreover, Proposition~\ref{prop:shadow lem} implies that $\mu^u_{y_\rho}(W_\rho) \asymp \norm{\xi}^{-\d/3}$.
    Hence, the lemma follows in light of \eqref{eq:count Pcal_xi}.
\end{proof}

\subsection{Explicit formula for stable holonomy maps and Proof of Lemma~\ref{lem:phi_ell formula}}
\label{sec:temporal function}

In this section, we give explicit formulas for the commutation relations between stable and unstable subgroups which we need for the proof of Lemma~\ref{lem:phi_ell formula}.

Consider the following quadratic form on $\R^{d+2}$: for $x=(x_i)\in \R^{d+2}$, 
\begin{align*}
    Q(x) = 2x_0x_{d+1}- |x_1|^2 - \cdots - |x_{d}|^2.
\end{align*}
Let $\mrm{SO}_{\R}(Q)\cong \mrm{SO}(d+1,1)$ be the orthogonal group of $Q$; i.e.~the subgroup of $ \mrm{SL}_{d+2}(\R)$
preserving $Q$.
Then, we have a surjective homomorphism  $\mrm{SO}_\R(Q) \to G=\mrm{Isom}^+(\H^{d+1})$ with finite kernel.
The geodesic flow is induced by the diagonal group
$$A=\set{g_t= \mrm{diag}(e^t,\mrm{I}_{d},e^{-t}):t\in\R},$$
where $\mrm{I}_{d}$ denotes the identity matrix in dimension $d$.

For $x\in \R^{d}$, viewed as a row vector, we write $x^{t}$ for its transpose.
We let $\norm{x}^2 :=  x \cdot x$, and $x\cdot x$ denotes the sum of coordinate-wise products in the standard basis on $\R^d$.
Hence, $N^+$ can be parametrized as follows:
\begin{align}\label{eq:parametrizing N+}
    N^+ = \set{n^+(x):= \begin{pmatrix}
    1 & x & \frac{\norm{x}^2}{2}\\
    \mathbf{0} &\mrm{I}_{d} & x^t\\ 
    0 & \mathbf{0} & 1
    \end{pmatrix}
    : x \in \R^{d}  }.
\end{align}
The group $N^-$ is parametrized by the transpose of the elements of $N^+$.
Recall that $M=\mrm{SO}_{d}(\R)$ denotes the centralizer of $A$ inside the standard maximal compact subgroup $K\cong \mrm{SO}_{d+1}(\R)$ of $G$.
In particular, $M$ is given by
\begin{align*}
    M = \set{ m(O):= \mrm{diag}( 1, O, 1): O \in \mrm{SO}_{d}(\R) }.
\end{align*}
Finally, we recall that the product map $ N^-\times A \times M\times N^+\to G$ is a diffeomorphism near to the identity.

We are now ready for the proof. 
Recall from~\eqref{eq:stable hol map} that $\phi_\ell(n)$ is defined to be the element of $N^+$ satisfying 
$ n p_{\rho,\ell}^- \in N^- AM\phi_\ell(n)$. In particular, $\phi_\ell^{-1}(n)$ is the unique element of $N^+$ satisfying
\begin{align*}
    n (p_{\rho,\ell}^-)^{-1} \in N^-AM \phi_\ell^{-1}(n).
\end{align*}
Hence, in view of the explicit parametrization above, in order to compute $\phi_\ell^{-1}(n)$, it suffices to compute the top row of the matrix $n (p_{\rho,\ell}^-)^{-1}$ and to note that
\begin{align*}
    g_s n^+(x) = \begin{pmatrix}
    e^s & e^s x & \frac{e^s \norm{x}^2}{2}\\
    \mathbf{0} &\mrm{I}_{d} & x^t\\ 
    0 & \mathbf{0} & e^{-s}
    \end{pmatrix},
\end{align*}
for all $s\in\R$ and $x\in \R^d$. This allows us to extract the $N^+$ component from the top row of $n (p_{\rho,\ell}^-)^{-1}$ by scaling it suitably so that the top left entry is $1$.
In particular, the claimed formula follows by a direct calculation upon  recalling from~\eqref{eq:p_rho,ell} that $p_{\rho,\ell}^- = \exp(w_{\rho,\ell}) m_{\rho,\ell}g_{t_{\rho,\ell}}$.

 \section{Self-conformal measures: Proof of Theorem \ref{thm:mainNonlinear}}\label{sec:proofnon-linear}

Throughout this section we fix $\Phi,$ $\Sigma_{A},$ and $\psi$ satisfying the assumptions of Theorem \ref{thm:mainNonlinear}. To simplify our notation we will denote the Gibbs measure $\mu_{\psi}$ by $\mu$.

 \subsection{The Gibbs property and bounding the Fourier transform by an average}
	Let $\xi\in \mathbb{R}^{d}\setminus\set{0}$. The first step in our proof is to bound $|\widehat{\mu}(\xi)|$ from above by an average of the Fourier transform over multiple frequencies using \eqref{Eq:Gibbs measure invariance}. 
 
Let us define the following partition of $\Sigma_{A}$ depending upon $\xi$:
 $$\cA_{\xi}:=\set{\b\in \cW_{A}:\sup_{x\in [0,1]^{d}}\|Df_{\b}(x)\|\leq |\xi|^{-2/3}\textrm{ and }\sup_{x\in [0,1]^{d}}\|Df_{\b^{-}}(x)\|>  |\xi|^{-2/3}}.$$
For a word $\b\in \cW_{A}$ we let $\mu_{\b}$ denote the normalised restriction of $\mu_{\b}$ to the cylinder $X_{\b}$, i.e. $$\mu_{\b}:=\frac{\mu|_{X_{\b}}}{\mu(X_{\b})}.$$
For each $\b \in \cW_{A}$ we also fix a point $x_\b \in X_{\b}$ arbitrarily. Given $\xi\in\mathbb{R}^{d}\setminus \set{0}$ and a word $\b\in \cW_{A}$ we let $$\mu_{\b,\xi}:=S_{\|\xi\|^{2/3}}^{*}\mu_{\b},$$ 
 where  $S_{\|\xi\|^{2/3}}:\mathbb{R}^{d}\to\mathbb{R}^{d}$ is given by $S_{\|\xi\|^{2/3}}(x):=\|\xi\|^{2/3}\cdot x.$ We then have the following key decomposition of the Fourier transform using the Fourier transforms of $\mu_{\b,\xi}$ at random frequencies sampled using the derivative matrices:

	\begin{prop}
		\label{Prop:Averaging prop}
	If $\xi\in\mathbb{R}^{d}\setminus \set{0}$ and $n(\xi) = \lfloor c\log \|\xi\|\rfloor$ for some constant $c > 0$, we have
	\begin{align*}
		|\widehat{\mu}(\xi)|\leq \sum_{\b\in \cA_{\xi}}\mu(X_{\b}) \sum_{\substack{\a\in \cA^{n(\xi)} \\ \a \leadsto \b}} w_{\a}(x_\b) \left|\widehat{\mu_{\b,\xi}}\left(( D_{x_{\b}}f_{\a})^{T}\frac{\xi}{\|\xi\|^{2/3}}\right)\right|+\cO(\|\xi\|^{-1/3}).
		\end{align*}
	\end{prop}
\begin{proof}
Applying \eqref{Eq:Gibbs measure invariance} and then splitting our integral over cylinder sets $X_\b$ we have 
\begin{align*}
\int e^{2\pi i\langle\xi,x\rangle}\, d\mu(x)&=\int \sum_{\substack{\a \in \cA^{n(\xi)} \\ \a \leadsto x}}w_{\a}(x)e^{2\pi i\langle\xi,f_{\a}(x)\rangle}\, d\mu(x)\\
&=\sum_{\b\in \cA_{\xi}}\sum_{\substack{\a \in \cA^{n(\xi)} \\ \a \leadsto \b}}\int_{X_{\b}} w_{\a}(x)e^{2\pi i\langle\xi,f_{\a}(x)\rangle}\, d\mu(x)\\
&=\sum_{\b\in \cA_{\xi}}\mu(X_\b)\sum_{\substack{\a \in \cA^{n(\xi)} \\ \a \leadsto \b}} \int  w_{\a}(x)e^{2\pi i\langle\xi,f_{\a}(x)\rangle}\, d\mu_\b(x)
	\end{align*}
 In the penultimate line we used that for a point $x\in X_{\b}$ we have $\a \leadsto x$ if and only if $\a \leadsto \b$. Using our assumption $P(\psi)=0,$ we see that the Gibbs property \eqref{eq:GibbsProperty} gives a constant $C>0$ such that $ |w_{\a}(x)|\leq C$ for all $x\in X.$ Using this inequality together with the fact our potential $\psi$ is $C^{1}$ we can show that for any $x,y\in X_{A}$ we have  
 \begin{equation}
		\label{eq:Gibbs weights bound}|w_{\a}(x)-w_{\a}(y)|= |w_{\a}(x)||1-e^{S_{n}\psi(f_{\a}(y))-S_{n}\psi(f_{\a}(x))}|\ll |x-y|.
	\end{equation} 
Thus for any $\b \in \cA_\xi$ and $x\in X_{\b}$   we have
$$|w_{\a}(x)-w_{\a}(x_{\b})| \ll \|\xi\|^{-2/3}$$
since $|x-x_\b| \ll |\xi|^{-2/3}$ by Lemma \ref{Lemma:Derivative and Diameter}. Thus 
\begin{align*}
& \sum_{\b\in \cA_{\xi}}\mu(X_\b)\sum_{\substack{\a \in \cA^{n(\xi)} \\ \a \leadsto \b}} \int  w_{\a}(x)e^{2\pi i\langle\xi,f_{\a}(x)\rangle}\, d\mu_\b(x)\\
 =&\sum_{\b\in \cA_{\xi}}\mu(X_\b)\sum_{\substack{\a \in \cA^{n(\xi)} \\ \a \leadsto \b}} w_{\a}(x_\b) \int e^{2\pi i\langle\xi,f_{\a}(x)\rangle}\, d\mu_\b(x) +\cO(\|\xi\|^{-2/3}).
\end{align*}

Now, towards the claim, we need to linearize $f_\a(x)$ for $x \in X_\b$ using the value of the derivative at $x_\b \in X_\b$. For this purpose, given $t\in[0,1]$ let $x_{t}=x_{\b}+t(x-x_{\b})$. Using this notation we have
\begin{equation}
	\label{eq:linearising}
	f_{\a}(x)=f_{\a}(x_{\b})+\Big(\int_{0}^{1} D_{x_{t}}f_{\a}\, dt\Big) (x-x_{\b}).
	\end{equation}
By Lemma \ref{Lemma:Bounded distortions} we have 
$$\|D_{x_{t}}f_{\a}-D_{x_{\b}}f_{\a}\|\ll \sup_{x\in [0,1]^{d}}\|Df_{\a}(x)\|\cdot \|x_{t}-x_{\b}\|=\mathcal{O}(\|\xi\|^{-2/3})$$
because $\|x_{t}-x_{\b}\| \ll \|\xi\|^{-2/3}$ for all $t\in [0,1]$. This in turn implies that for $x \in X_\b$ we have
$$\Big\|\Big(\int_0^1 D_{x_{t}}f_{\a}-D_{x_{\b}}f_{\a} \, dt \Big)(x-x_\b)\Big\|\ll\|\xi\|^{-2/3} \|x-x_{\b}\| \ll \|\xi\|^{-4/3}.$$
Thus by \eqref{eq:linearising} we have
	$$f_{\a}(x)=f_{\a}(x_{\b})+ D_{x_{\b}}f_{\a}(x-x_{\b})+\cO(\|\xi\|^{-4/3}).$$
By the $1$-Lipschitz property of $x \mapsto e^{ix}$ and Cauchy-Schwartz inequality, we have for the matrix $A_\b = D_{x_{\b}}f_{\a}$ and a vector $v_\b = f_{\a}(x_{\b}) - D_{x_{\b}}f_{\a}x_\b$ that
$$\Big|\int e^{2\pi i\langle \xi , f_\a(x)\rangle } - e^{2\pi i\langle \xi , A_\b x + v_\b\rangle } \, d\mu_\b(x)\Big| \ll \|\xi\| \sup_{x \in X_\b}\| f_\a(x) - A_\b x - v_\b \| \ll \|\xi\|^{-1/3}.$$
Thus we have proved:
	\begin{align*}
	\int e^{2\pi i\langle\xi,x\rangle}\, d\mu(x)=\sum_{\b\in \cA_{\xi}}\mu(X_\b)\sum_{\substack{\a \in \cA^{n(\xi)} \\ \a \leadsto \b}} w_{\a}(x_\b) \int e^{2\pi i\langle\xi,f_{\a}(x_{\b})+ D_{x_{\b}}f_{\a}(x-x_{\b})\rangle}\, d\mu_\b(x)+\cO(\|\xi\|^{-1/3}).
	\end{align*}
Taking absolute values and applying the triangle inequality this implies:
	\begin{align*}
	\int e^{2\pi i\langle\xi,x\rangle}\, d\mu(x)\leq \sum_{\b\in \cA_{\xi}}\mu(X_\b)\sum_{\substack{\a \in \cA^{n(\xi)} \\ \a \leadsto \b}} w_{\a}(x_\b) \left|\int e^{2\pi i\langle\xi, D_{x_{\b}}f_{\a}x\rangle}\, d\mu_\b(x)\right|+\cO(\|\xi\|^{-1/3}).
	\end{align*}	
	 Recalling the definition of $\mu_{\b,\xi}=S_{\|\xi\|^{2/3}}^{*}\mu_{\b},$ 
 where $S_{\|\xi\|^{2/3}}(x)=\|\xi\|^{2/3}\cdot x$, and taking the transpose of the matrices appearing in our inner product yields:
	\begin{align*}
		\int e^{2\pi i\langle\xi,x\rangle}\, d\mu(x)
\leq & \sum_{\b\in \cA_{\xi}}\mu(X_\b)\sum_{\substack{\a \in \cA^{n(\xi)} \\ \a \leadsto \b}} w_{\a}(x_\b)\left|\int e^{2\pi i\langle( D_{x_{\b}}f_{\a})^{T}\frac{\xi}{\|\xi\|^{2/3}},x\rangle}\, d\mu_{\b,\xi}(x)\right|+\cO(\|\xi\|^{-1/3})\\
=&\sum_{\b\in \cA_{\xi}}\mu(X_\b)\sum_{\substack{\a \in \cA^{n(\xi)} \\ \a \leadsto \b}} w_{\a}(x_\b)\left|\widehat{\mu_{\b,\xi}}\left(( D_{x_{\b}}f_{\a})^{T}\frac{\xi}{\|\xi\|^{2/3}}\right)\right|+\cO(\|\xi\|^{-1/3}).
		\end{align*} 
This completes the proof.  
\end{proof}

\subsection{Affine non-concentration of the averaged measures}

To successfully apply Theorem \ref{thm:flattening} to to the measures $\mu_{\b,\xi}$ appearing in Proposition \ref{Prop:Averaging prop}, we need to show that they are uniformly affinely non-concentrated and that the non-concentration parameters can be taken to be independent of $\b$ and $\xi$. This property is guaranteed by the following proposition. 

\begin{prop}
	\label{Prop:affine non-concentration prop}
	Assume that $\mu$ is uniformly affinely non-concentrated. Then there exists $\delta':(0,\infty)\to(0,\infty)$ satisfying the following:
     \begin{enumerate}
         \item $\delta'(\epsilon)\to 0$ as $\epsilon\to 0.$
         \item For every $\xi\neq 0,$  $\b \in A_{\xi},$ $x\in \supp(\mu_{\b,\xi}),$ $0<r\leq 1,$ $\eps>0$ and affine hyperplane $W<\R^d$ we have $$\mu_{\b,\xi}(W^{(\epsilon r)}\cap B(x,r))\leq \delta'(\epsilon)\mu(B(x,r))$$
     \end{enumerate} 
\end{prop}
\begin{proof}
Fix $\xi\neq 0$ and $\b \in \cA_{\xi}.$ Let $r\in(0,1]$ and $\epsilon>0$. Then for any $x\in \supp(\mu_{\b,\xi})$ and affine hyperplane $W<\R^{d}$ there exists $V< \mathbb{R}^{d}$ and $y\in X_{\b}$ such that
\begin{align*}
	\mu_{\b,\xi}(W^{(\epsilon r)}\cap B(x,r))&=\mu_{\b}(V^{(\epsilon\cdot \|\xi\|^{-2/3} r)}\cap B(y,r\cdot \|\xi\|^{-2/3}))\\
	&=\frac{\mu|_{X_{\b}}(V^{(\epsilon\cdot \|\xi\|^{-2/3} r)}\cap B(y,r\cdot \|\xi\|^{-2/3}))}{\mu(X_{\b})}\\
	&\leq \frac{\mu(V^{(\epsilon\cdot \|\xi\|^{-2/3} r)}\cap B(y,r\cdot \|\xi\|^{-2/3}))}{\mu(X_{\b})}\\
	&\leq \frac{\delta(\epsilon)\mu(B(y,r\cdot \|\xi\|^{-2/3}))}{\mu(X_{\b})}.
\end{align*} Where in the final line we let $\delta:(0,\infty)\to(0,\infty)$ be as in the definition of uniformly affinely non-concentrated. Therefore to complete our proof it remains to show that 
\begin{equation*}
\frac{\mu(B(y,r\cdot \|\xi\|^{-2/3}))}{\mu(X_{\b})}\ll \mu_{\b,\xi}(B(x,r)),
\end{equation*} which by the definition of $\mu_{\b,\xi}$ is equivalent to 
\begin{equation}
	\label{eq:ball measure}
	\mu(B(y,r\cdot \|\xi\|^{-2/3}))\ll \mu(X_{\b} \cap B(y,r\cdot \|\xi\|^{-2/3})).
\end{equation} Recall that $\mrm{Diam}(X_{\b})\asymp \|\xi\|^{-2/3}$ since $\b \in \cA_\xi$. It follows from this observation and the Strong Separation Condition that that there exists $k\in\mathbb{N}$ depending only upon the underlying IFS such that one of two cases must occur:
\begin{enumerate}
	\item $X_{\b\a}\subset B(y,r\|\xi\|^{-2/3})\cap X_{A}\subset X_{b_{1},\ldots, b_{|\b|-k}}$ for some $\a\in \cA^{k}.$
	\item There exists $\a\in \cA^{*}$ and $\a^{*} \in \cA^{k}$ such that $X_{\b\a\a^{*}}\subset B(y,r\|\xi\|^{-2/3})\cap X_{A}\subset X_{\b\a}.$
\end{enumerate} In the first case, it follows from the quasi-Bernoulli property of $\mu$ \eqref{eq:quasibernoulli} that
$$\mu(X_{\b} \cap B(y,r\cdot \|\xi\|^{-2/3}))\geq \mu(X_{\b\a})\gg \mu(X_{b_{1},\ldots b_{|\b|-k}})\geq \mu(B(y,r\|\xi\|^{-2/3}).$$
Again appealing to the quasi-Bernoulli property of $\mu$, in the second case we have 
$$\mu(X_{\b} \cap B(y,r\cdot \|\xi\|^{-2/3}))\geq  \mu(X_{\b\a\a^{*}})\gg \mu(X_{\b\a})\geq \mu(B(y,r\|\xi\|^{-2/3}).$$ Thus \eqref{eq:ball measure} holds and our proof is complete.
\end{proof} 

\subsection{Spectral gap and non-concentration of frequencies}
We now fix the parameter $c$ appearing in Proposition \ref{Prop:Averaging prop}. We assume that $c$ is sufficiently small that 
\begin{equation}
\label{Eq:Choosing c appropriately}
\Big\|(Df_{\a}(x_{\b}))^{T}\frac{\xi}{|\xi|^{2/3}}\Big\|\geq \|\xi\|^{1/6}
\end{equation}for any $\a\in \cA^{n(\xi)}$ and $\b \in \cW_{A}$ with $\a \leadsto \b$. The final part we need for our proof of Theorem \ref{thm:mainNonlinear} is the following proposition. It is in the proof of this proposition that we will use our spectral gap assumption.

\begin{prop}
\label{Prop:Self-conformal non-concentration}
There exists $\gamma>0$ such that for any $\b\in \cA_{\xi}$ and integer $n\in \mathbb{N}$ satisfying $\|\xi\|^{1/6}\leq n\leq \|\xi\|^{1/3}$ we have 
	$$\sum_{\substack{\a\in \cA^{n(\xi)} \\ \a \leadsto \b}} w_{\a}(x_\b)\chi_{[n,n+1]}\left( \Big\|( D_{x_{\b}}f_{\a})^{T}\frac{\xi}{\|\xi\|^{2/3}} \Big\|\right)=\mathcal{O}(\|\xi\|^{-\gamma}).$$
\end{prop}

\begin{proof}
Let $n\in \N$ satisfy $\|\xi\|^{1/6}\leq n\leq \|\xi\|^{1/3}.$ Observe that 
\begin{align*}
    \Big\|( D_{x_{\b}}f_{\a})^{T}\frac{\xi}{\|\xi\|^{2/3}} \Big\|\in [n,n+1]
    \Leftrightarrow &\Big|\prod_{i=1}^{n(\xi)}\lambda_{a_{i}}(f_{a_{i+1}\ldots a_{n}}(x_{\b}))\Big|\cdot \Big\|\frac{\xi}{|\xi|^{2/3}}\Big\|\in[n,n+1]\\
    \Leftrightarrow &\log \Big|\prod_{i=1}^{n(\xi)}\lambda_{a_{i}}(f_{a_{i+1}\ldots a_{n}}(x_{\b}))\Big|\in I_{n,\xi}
\end{align*}
for $$I_{n,\xi} = [\log(\|\xi\|^{-1/3} n),\log(\|\xi\|^{-1/3} (n+1))].$$ Thus we have
\begin{equation}
\sum_{\substack{\a\in \cA^{n(\xi)} \\ \a \leadsto \b}} w_{\a}(x_\b)\chi_{[n,n+1]}\left( \Big\|( D_{x_{\b}}f_{\a})^{T}\frac{\xi}{\|\xi\|^{2/3}} \Big\|\right)=\sum_{\substack{\a\in \cA^{n(\xi)} \\ \a \leadsto \b}} w_{\a}(x_\b)\chi_{I_{n,\xi}}\left(\log \left|\prod_{i=1}^{n(\xi)}\lambda_{a_{i}}(f_{a_{i+1}\ldots a_{n}}(x_{\b}))\right|\right)
\end{equation}
By the mean value theorem $$|I_{n,\xi}| \asymp \frac{1}{n} \leq |\xi|^{-1/6}.$$  We replace $I_{n,\xi}$ with an interval $\tilde{I}_{n,\xi}$ that contains $I_{n,\xi}$ and is of length $|\xi|^{-\delta}$ where $$\delta<\min\set{1/6,-c\log \rho}.$$ Here $\rho\in(0,1)$ is as in the definition of spectral gap and $c$ is as in \eqref{Eq:Choosing c appropriately}. We let $h_{n,\xi}$ be a smooth mollifier approximating $\chi_{\tilde{I}_{n,\xi}}$ such that $h_{n,\xi}\geq \chi_{\tilde{I}_{n,\xi}},$ $\|h_{n,\xi}\|_{1}\ll |\xi|^{-\delta}$ and $\|h_{n,\xi}''\|_{1}\ll |\xi|^{\delta}$ as in \cite{BakerSahlsten,SahlstenStevens}. Using Fourier inversion we have the following 
\begin{align*}
	&\sum_{\substack{\a\in \cA^{n(\xi)} \\ \a \leadsto \b}} w_{\a}(x_\b)\chi_{I_{n,\xi}}\left(\log \left|\prod_{i=1}^{n(\xi)}\lambda_{a_{i}}(f_{a_{i+1}\ldots a_{n}}(x_{\b}))\right|\right)\\
	 \leq&\sum_{\substack{\a\in \cA^{n(\xi)} \\ \a \leadsto \b}} w_{\a}(x_\b)h_{n,\xi}\left(\log\left |\prod_{i=1}^{n(\xi)}\lambda_{a_{i}}(f_{a_{i+1}\ldots a_{n}}(x_{\b}))\right|\right)\\
	=&\sum_{\substack{\a\in \cA^{n(\xi)} \\ \a \leadsto \b}} w_{\a}(x_\b)\int_{-\infty}^{\infty}\widehat{h_{n,\xi}}(\eta)e^{-2\pi i \eta\log |\prod_{i=1}^{n(\xi)|}\lambda_{a_{i}}(f_{a_{i+1}\ldots a_{n}}(x_{\b}))|}\, d\eta\\
	=&\sum_{\substack{\a\in \cA^{n(\xi)} \\ \a \leadsto \b}} w_{\a}(x_\b)\int_{-\infty}^{\infty}\widehat{h_{n,\xi}}(\eta)\prod_{i=1}^{n(\xi)}|\lambda_{a_{i}}(f_{a_{i+1}\ldots a_{n}}(x_{\b}))|^{2\pi i \eta}\, d\eta\\
 =&\int_{-\infty}^{\infty}\widehat{h_{n,\xi}}(\eta)\mathcal{L}_{2\pi i\eta}^{n(\xi)}(1)(x_{\b}) d\eta.\\
 \end{align*}
Let $\Theta>0$ be some parameter such that our spectral gap bound can be applied for $|\eta|\geq \Theta.$ Splitting the integral above we have:
 \begin{align*}
	\int_{-\infty}^{\infty}\widehat{h_{n,\xi}}(\eta)\mathcal{L}_{2\pi i\eta}^{n(\xi)}(1)(x_{\b}) d\eta=&\int_{|\eta|<\Theta}\widehat{h_{n,\xi}}(\eta)\mathcal{L}_{2\pi i\eta}^{n(\xi)}(1)(x_{\b}) d\eta+\int_{|\eta|\geq \Theta}\widehat{h_{n,\xi}}(\eta)\mathcal{L}_{2\pi i\eta}^{n(\xi)}(1)(x_{\b}) d\eta\\
	\ll& |\widehat{h_{n,\xi}}(0)|+\int_{|\eta|\geq \Theta}\widehat{h_{n,\xi}}(\eta)\mathcal{L}_{2\pi i\eta}^{n(\xi)}(1)(x_{\b}) d\eta\\
	 \ll&|\xi|^{-\delta}+\int_{|\eta|\geq \Theta}\widehat{h_{n,\xi}}(\eta)\mathcal{L}_{2\pi i\eta}^{n(\xi)}(1)(x_{\b}) d\eta\\
\end{align*}
Applying our spectral gap bound we have
$$\|\mathcal{L}_{2\pi i\eta}^{n(\xi)} 1\|_\infty \ll \rho^{n(\xi)}|\eta|^{\upsilon}$$
for some $\upsilon\in (0,1)$. Since $n(\xi) = \lfloor c \log \|\xi\|\rfloor$, this gives us:
\begin{align*}
	\left|\int_{|\eta|\geq \Theta}\widehat{h_{n,\xi}}(\eta)\mathcal{L}_{2\pi i\eta}^{n(\xi)}(1)(x_{\b}) d\eta\right|
	\leq &\int_{|\eta|\geq \Theta}|\widehat{h_{n,\xi}}(\eta)|\rho^{n(\xi)}|\eta|^{\upsilon}\, d\eta\\
	\ll &\|\xi\|^{c\log \rho}\int_{|\eta|\geq \Theta}|\widehat{h_{n,\xi}}(\eta)||\eta|^{\upsilon}\, d\eta.
 \end{align*}
Using integration by parts it can be shown that 
$$|\widehat{h_{n,\xi}}(\eta)|\ll \frac{1}{1+|\eta|^2}(\|h_{n,\xi}\|_{1}+\|h_{n,\xi}''\|_{1}).$$ Substituting this bound into the above yields:
 \begin{align*}
 \|\xi\|^{c\log \rho}\int_{|\eta|\geq \Theta}|\widehat{h_{n,\xi}}(\eta)||\eta|^{\upsilon}\, d\eta 	\leq& \|\xi\|^{c\log \rho}\int_{|\eta|\geq \Theta}\frac{|\eta|^{\upsilon}}{1+|\eta|^2}(\|h_{n,\xi}\|_{1}+\|h_{n,\xi}''\|_{1})\, d\eta\\
	\ll& \|\xi\|^{c\log \rho}(\|h_{n,\xi}\|_{1}+\|h_{n,\xi}''\|_{1})\\
	\ll& \|\xi\|^{\delta+c\log\rho }.
	\end{align*}
In the penultimate line we used that $\upsilon\in(0,1)$ thus $\int_{|\eta|\geq \Theta}\frac{|\eta|^{\upsilon}}{1+|\eta|^2}\, d\eta\ll 1$. In the final line we used that $\|h_{n,\xi}\|_{1}\leq \|\xi\|^{-\delta}$ and $\|h_{n,\xi}''\|_{1}\leq \|\xi\|^{\delta}$. By the definition of $\delta$ we know that $\delta+c\log\rho <0$. This completes the proof of this proposition.
\end{proof}

\subsection{Proof of Theorem~\ref{thm:mainNonlinear}}

Equipped with the propositions above we can now prove Theorem \ref{thm:mainNonlinear}.
By Proposition \ref{Prop:Averaging prop}, we have
	\begin{align*}
		|\widehat{\mu}(\xi)|\leq \sum_{\b\in \cA_{\xi}}\mu(X_{\b}) \sum_{\substack{\a\in \cA^{n(\xi)} \\ \a \leadsto \b}} w_{\a}(x_\b) \left|\widehat{\mu_{\b,\xi}}\left(( D_{x_{\b}}f_{\a})^{T}\frac{\xi}{\|\xi\|^{2/3}}\right)\right|+\cO(\|\xi\|^{-1/3}).
		\end{align*}
Thus to prove our result it suffices to show that 
\begin{equation}
    \label{Eq:WTS non-linear}
    \sum_{\b\in \cA_{\xi}}\mu(X_{\b}) \sum_{\substack{\a\in \cA^{n(\xi)} \\ \a \leadsto \b}} w_{\a}(x_\b) \left|\widehat{\mu_{\b,\xi}}\left(( D_{x_{\b}}f_{\a})^{T}\frac{\xi}{\|\xi\|^{2/3}}\right)\right|=\cO(\|\xi\|^{-\eta}),
\end{equation} for some $\eta>0.$ 

Let $\gamma$ be as in Proposition \ref{Prop:Self-conformal non-concentration}. Applying Theorem \ref{thm:flattening}, we can assert that there exists $\tau>0$ such that for each $\b\in \cA_{\xi},$ if we let $$\mathrm{Bad}_{\xi,\b} := \set{
n\in \mathbb{Z}: |n|\leq |\xi|^{1/3}, \exists \eta \in \R^d \textrm{ such that }\norm{\eta}\in [n,n+1] \textrm{ and }|\widehat{\mu_{\b,\xi}}(\eta)|\geq \|\xi\|^{-\tau}}$$ then 
\begin{equation}
    \label{Eq:Counting bad intervals}
    \#\mathrm{Bad}_{\xi,\b}=\cO_{\gamma}(\|\xi\|^{\gamma/2}).
\end{equation} We emphasise that the underyling constants in \eqref{Eq:Counting bad intervals} do not depend upon $\b$ or $\xi$. This follows from Proposition \ref{Prop:affine non-concentration prop} which asserts that the non-concentration parameters can be taken to be independent of $\b$ and $\xi.$ By our choice of $c$ we know that 
$$\left\|( D_{x_{\b}}f_{\a})^{T}\frac{\xi}{\|\xi\|^{2/3}}\right\|\geq \|\xi\|^{1/6}$$ for all $\b\in \cA_{\xi}$ and $\a\in \cA^{n(\xi)}$ satisfying $\a\leadsto\b.$ Therefore we can apply Proposition \ref{Prop:Self-conformal non-concentration} to the frequencies appearing in \eqref{Eq:WTS non-linear}. In particular, using this proposition and \eqref{Eq:Counting bad intervals} we have 

\begin{align*}
	& \sum_{\b\in \cA_{\xi}}\mu(X_{\b}) \sum_{\substack{\a\in \cA^{n(\xi)} \\ \a \leadsto \b}} w_{\a}(x_\b) \left|\widehat{\mu_{\b,\xi}}\left(( D_{x_{\b}}f_{\a})^{T}\frac{\xi}{\|\xi\|^{2/3}}\right)\right|  \\
	\leq &\sum_{\b\in A_{\xi}}\mu(X_{\b})\sum_{n\in \mathrm{Bad}_{\xi,\b}}\sum_{\substack{\a\in \cA^{n(\xi)} \\ \a \leadsto \b}}w_{\a}(x_{\b})\chi_{[n,n+1]}\left(\left\|( D_{x_{\b}}f_{\a})^{T}\frac{\xi}{\|\xi\|^{2/3}}\right\|\right)\\
	+ &\sum_{\b\in A_{\xi}}\mu(X_{\b})\sum_{\substack{\a\in \cA^{n(\xi)} \\ \a \leadsto \b}} \quad w_{\a}(x_{\b})\left|\widehat{\mu_{\b,\xi}}\left(( D_{x_{\b}}f_{\a})^{T}\frac{\xi}{\|\xi\|^{2/3}}\right)\right|\chi_{(\cup_{n\in \mrm{Bad}_{\xi,\b}}[n,n+1])^{c}}\left(\left\|( D_{x_{\b}}f_{\a})^{T}\frac{\xi}{\|\xi\|^{2/3}}\right\|\right)\\
	\ll&\sum_{\b\in A_{\xi}}\mu(X_{\b})\sum_{n\in \mathrm{Bad}_{\xi,\b}}\|\xi\|^{-\gamma}+\sum_{\b\in A_{\xi}}\mu(X_{\b})\sum_{\substack{\a\in \cA^{n(\xi)} \\ \a \leadsto \b}}w_{\a}(x_{\b})\|\xi\|^{-\tau}\\
	\ll & \left(\|\xi\|^{-\gamma/2}+\|\xi\|^{-\tau}\right).
\end{align*}  In the final line we used the trivial estimates: $$\sum_{\b\in A_{\xi}}\mu(X_{\b})=1\quad \textrm{and} \sum_{\substack{\a\in \cA^{n(\xi)} \\ \a \leadsto \b}}w_{\a}(x_{\b})\ll 1.$$ 
Therefore, \eqref{Eq:WTS non-linear} holds and our proof is complete.


\section{Non-conformal systems: Proof of Theorem \ref{Theorem:Non-conformal}}\label{sec:proofnonconf}

\subsection{Outline of the argument}

We begin this section by outlining our proof of Theorem \ref{Theorem:Non-conformal}. Given a frequency $\xi\in\mathbb{R}^{d},$ we will begin by picking a dominant direction $1\leq i^*\leq d$ such that $\max_{1\leq i\leq d}|\xi_i|=|\xi_{i^{*}}|$. Without loss of generality we can take $i^*=1$. We will then disintegrate our stationary measure in this direction to express it as an integral of random measures. Each of these random measures will be of the form $\mu_{\b}\times \delta_{x_{\b}}$, where $\mu_{\b}$ is a probability measure on $[0,1]$ and $\delta_{x_{\b}}$ is simply a Dirac mass supported on a point $x_{\b}\in[0,1]^{d-1}$.

The fact that our random measures are of this form and $\max_{1\leq i\leq d}|\xi_i|=|\xi_{1}|$ will mean that to prove Theorem \ref{Theorem:Non-conformal}, it is sufficient to show that a typical $\mu_{\b}$ will have polynomial Fourier decay. This will be done by using the strategy employed throughout this paper. 

To this end, we will bound $\widehat{\mu_{\b}}(\xi_1)$ from above by an average of Fourier transforms over a set of measures and a set of frequencies. We will show that the measures appearing in this average are all uniformly affinely non-concentrated and the underlying non-concentration parameters can be taken to be independent of the measure (Proposition \ref{Prop:Non-conformal affinely non-concentrated}). We will also show that the set of frequencies appearing in this average typically satisfy a non-concentration property (Proposition \ref{Prop:Applying the random spectral gap}). We will then apply Theorem \ref{thm:flattening} to prove the desired Fourier decay result. 

The major difference in our proof when compared to Theorem \ref{thm:mainNonlinear} is that the measures $\mu_{\b}$ are no longer self-conformal measures and only satisfy a weaker form of dynamical self-conformality (see Lemma \ref{Lemma:Dynamical self-conformality}). It is a consequence of this weaker form of self-conformality that we do not consider a single transfer operator in our analysis, but instead must consider a family of transfer operators and how they behave under random compositions. Here we will rely heavily on work on the first and third authors \cite{BakerSahlsten}, which establishes that  a spectral gap typically exists for these random compositions under the third assumption of Theorem \ref{Theorem:Non-conformal}. 

The aforementioned technique of disintegrating a stationary measure in terms of simpler, albeit random, measures has its origins in a paper of Galicer et al \cite{GSSY}. It has subsequently been applied to problems relating to absolute continuity of self-similar measures \cite{KO,SSS}, normal numbers in fractal sets \cite{ABS}, and Fourier transforms of fractal measures \cite{BakerBanaji}.

\subsection{Beginning our proof and disintegrating $\mu$.}
Let us begin our proof of Theorem \ref{Theorem:Non-conformal} by fixing a restricted product IFS $\set{F_{\oa}}_{\oa\in \cA}$ satisfying the assumptions of this theorem and a probability vector $\p=(p_{\oa})_{\oa\in \cA}.$ To ease our notation we will denote the corresponding stationary measure by $\mu$ instead of $\mu_{\p}$. Let us also fix $\xi\in \mathbb{R}^{d}$. Without loss of generality we will assume that $\max_{1\leq i\leq d}|\xi|=|\xi_1|$. Thus to prove our result it suffices to show that 
\begin{equation}
	\label{Eq:xi_1 decay}
|\widehat{\mu}(\xi)|\ll |\xi_{1}|^{-\eta}
\end{equation}
for some $\eta>0$. Our first step towards proving \eqref{Eq:xi_1 decay} is to disintegrate our measure $\mu$. 
This is done in Proposition~\ref{Prop:Disintegration prop} after introducing the necessary notation.

Let $\pi:\cA\to  \cA_{2}\times \cdots \times \cA_{d}$ be the map given by $\pi(a_1,\ldots,a_{d})=(a_2,\ldots,a_d)$. We define $\cB=\pi(\cA)$ and define a new probability vector $\q=(q_{\overline{b}})_{\overline{b}\in \cB}$ by the formula $$q_{\overline{b}}=\sum_{\overline{a}:\pi(\overline{a})=\overline{b}}p_{\overline{a}}.$$ We let $Q$ be the probability measure on $\cB^{\N}$ corresponding to $\q$. Given $\overline{b}=(b_1,\ldots,b_{d-1})\in \cB$ we define $\tilde{F_{\overline{b}}}:[0,1]^{d-1}\to [0,1]^{d-1}$ to be the contraction given by $$\tilde{F_{\overline{b}}}(x_{1},\ldots, x_{d-1})=(f_{b_{1}}^{(2)}(x_{1}),\ldots, f_{b_{d-1}}^{(d)}(x_{d-1})).$$ 
Given $\overline{b}\in \cB$, we let $\cA_{\overline{b}}=\set{a\in \cA_{1}:(a,\overline{b})\in \cA}.$ 
For $\b=(\overline{b}_i)\in \cB^{\N}$, we also define $\Sigma_{\b}=\prod_{i=1}^{\infty}\cA_{\overline{b}_i}$ and let $x_{\b}\in [0,1]^{d-1}$ be the point given by the formula $$x_{\b}:=\lim_{n\to\infty}(\tilde{F_{\overline{b}_1}}\circ \cdots \circ \tilde{F_{\overline{b}_n}})(0)$$ We define a Bernoulli measure on $\Sigma_{\b}$ according to the rule $$m_{\b}([a_1,\ldots,a_n])=\prod_{i=1}^{n}\frac{p_{a\overline{b_{i}}}}{q_{\overline{b_{i}}}}$$ and $\pi_{\b}:\Sigma_{\b}\to [0,1]$ according to the rule $$\pi_{\b}(\a)=\lim_{n\to\infty}(f_{a_1}^{(1)}\circ \cdots \circ f_{a_n}^{(1)})(0).$$ We let $$X_{\b}:=\pi_{\b}(\Sigma_{\b}) \textrm{ and }\mu_{\b}:=\pi_{\b}m_{\b}.$$ We remark that it follows from the second assumption in Theorem \ref{Theorem:Non-conformal} that $\# \cA_{\overline{b}}\geq 2$ for all  $\overline{b}\in \cB.$ Consequently, there exists $\gamma\in(0,1)$ such that 
\begin{equation}
	\label{Eq:Exponential decay for cylinders}
\max_{\overline{b}\in \cB}\max_{a\in \cA_{\overline{b}}}\frac{p_{a\overline{b}}}{q_{\overline{b}}}\leq \gamma.
 \end{equation}
We let $\sigma:\cB^{\N}\to \cB^{\N}$ be the left shift map given by $\sigma((\overline{b}_i)_{i=1}^{\infty})=(\overline{b}_{i+1})_{i=1}^{\infty}$. The following lemma shows that the measures $\mu_{\b}$ exhibit a form of dynamical self-conformality.

\begin{lem}
	\label{Lemma:Dynamical self-conformality}
	Let $\b\in \cB^{\N}$. Then for all $n\in \N$ we have $$\mu_{\b}=\sum_{\a\in \cA_{\overline{b}_1}\times \cdots \times\cA_{\overline{b}_n}}m_{\b}([\a])f_{\a}\mu_{\sigma^{n}(\b)}.$$
\end{lem}  
\begin{proof}
Let $\b\in \cB^{\N}$ and $n\in \N$. By considering $\mu_{\b}$ and $\mu_{\sigma^{n}(\b)}$ as weak star limits of weighted Dirac masses we have
\begin{align*}
\mu_{\b}=&\lim_{M\to\infty}\sum_{\a\in \cA_{\overline{b}_1}\times \cdots \cA_{\overline{b}_M}}m_{\b}([\a])\delta_{f_{\a}(0)}\\
=&\lim_{M\to\infty}\sum_{\a\in \cA_{\overline{b}_1}\times \cdots \cA_{\overline{b}_n}}m_{\b}([\a]) \sum_{\a'\in \cA_{\overline{b}_{n+1}}\times \cdots \cA_{\overline{b}_M}}m_{\sigma^{n}(\b)}([\a'])\delta_{f_{\a\a'}(0)}\\
=&\lim_{M\to\infty}\sum_{\a\in \cA_{\overline{b}_1}\times \cdots \cA_{\overline{b}_n}}m_{\b}([\a]) f_{\a}\left(\sum_{\a'\in \cA_{\overline{b}_{n+1}}\times \cdots \cA_{\overline{b}_M}}m_{\sigma^{n}(\b)}([\a'])\delta_{f_{\a'}(0)}\right)\\
=&\sum_{\a\in \cA_{\overline{b}_1}\times \cdots \times\cA_{\overline{b}_n}}m_{\b}([\a])f_{\a}\mu_{\sigma^{n}(\b)}.
\end{align*}
\end{proof}

 The following proposition is our aforementioned disintegration result.

\begin{prop}We have
	\label{Prop:Disintegration prop}
	$$\mu=\int_{\cB^{\mathbb{N}}} \mu_{\b}\times \delta_{x_{\b}} \, dQ.$$
\end{prop}
\begin{proof}
Consider the Borel probability measure $\tilde{\mu}$ given by 
$$\tilde{\mu}=\int_{\cB^{\mathbb{N}}} \mu_{\b}\times \delta_{x_{\b}}\, dQ.$$
We will show that $\tilde{\mu}$ satisfies 
\begin{equation}
	\label{Eq:Self-conformal equation}
	\tilde{\mu}=\sum_{\overline{a}\in \cA}p_{\overline{a}}F_{\overline{a}}\tilde{\mu}.
\end{equation}
Since the stationary measure $\mu$ is the unique Borel probability measure satisfying this equation, our result will follow. 

Using Lemma \ref{Lemma:Dynamical self-conformality}, we have 
\begin{align*}
\tilde{\mu}=\int_{\cB^{\mathbb{N}}} \mu_{\b}\times \delta_{x_{\b}}\, dQ=&\int_{\cB^{\mathbb{N}}} \sum_{a\in \cA_{\overline{b}_1}}m_{\b}([a])f_{a}\mu_{\sigma(\b)}\times \delta_{x_{\b}}\, dQ\\
=&\int_{\cB^{\mathbb{N}}} \sum_{a\in \cA_{\overline{b}_1}}m_{\b}([a])f_{a}\mu_{\sigma(\b)}\times \tilde{F}_{\overline{b}_1}\delta_{x_{\sigma(\b)}}\, dQ\\
=&\int_{\cB^{\mathbb{N}}} \sum_{a\in \cA_{\overline{b}_1}}m_{\b}([a])F_{a\overline{b}_1}\left(\mu_{\sigma(\b)}\times \delta_{x_{\sigma(\b)}}\right)\, dQ\\
=&\sum_{\ob\in \cB}\sum_{a\in \cA_{\overline{b}}}\frac{p_{a\overline{b}}}{q_{\overline{b}}}\int_{\b:\overline{b}_1=\overline{b}} F_{a\overline{b}}\left(\mu_{\sigma(\b)}\times \delta_{x_{\sigma(\b)}}\right)\, dQ.
\end{align*}
In the final line we used that $$m_{\b}([a])=\frac{p_{a\overline{b}}}{q_{\overline{b}}}$$ if $\ob_{1}=\ob$. 
Now using the fact that $Q$ is a $\sigma$-invariant product measure and $F_{a\overline{b}}\left(\mu_{\sigma(\b)}\times \delta_{x_{\sigma(\b)}}\right)$ does not depend on the first digit of $\b,$ we have
\begin{align*}
\sum_{\ob\in \cB}\sum_{a\in \cA_{\overline{b}}}\frac{p_{a\overline{b}}}{q_{\overline{b}}}\int_{\b:\overline{b}_1=\overline{b}} F_{a\overline{b}}\left(\mu_{\sigma(\b)}\times \delta_{x_{\sigma(\b)}}\right)\, dQ
=&\sum_{\ob\in \cB}\sum_{a\in \cA_{\overline{b}}}\frac{p_{a\overline{b}}}{q_{\overline{b}}}\times q_{\overline{b}}\int_{\cB^{\N}} F_{a\overline{b}}\left(\mu_{\b}\times \delta_{x_{\b}}\right)\, dQ\\
=&\sum_{\ob\in \cB}\sum_{a\in \cA_{\overline{b}}}p_{a\overline{b}}\int_{\cB^{\N}} F_{a\overline{b}}\left(\mu_{\b}\times \delta_{x_{\b}}\right)\, dQ\\
=&\sum_{\overline{a}\in \cA}p_{\overline{a}}\int_{\cB^{\N}} F_{\overline{a}}\left(\mu_{\b}\times \delta_{x_{\b}}\right)\, dQ\\
=&\sum_{\overline{a}\in \cA}p_{\overline{a}}F_{\overline{a}}\tilde{\mu}.
\end{align*}
Thus, $\tilde{\mu}$ satisfies \eqref{Eq:Self-conformal equation} and our proof is complete.
\end{proof}
We finish this subsection by recording one consequence of Proposition \ref{Prop:Disintegration prop} with regards to the Fourier transform of $\mu$:
\begin{align}
	\label{Eq:Disintegrating the Fourier transform}
|\widehat{\mu}(\xi)|=\left|\int e^{2\pi i\langle\xi,x\rangle}\, d\mu\right|&=\left|\int_{B^{\N}} \int e^{2\pi i\langle\xi,x\rangle}d(\mu_{\b}\times \delta_{x_\b})\, dQ(\b)\right|\nonumber\\
	&=\left|\int_{\cB^{\mathbb{N}}}e^{2\pi i\langle(\xi_i)_{i=2}^{d},x_{\b}\rangle}\int e^{2\pi i \xi_{1}x}\, d\mu_{\b}\, dQ(\b)\right|\nonumber\\
	&\leq \int_{\cB^{\mathbb{N}}}|\widehat{\mu_{\b}}(\xi_1)|\, dQ(\b).
\end{align}
The significance of \eqref{Eq:Disintegrating the Fourier transform} is that it tells us that to prove \eqref{Eq:Non-conformal WTS} it is sufficient to show that outside of a set of small $Q$ measure, each $\mu_{\b}$ has polynomial Fourier decay.

\subsection{Bounding the Fourier transform by an average}

Let us now focus on $\widehat{\mu_{\b}}(\xi_1)$.
We show in Proposition~\ref{Prop:Non-conformal averaging} that $\widehat{\mu_{\b}}(\xi_1)$ is bounded above by an expression involving an average of Fourier transforms. Importantly Theorem \ref{thm:flattening} can be meaningfully applied to this average.

To this end, we let 
\begin{equation}
	\label{Eq:Define n}
	n(\xi_1)=\lfloor c\log |\xi_1|\rfloor
\end{equation}
 for some $c>0$ chosen to be sufficiently small that 
 \begin{equation}
 	\label{Eq: a words contraction}
\min_{\a\in \cA_{1}^{n(\xi_1)}}\min_{x\in [0,1]} |f_{\a}^{(1)}(x)|\geq \xi_{1}^{-1/3}.
 \end{equation}To each $\b\in \cB^{\N}$ we associate the set 
 $$\cA_{\xi_{1},\b}:=\set{\c\in \bigcup_{j=1}^{\infty}\prod_{l=1}^{j}\cA_{\overline{b}_{n(\xi_1)+l}}:\mrm{Diam}(f_{\c}([0,1]))<\xi_1^{-2/3},\, \mrm{Diam}(f_{\c^{-}}([0,1]))\geq\xi_1^{-2/3} }.$$ For each $\c\in \cA_{\xi_{1},\b}$ we pick $x_{\c}\in \pi_{\sigma^{n(\xi_1)}(\b)}([\c]).$ Given $\b\in \cB^{\N}$ and $\c\in \cA_{\xi_{1},\b}$ we let 
 $\mu_{\c,\sigma^{n(\xi_1)}(\b),\xi_{1}}$ be the pushforward of 
 $$\frac{\mu_{\sigma^{n(\xi_1)}(\b)}|_{\pi_{\sigma^{n(\xi_1)}(\b)}([\c])}}{\mu_{\sigma^{n(\xi_1)}(\b)}(\pi_{\sigma^{n(\xi_1)}(\b)}([\c]))}$$  under $x\to \xi_{1}^{2/3}x.$ The following statement bounds $|\widehat{\mu_{\b}}(\xi_1)|$ by an average of $\widehat{\mu_{\c,\sigma^{n(\xi_1)}(\b),\xi_{1}}}$ taken over all $\c\in \cA_{\xi_{1},\b}$ and a range of frequencies plus a polynomially small error.
 \begin{prop}
 	\label{Prop:Non-conformal averaging}
 	For any $\b\in \cB^{\N}$ we have
 	$$|\widehat{\mu_{\b}}(\xi_1)|\leq \sum_{\a\in \cA_{\overline{b}_1}\times \cdots \times  \cA_{\overline{b}_{n(\xi_1)}}}m_{\b}([\a])\sum_{\c\in \cA_{\xi_1,\b}}m_{\sigma^{n(\xi_1)}(\b)}([\c])|\widehat{\mu_{\c,\sigma^{n(\xi_1)}(\b),\xi_1}}(\xi_{1}^{1/3} (f_{\a}^{(1)})'(x_{\c}))|+\cO(|\xi_1|^{-1/3})$$
 \end{prop}
 \begin{proof}
Fix $\b\in \cB^{\N}$. Using Lemma \ref{Lemma:Dynamical self-conformality} we have
$$|\widehat{\mu_{\b}}(\xi_1)|=\left|\int e^{2\pi i\xi_1x}\, d\mu_{\b}\right|=\left| \sum_{\a\in \cA_{\overline{b}_1}\times \cdots \times \cA_{\overline{b}_{n(\xi_1)}}}m_{\b}([\a])\int e^{2\pi i\xi_{1}f_{\a}^{(1)}(x)}\, d\mu_{\sigma^{n(\xi_1)}(\b)}\right|$$
We now split the latter integrals over cylinder sets determined by element of $\cA_{\xi_1,\b}$:
\begin{align*}
	&\left| \sum_{\a\in \cA_{\overline{b}_1}\times \cdots \times \cA_{\overline{b}_{n(\xi_1)}}}m_{\b}([\a])\int e^{2\pi i\xi_{1}f_{\a}^{(1)}(x)}\, d\mu_{\sigma^{n(\xi_1)}(\b)}\right|\\
	=&\left| \sum_{\a\in \cA_{\overline{b}_1}\times \cdots \times \cA_{\overline{b}_{n(\xi_1)}}}m_{\b}([\a])\sum_{\c\in \cA_{\xi_1,\b}}\int_{\pi_{\sigma^{n(\xi_1)}(\b)}([\c])} e^{2\pi i\xi_{1}f_{\a}^{(1)}(x)}\, d\mu_{\sigma^{n(\xi_1)}(\b)}\right|
\end{align*} 
For each $\a\in \cA_{\overline{b}_1}\times \cdots \times \cA_{\overline{b}_{n(\xi_1)}}$ we can linearize $f_{\a}^{(1)}(x)$ around $x_{\c}$ and use Lemma \ref{Lemma:Bounded distortions} to derive the bound
\begin{equation}
	\label{Eq:Linearising error}
f_{\a}^{(1)}(x)=f_{\a}^{(1)}(x_{\c})+(f_{\a}^{(1)})'(x_{\c})(x-x_{\c})+\cO(|\xi_{1}|^{-4/3})
\end{equation}
for all $x\in \pi_{\sigma^{n(\xi_1)}(\b)}([\c]).$ Here we have used the definition of $\cA_{\xi_1,\b}.$
Applying \eqref{Eq:Linearising error} in the above we have
\begin{align*}
&\left| \sum_{\a\in \cA_{\overline{b}_1}\times \cdots \times \cA_{\overline{b}_{n(\xi_1)}}}m_{\b}([\a])\sum_{\c\in \cA_{\xi_1,\b}}\int_{\pi_{\sigma^{n(\xi_1)}(\b)}([\c])} e^{2\pi i\xi_{1}f_{\a}^{(1)}(x)}\, d\mu_{\sigma^{n(\xi_1)}(\b)}\right|\\
\leq &\sum_{\a\in \cA_{\overline{b}_1}\times \cdots \times \cA_{\overline{b}_{n(\xi_1)}}}m_{\b}([\a])\sum_{\c\in \cA_{\xi_1,\b}}\left|\int_{\pi_{\sigma^{n(\xi_1)}(\b)}([\c])} e^{2\pi i\xi_{1}(f_{\a}^{(1)})'(x_{\c})x}\, d\mu_{\sigma^{n(\xi_1)}(\b)}\right|+\cO(|\xi_1|^{-1/3})\\
=&\sum_{\a\in \cA_{\overline{b}_1}\times \cdots \times  \cA_{\overline{b}_{n(\xi_1)}}}m_{\b}([\a])\sum_{\c\in \cA_{\xi_1,\b}}m_{\sigma^{n(\xi_1)}(\b)}([\c])\left|\int e^{2\pi i\xi_{1}^{1/3}(f_{\a}^{(1)})'(x_{\c})x}\, d\mu_{\c,\sigma^{n(\xi_1)}(\b),\xi}\right|+\cO(|\xi_1|^{-1/3})\\
=&\sum_{\a\in \cA_{\overline{b}_1}\times \cdots \times  \cA_{\overline{b}_{n(\xi_1)}}}m_{\b}([\a])\sum_{\c\in \cA_{\xi_1,\b}}m_{\sigma^{n(\xi_1)}(\b)}([\c])\left|\widehat{\mu_{\c,\sigma^{n(\xi_1)}(\b),\xi_1}}(\xi_1 ^{1/3}(f_{\a}^{(1)})'(x_{\c}))\right|+\cO(|\xi_1|^{-1/3}).
\end{align*}
This completes our proof.
\end{proof}

\subsection{Affine non-concentration of the averaged measures}
In light of the above proposition,
to successfully apply Theorem~\ref{thm:flattening} to prove Theorem \ref{Theorem:Non-conformal}, we need to show that each $\mu_{\c,\sigma^{n(\xi_1)}(\b),\xi}$ is uniformly affinely non-concentrated and the non-concentration parameters do not depend upon our choice of measure. 
This fact is established in the following proposition.
\begin{prop}
	\label{Prop:Non-conformal affinely non-concentrated}
There exists $C,\alpha>0$ such that for any $\b\in \cB^{\N},$ $\xi_{1}\in\mathbb{R}$, $\c\in A_{\xi_1,\b},$ $x\in \supp(\mu_{\c,\sigma^{n(\xi_1)}(\b),\xi_1}),$ $y\in \R,$ $r\in (0,1]$ and $\epsilon>0$ we have 
$$\mu_{\c,\sigma^{n(\xi_1)}(\b),\xi_1}( B(x,r)\cap B(y,\epsilon r))\leq C\epsilon^{\alpha}\mu_{\c,\sigma^{n(\xi_1)}(\b),\xi_1}(B(x,r)).$$
\end{prop}
\begin{proof}
Let us fix $\b\in \cB^{\N},$ $\xi_{1}\in\mathbb{R}$ and $\c\in A_{\xi_1,\b}.$ Let $x\in \supp(\mu_{\c,\sigma^{n(\xi_1)}(\b),\xi_1})$, $r\in(0,1]$ and $\epsilon>0$. Then, by definition, for any $y\in \R$ we have 
\begin{align}
\label{eq:First step in measure bound}
	&\mu_{\c,\sigma^{n(\xi_1)}(\b),\xi_1}(B(x,r)\cap B(y,\epsilon r))
    \nonumber\\
 =&\frac{\mu_{\sigma^{n(\xi_1)}(\b)}|_{\pi_{\sigma^{n(\xi_1)}(\b)}([\c])}(B(x\cdot \xi_1^{-2/3},r\cdot \xi_1^{-2/3})\cap B(y\cdot \xi_1^{-2/3},\epsilon r\cdot \xi_1^{-2/3}))}{\mu_{\sigma^{n(\xi_1)}(\b)}(\pi_{\sigma^{n(\xi_1)}(b)}([\c]))}\nonumber\\
	\leq& \frac{\mu_{\sigma^{n(\xi_1)}(\b)}(B(x\cdot \xi_1^{-2/3},r\cdot \xi^{-2/3})\cap B(y\cdot \xi_1^{-2/3},\epsilon r\cdot \xi_1^{-2/3}))}{\mu_{\sigma^{n(\xi_1)}(\b)}(\pi_{\sigma^{n(\xi_1)}(b)}([\c]))}\nonumber\\
	=&\frac{\mu_{\sigma^{n(\xi_1)}(\b)}(B(x',r\cdot \xi_{1}^{-2/3})\cap B(y',\epsilon r\cdot \xi_1^{-2/3}))}{\mu_{\sigma^{n(\xi_1)}(\b)}(\pi_{\sigma^{n(\xi_1)}(b)}([\c]))}.
\end{align} Where in the last line we have adopted the notation: $x'=x\cdot \xi_{1}^{-2/3}\in \pi_{\sigma^{n(\xi_1)}(\b)}([\c])$  and $y'=y\cdot \xi_{1}^{-2/3}.$ We begin our proof by analysing the sets appearing in the numerator in \eqref{eq:First step in measure bound}. Without loss of generality we can assume that $y'\in X_{\sigma^{n(\xi_1)}(\b)}.$

It is useful at this point to recall that the IFS $\set{f_{a}^{(1)}}_{a\in \cA}$ satisfies the strong separation condition. Therefore, if $X_{1}\subset \R$ denotes the attractor for this IFS, we have $f_{a}^{(1)}(X_1)\cap f_{a'}^{(1)}(X_1)=\emptyset$ for distinct $a,a'\in \cA_{1}$. Therefore $$\min_{a\neq a'}d\left(f_{a}^{(1)}(X_1),f_{a'}^{(1)}(X_1)\right)>0.$$ It is a consequence of Lemma \ref{Lemma:Derivative and Diameter} and this property that for $\a,\a'\in \cA_{1}^{*}$ we have 
\begin{equation}
	\label{eq:IFS separated cylinders}
	d\left(f_{\a}^{(1)}(X_{1}),f_{\a'}^{(1)}(X_1)\right)\asymp Diam\left(f_{|\a\wedge\a'|}([0,1])\right).
\end{equation}  Since $\pi_{\sigma^{n(\xi_1)}(\b)}([\bd])\subset f_{\bd}^{(1)}(X_{1})$ for all $\bd\in \cup_{j=1}^{\infty}\prod_{i=1}^{j}\cA_{\overline{b}_{n(\xi_1)+i}},$ \eqref{eq:IFS separated cylinders} implies 
\begin{equation}
	\label{eq:separated cylinders}
	d(\pi_{\sigma^{n(\xi_1)}(\b)}([\bd]),\pi_{\sigma^{n(\xi_1)}(\b)}([\bd']))\asymp \mrm{Diam}(f_{|\bd\wedge\bd'|}([0,1]))
\end{equation} for distinct $\bd,\bd'\in \cup_{j=1}^{\infty}\prod_{i=1}^{j}\cA_{\overline{b}_{n(\xi_1)+i}}.$

We let $\bd=(d_i)\in \prod_{i=1}^{\infty}\cA_{\overline{b}_{n(\xi_1)+i}}$ be the unique sequence satisfying $\pi_{\sigma^{n(\xi_1)}(\b)}(\bd)=x'.$ Since $x'\in \pi_{\sigma^{n(\xi_1)}(\b)}([\c])$ the sequence $\bd$ must begin with $\c$. Let $m\in \N$ be minimal such that $$\pi_{\sigma^{n(\xi_1)}(\b)}([d_1,\ldots,d_m])\subset B(x',r\cdot\xi_{1}^{-2/3}).$$ It follows from \eqref{eq:separated cylinders} that a bounded number of $(d_1',\ldots,d_m')\in \prod_{i=1}^{m}\cA_{\overline{b}_{n(\xi_1)+i}}$ satisfy 
\begin{equation}
	\label{eq:Intersecting words}
	\pi_{\sigma^{n(\xi_1)}(\b)}([d_1',\ldots,d_m'])\cap B(x',r\cdot\xi_{1}^{-2/3})\neq\emptyset.
\end{equation} It also follows from \eqref{eq:separated cylinders} that there exists $l\in\N$ depending only on our IFS, such that if $\bd'$ satisfies \eqref{eq:Intersecting words} then $d_i=d_i'$ for all $1\leq i\leq m-l$. Therefore for such a $(d_1',\ldots,d_m')$ we have $$m_{\sigma^{n(\xi_1)}(\b)}([d_1,\ldots,d_m])\asymp m_{\sigma^{n(\xi_1)}(\b)}([d_1',\ldots,d_m']).
$$ Combining these observations we have
\begin{equation}
    \label{Eq:R ball measure}
\mu_{\sigma^{n(\xi_1)}(\b)}(B(x',r\cdot \xi_{1}^{-2/3})\asymp m_{\sigma^{n(\xi_1)}(\b)}([d_1,\ldots,d_m]).
\end{equation}

Let us now consider $ B(y',\epsilon r\cdot \xi_{1}^{-2/3}).$ Let $\be=(e_i)\in \Sigma_{\sigma^{n(\xi_1)}(\b)}$ be the unique sequence satisfying $\pi_{\sigma^{n(\xi_1)}(\b)}(\be)=y'.$ Let $m'\geq m$ be minimal such that $$\pi_{\sigma^{n(\xi_1)}(\b)}([e_1,\ldots,e_{m'}])\subset B(y',\epsilon r\xi_{1}^{-2/3}).$$ By an analogous argument to that given above, we must have 
\begin{equation}
	\label{Eq:epsilon r bound}
	\mu_{\sigma^{n(\xi_1)}(\b)}(B(y',\epsilon r\xi_{1}^{-2/3}))\asymp m_{\sigma^{n(\xi_1)}(\b)}([e_1,\ldots,e_{m'}]).
\end{equation}
Using the fact that our IFS is uniformly contracting, it can be shown that $m'-m\geq \lfloor d\log \epsilon^{-1}\rfloor$ for some $d>0$ depending only on our IFS. Using this inequality and recalling the definition of $\gamma$ given by \eqref{Eq:Exponential decay for cylinders}, we have 
\begin{align}
	\label{Eq:epsilon decay}
	\mu_{\sigma^{n(\xi_1)}(\b)}(B(y',\epsilon r\xi_{1}^{-2/3}))&\ll m_{\sigma^{n(\xi_1)}(\b)}([e_1,\ldots,e_{m'}])\nonumber\\
	&\ll m_{\sigma^{n(\xi_1)}(\b)}([e_1,\ldots,e_{m}])\epsilon^{-d\log \gamma}\nonumber\\
	&\ll m_{\sigma^{n(\xi_1)}(\b)}([d_1,\ldots,d_{m}])\epsilon^{-d\log \gamma}.
	\end{align}
In the last line we used that $(e_1,\ldots,e_m)$ must satisfy \eqref{eq:Intersecting words} therefore $e_i=d_i$ for $1\leq i\leq m-l$ and so $m_{\sigma^{n(\xi_1)}(\b)}([e_1,\ldots,e_{m}])\asymp m_{\sigma^{n(\xi_1)}(\b)}([d_1,\ldots,d_{m}])$.
Summarizing, it follows from \eqref{eq:First step in measure bound}, \eqref{Eq:R ball measure}, \eqref{Eq:epsilon r bound} and \eqref{Eq:epsilon decay} that
\begin{align*}
		\mu_{\c,\sigma^{n(\xi_1)}(\b),\xi}(B(x,r)\cap B(y,\epsilon r))&\leq  \frac{\mu_{\sigma^{n(\xi_1)}(\b)}(B(x',r\cdot \xi_{1}^{-2/3})\cap B(y',\epsilon r\cdot \xi_1^{-2/3}))}{\mu_{\sigma^{n(\xi_1)}(\b)}(\pi_{\sigma^{n(\xi_1)}(b)}([\c]))}\\
		&\ll \frac{\mu_{\sigma^{n(\xi_1)}(\b)}( B(y',\epsilon r\cdot \xi_{1}^{-2/3}))}{\mu_{\sigma^{n(\xi_1)}(\b)}(\pi_{\sigma^{n(\xi_1)}(b)}([\c]))}\\
		&\ll \epsilon^{-d\log \gamma}\frac{\mu_{\sigma^{n(\xi_1)}(\b)}(B(x',r\cdot \xi_{1}^{-2/3}))}{\mu_{\sigma^{n(\xi_1)}(\b)}(\pi_{\sigma^{n(\xi_1)}(b)}([\c]))}.
\end{align*}
Taking $\alpha=-d\log \g,$ we see that to complete our proof it remains to show that 
\begin{equation*}
	\label{Eq:Final part}
	\frac{\mu_{\sigma^{n(\xi_1)}(\b)}(B(x',r\cdot \xi_1^{-2/3}))}{\mu_{\sigma^{n(\xi_1)}(\b)}(\pi_{\sigma^{n(\xi_1)}(b)}([\c]))}\ll \mu_{\c,\sigma^{n(\xi_1)}(\b),\xi}(B(x,r)).
\end{equation*} This is equivalent to  
\begin{equation}
	\label{Eq:Final part 2}
	\mu_{\sigma^{n(\xi_1)}(\b)}(B(x',r\cdot \xi_1^{-2/3}))\ll \mu_{\sigma^{n(\xi_1)}(\b)}|_{\pi_{\sigma^{n(\xi_1)}(b)}([\c])}(B(x',r\cdot\xi_1^{-2/3})).
\end{equation} It is useful to recall at this point that by the definition of $A_{\xi_1,\b}$ we know that 
\begin{equation}
\label{Eq:Random diameter bound}
\mrm{Diam}(\pi_{\sigma^{n(\xi_1)}(b)}([\c]))\asymp \xi_{1}^{-2/3}.
\end{equation}Therefore, it follows from \eqref{eq:separated cylinders} that there exists $r_{0}>0$ depending only on our IFS such that if $r<r_{0}$ then $$\mu_{\sigma^{n(\xi_1)}(\b)}|_{\pi_{\sigma^{n(\xi_1)}(b)}([\c])}(B(x',r\cdot\xi_1^{-2/3}))=\mu_{\sigma^{n(\xi_1)}(\b)}(B(x',r\cdot\xi_1^{-2/3})).$$ Therefore \eqref{Eq:Final part 2} holds for $r\leq r_0$. For $r\geq r_{0},$ one can appeal to $\eqref{Eq:Random diameter bound}$ and the fact $\bd$ begins with $\c$ to show that  
$$m_{\sigma^{n(\xi_1)}(\b)}([d_1,\ldots,d_m])\asymp m_{\sigma^{n(\xi_1)}(\b)}([\c]).$$ It is also straightforward to show that $$\mu_{\sigma^{n(\xi_1)}(\b)}|_{\pi_{\sigma^{n(\xi_1)}(b)}([\c])}(B(x',r\cdot\xi_1^{-2/3}))\asymp  m_{\sigma^{n(\xi_1)}(\b)}([\c]).$$ Thus \eqref{Eq:Final part 2} follows for $r\geq r_{0}$ by these observations and \eqref{Eq:R ball measure}. This completes our proof.
\end{proof}

\subsection{Spectral gap and non-concentration of frequencies}
The final step in the proof of Theorem~\ref{Theorem:Non-conformal} is to show that the set of frequencies $$\set{\xi_{1}^{1/3}(f_{\a}^{(1)})'(x_{\c}): \a\in \cA_{\overline{b}_1}\times \cdots \times  \cA_{\overline{b}_{n(\xi_1)}}}$$
appearing in Proposition~\ref{Prop:Non-conformal averaging} satisfy a non-concentration property. 
This is established in Proposition~\ref{Prop:Applying the random spectral gap} below.
The key ingredient is a spectral gap for certain random transfer operators (Proposition~\ref{Prop:Spectral gap on average}).

More precisely, to each $\b\in \cB^{\N}$ and $b\in\mathbb{R}$, we associate a transfer operator that acts on $C^{1}(U)$:
$$\cL_{ib}^{(\b)}h(x)=\sum_{a\in \cA_{\overline{b}_1}}m_{\b}([a])|(f_{a}^{(1)})'(x)|^{ib}h(f_{a}(x)).$$ Iterating this formula we have  
\begin{equation}
	\label{Eq:Iterated random transfer operators}(\cL_{ib}^{(\sigma^{n-1}(\b))}\circ \cdots \circ \cL_{ib}^{(\b)})h(x)=\sum_{\a\in \cA_{\overline{b}_{1}}\times \cdots \times \cA_{\overline{b}_{n} }}m_{\b}([\a])|(f_{\a}^{(1)})'(x)|^{ib}h(f_{\a}^{(1)}(x))
\end{equation}
 for any $\b\in \cB^{\N},$ $h\in C^{1}(U)$, $n\in\mathbb{N}$ and $x\in U$.

The following statement asserts that a random operator $\cL_{ib}^{(\sigma^{n-1}(\b))}\circ \cdots \circ \cL_{ib}^{(\b)}$ has a spectral gap for all $\b\in \cB^{\N}$ outside of a small measure set.
It is in the proof of this proposition where we use the third assumption in Theorem \ref{Theorem:Non-conformal}.

\begin{prop}
	\label{Prop:Spectral gap on average}
 There exists $\rho_{1}\in(0,1)$ and $\Theta>0$ such that the following statements are true: 
\begin{itemize}
	\item For all $n\in\mathbb{N}$ there exists $\Omega_{n}\in \cB^{\N}$ such that $$\|\cL_{ib}^{(\sigma^{n-1}(\b))}\circ \cdots \circ \cL_{ib}^{(\b)}\|_{\infty}\ll \rho_{1}^{n}|b|^{1/2}\|h\|_{b}$$ for all $\overline{\b}\in \Omega_{n},$ $h\in C^{1}(U)$ and $|b|>\Theta$.
	\item There exists $\delta>0$ such that $\Omega_n$ satisfies $Q(\Omega_{n}^{c})\ll e^{-\delta n}$ for all $n\in \N$. 
\end{itemize}  
\end{prop}

\begin{proof}[Sketch of proof of Proposition~\ref{Prop:Spectral gap on average}]

We do not include a full proof of Proposition  \ref{Prop:Spectral gap on average} because it is essentially contained in \cite{BakerSahlsten}. See Theorem 3.2 and (3.9) from this paper. We give a rough outline of the argument for convenience of the reader. 
The third assumption in Theorem \ref{Theorem:Non-conformal} implies that the IFS $\set{f_{a}^{(1)}:a\in \cA_{\overline{b}^{*}}}$ must satisfy the UNI property for some $\overline{b}^{*}\in B$.
A large deviation argument tells us that outside of an exponentially small measure set, for a typical $\b\in \cB^{\N}$, a significant proportion of the operators appearing in the composition $\cL_{ib}^{(\sigma^{n-1}(\b))}\circ \cdots \circ \cL_{ib}^{(\b)}$ will equal $\cL_{ib}^{(\overline{b}^{*})}.$ The existence of the desired spectral gap for the single operator $\cL_{ib}^{(\overline{b}^{*})}$ goes back to Naud \cite{Naud-Cantor} and Stoyanov \cite{Stoyanov}. By adapting Naud's argument, the first author and third author showed in \cite{BakerSahlsten} that if $\cL_{ib}^{(\overline{b}^{*})}$ appeared sufficiently many times within $\cL_{ib}^{(\sigma^{n-1}(\b))}\circ \cdots \circ \cL_{ib}^{(\b)}$ then we have the desired spectral gap. Combining this with our large deviation bound implies Proposition \ref{Prop:Spectral gap on average}.

\end{proof}
\begin{remark}
    We remark that in \cite{BakerSahlsten} the separation assumption for our IFS is that $f_{a}^{(1)}([0,1])\cap f_{a'}^{(1)}([0,1])=\emptyset$ for distinct $a,a'\in \cA_{1}$. This is a stronger assumption than the strong separation condition. Nevertheless it is possible to adapt the arguments in \cite{BakerSahlsten} to weaken this assumption to the strong separation condition. 
\end{remark}

\begin{remark}
    We note that Proposition \ref{Prop:Spectral gap on average} gives a spectral gap in terms of the $\|\cdot\|_{\infty}$ norm instead of the usual $\|\cdot\|_{b}$ norm. Despite the spectral gap being phrased in terms of this norm, it causes no issues in our proof of Theorem \ref{Theorem:Non-conformal}. 
\end{remark}

Equipped with Proposition \ref{Prop:Spectral gap on average}, we can now prove our non-concentration statement for frequencies. 
\begin{prop}
	\label{Prop:Applying the random spectral gap}
	There exists $\kappa>0$ such that for any $\b\in \Omega_{n(\xi_1)},$ $\c\in \cA_{\xi_1,\b}$ and $n\in \N$ satisfying $|\xi_{1}|^{1/6}\leq n \leq |\xi_{1}|^{1/3}$ we have 
	$$\sum_{\a\in \cA_{\overline{b}_1}\times \cdots \times  \cA_{\overline{b}_{n(\xi_1)}}}m_{\b}([\a])\chi_{[n,n+1]}(\xi_1^{1/3}|(f_{\a}^{(1)})'(x_{\c})|)=\cO(|\xi_{1}|^{-\kappa}).$$
\end{prop}
\begin{proof}
	The proof of Proposition \ref{Prop:Applying the random spectral gap} is analogous to Proposition \ref{Prop:Self-conformal non-concentration}. Thus, we only explain how the random transfer operators appear in our analysis. This is the only point where the proof differs.
	
	As in the proof of Proposition \ref{Prop:Self-conformal non-concentration}, we can replace $\chi_{[n,n+1]}$ with a smooth function $h_{n,\xi_1}$ that satisfies $$\chi_{[n,n+1]}(\xi_1^{1/3}|(f_{\a}^{(1)})'(x_{\c})|)\leq h_{n,\xi_1}(\log |(f_{\a}^{(1)})'(x_{\c})|)$$ and for which we have useful bounds on the $L^{1}$ norm of its second derivative. Applying the above and the Fourier inversion formula, we have the following:
	\begin{align*}
		&\sum_{\a\in \cA_{\overline{b}_1}\times \cdots \times  \cA_{\overline{b}_{n(\xi_1)}}}
        m_{\b}([\a]) 
        \chi_{[n,n+1]}(\xi^{1/3}|
        (f_{\a}^{(1)})'(x_{\c})|)\\
		&\leq \sum_{\a\in \cA_{\overline{b}_1}\times \cdots \times  \cA_{\overline{b}_{n(\xi_1)}}}m_{\b}([\a])h_{n,\xi}(\log |(f_{\a}^{(1)})'(x_{\c})|)\\
		&\leq  \sum_{\a\in \cA_{\overline{b}_1}\times \cdots \times  \cA_{\overline{b}_{n(\xi_1)}}}m_{\b}([\a])\int \widehat{h_{n,\xi}}(\eta)e^{-2\pi i\eta \log |(f_{\a}^{(1)})'(x_{\c})|}\, d\eta\\
	   &=
        \sum_{\a\in \cA_{\overline{b}_1}\times \cdots \times  \cA_{\overline{b}_{n(\xi_1)}}}m_{\b}([\a])\int \widehat{h_{n,\xi}}(\eta)|(f_{\a}^{(1)})'(x_{\c})|^{-2\pi i\eta }\, d\eta\\
		&=
        \int \widehat{h_{n,\xi}}(\eta)\sum_{\a\in \cA_{\overline{b}_1}\times \cdots \times  \cA_{\overline{b}_{n(\xi_1)}}}m_{\b}([\a])|(f_{\a}^{(1)})'(x_{\c})|^{-2\pi i\eta }\, d\eta\\
	    &=
        \int \widehat{h_{n,\xi}}(\eta)(\cL_{-2\pi i\eta}^{(\sigma^{n(\xi_1)-1}(\b))}\circ \cdots \circ \cL_{-2\pi i\eta}^{(\b)})(1)(x_{\c})\, d\eta,
	\end{align*}
where, in the last line, we used \eqref{Eq:Iterated random transfer operators}.
The rest of our proof is identical to the proof of Proposition \ref{Prop:Self-conformal non-concentration}. We split out integral into two parts, one part that can be controlled using the bound $|\widehat{h_{n,\xi}}(\eta)|\leq |\widehat{h_{n,\xi}}(0)|$, and one part on which we use our assumption $\b\in \Omega_{n(\xi_1)}$ and the spectral gap guaranteed by Proposition \ref{Prop:Spectral gap on average}.
\end{proof}

\subsection{Conclusion of the proof of Theorem~\ref{Theorem:Non-conformal}}

Recall that, by our assumption, we have $\max_{1\leq i\leq d}|\xi_i|=|\xi_1|$ and, hence. it suffices to show that \eqref{Eq:xi_1 decay} holds. By \eqref{Eq:Disintegrating the Fourier transform} and Proposition \ref{Prop:Non-conformal averaging}, we have 
\begin{align*}
    |\widehat{\mu}(\xi)| &\leq \cO(|\xi_1|^{-1/3})
    \nonumber\\
    &+\int \sum_{\a\in \cA_{\overline{b}_1}\times \cdots \times  \cA_{\overline{b}_{n(\xi_1)}}}m_{\b}([\a])\sum_{\c\in \cA_{\xi_1,\b}}m_{\sigma^{n(\xi_1)}(\b)}([\c])|\widehat{\mu_{\c,\sigma^{n(\xi_1)}(\b),\xi_1}}(\xi_{1}^{1/3} (f_{\a}^{(1)})'(x_{\c}))|\, dQ(\b).    
\end{align*}
So to complete our proof it is sufficient to show that 
\begin{equation}
	\label{Eq:Non-conformal WTS}
	\int \sum_{\a\in \cA_{\overline{b}_1}\times \cdots \times  \cA_{\overline{b}_{n(\xi_1)}}}m_{\b}([\a])\sum_{\c\in \cA_{\xi_1,\b}}m_{\sigma^{n(\xi_1)}(\b)}([\c])|\widehat{\mu_{\c,\sigma^{n(\xi_1)}(\b),\xi_1}}(\xi_{1}^{1/3} (f_{\a}^{(1)})'(x_{\c}))|\, dQ(\b)\ll |\xi_{1}|^{-\eta}
\end{equation} for some $\eta>0$. We begin by splitting our outer integral. Let $\Omega_{n(\xi)}$ be as in Proposition \ref{Prop:Spectral gap on average}. It follows from this proposition that 
\begin{align}
	\label{Eq:Reducing to Omega}
	\int &\sum_{\a\in \cA_{\overline{b}_1}\times \cdots \times  \cA_{\overline{b}_{n(\xi_1)}}}m_{\b}([\a])\sum_{\c\in \cA_{\xi_1,\b}}m_{\sigma^{n(\xi_1)}(\b)}([\c])|\widehat{\mu_{\c,\sigma^{n(\xi_1)}(\b),\xi_1}}(\xi_{1}^{1/3} (f_{\a}^{(1)})'(x_{\c}))|\, dQ(\b)\nonumber\\
	&\ll 
    |\xi_{1}|^{-\eta_1}
    \nonumber\\
    &+\int_{\Omega_{n(\xi_1)}} \sum_{\a\in \cA_{\overline{b}_1}\times \cdots \times  \cA_{\overline{b}_{n(\xi_1)}}}m_{\b}([\a])\sum_{\c\in \cA_{\xi_1,\b}}m_{\sigma^{n(\xi_1)}(\b)}([\c])|\widehat{\mu_{\c,\sigma^{n(\xi_1)}(\b),\xi_1}}(\xi_{1}^{1/3} (f_{\a}^{(1)})'(x_{\c}))|\, dQ(\b) 
    \nonumber\\
	&= |\xi_{1}|^{-\eta_1}
    \nonumber\\
    &+\int_{\Omega_{n(\xi_1)}}\sum_{\c\in \cA_{\xi_1,\b}}m_{\sigma^{n(\xi_1)}(\b)}([\c])\sum_{\a\in \cA_{\overline{b}_1}\times \cdots \times  \cA_{\overline{b}_{n(\xi_1)}}}m_{\b}([\a])|\widehat{\mu_{\c,\sigma^{n(\xi_1)}(\b),\xi_1}}(\xi_{1}^{1/3} (f_{\a}^{(1)})'(x_{\c}))|\, dQ(\b),
\end{align}
for some $\eta_{1}>0$. For each $\c\in \cA_{\xi_1,\b},$ $\b\in \cB^{\N}$ and $\tau>0$, we let $$\mrm{Bad}_{\c,\b,\tau}:=\set{n\in\mathbb{Z}:|n|\leq |\xi_{1}|^{1/3}\textrm{ and }\exists \zeta\in [n,n+1] \textrm{ satisfying }|\widehat{\mu_{\c,\sigma^{n(\xi_1)}(\b),\xi_1}}(\zeta)|\geq |\xi_1|^{-\tau}}.$$
By Proposition \ref{Prop:Non-conformal affinely non-concentrated}, we know that each $\mu_{\c,\sigma^{n(\xi_1)}(\b),\xi_1}$ is uniformly affinely non-concentrated and that the non-concentration parameters can be chosen independently of the measure. 
As such, we can apply Theorem \ref{thm:flattening} to assert that there exists $\tau>0$ such that 
\begin{equation}
	\label{Eq: Non-conformal number of bad frequency patches}
\#	\mrm{Bad}_{\c,\b,\tau}=\cO_{\kappa}(|\xi_{1}|^{\kappa/2}),
\end{equation} where $\kappa$ is as in Proposition \ref{Prop:Applying the random spectral gap}. Crucially the underlying constants in \eqref{Eq: Non-conformal number of bad frequency patches} do not depend upon $\c,$  $\b$ or $\xi$. Using Proposition \ref{Prop:Applying the random spectral gap} and \eqref{Eq: a words contraction}, we see that for any $\c\in \cA_{\xi_1,\b}$ and $\b\in \Omega_{n(\xi_1)}$ we have
\begin{align}
	\label{Eq:Almost finished}
&\sum_{\a\in \cA_{\overline{b}_1}\times \cdots \times  \cA_{\overline{b}_{n(\xi_1)}}}m_{\b}([\a])|\widehat{\mu_{\c,\sigma^{n(\xi_1)}(\b),\xi_1}}(\xi_{1}^{1/3} (f_{\a}^{(1)})'(x_{\c}))|\nonumber\\
\leq &\sum_{n\in \mrm{Bad}_{\c,\b,\tau}}\sum_{\a\in \cA_{\overline{b}_1}\times \cdots \times  \cA_{\overline{b}_{n(\xi_1)}}}m_{\b}([\a])\chi_{[n,n+1]}(\xi_{1}^{1/3}(f_{\a}^{(1)})'(x_{\c}))\nonumber\\
+&\sum_{\a\in \cA_{\overline{b}_1}\times \cdots \times  \cA_{\overline{b}_{n(\xi_1)}}}m_{\b}([\a])\chi_{(\cup_{n\in \mrm{Bad}_{\c,\b,\tau}}[n,n+1])^{c}}(\xi_{1}^{1/3}(f_{\a}^{(1)})'(x_{\c}))|\widehat{\mu_{\c,\sigma^{n(\xi_1)}(\b),\xi_1}}(\xi_{1}^{1/3} (f_{\a}^{(1)})'(x_{\c}))|\nonumber\\
\leq &\sum_{n\in \mrm{Bad}_{\c,\b,\tau}}|\xi_1|^{-\kappa}+\sum_{\a\in \cA_{\overline{b}_1}\times \cdots \times  \cA_{\overline{b}_{n(\xi_1)}}}m_{\b}([\a])|\xi_{1}|^{-\tau}\nonumber\\
\ll &|\xi_{1}|^{-\min\set{\kappa/2,\tau}}. 
\end{align}
In the last line, we used \eqref{Eq: Non-conformal number of bad frequency patches}. Substituting \eqref{Eq:Almost finished} into \eqref{Eq:Reducing to Omega}, we have 
\begin{align*}
	&\int_{\Omega_{n(\xi_1)}}\sum_{\c\in \cA_{\xi_1,\b}}m_{\sigma^{n(\xi_1)}(\b)}([\c])\sum_{\a\in \cA_{\overline{b}_1}\times \cdots \times  \cA_{\overline{b}_{n(\xi_1)}}}m_{\b}([\a])|\widehat{\mu_{\c,\sigma^{n(\xi_1)}(\b),\xi}}(\xi_{1}^{1/3} (f_{\a}^{(1)})'(x_{\c}))|\, dQ+|\xi_{1}|^{-\eta_1}\\
	\ll &\int_{\Omega_{n(\xi_1)}}|\xi_{1}|^{-\min\set{\kappa/2,\tau}}\sum_{\c\in \cA_{\xi_1,\b}}m_{\sigma^{n(\xi_1)}(\b)}([\c])\, dQ +|\xi_{1}|^{-\eta_1}\\
		\leq &|\xi_{1}|^{-\min\set{\kappa/2,\tau,\eta_{1}}}.
\end{align*}
Thus, \eqref{Eq:Non-conformal WTS} holds and our proof is complete.

\appendix
\section{Higher dimensional Davenport-Erd\Horig{o}s-LeVeque criterion}\label{appendix:DELhigh}

The purpose of this appendix is to give the proof of the following:

\begin{thm}
\label{thm:DELhigh}
Let $\mu$ be a probability measure on $\R^d$ with polylogarithmic Fourier decay. Then, $\mu$ almost every $x$ is $A$-normal for any expanding integer valued matrix $A$.
\end{thm}

The strategy of proof for Theorem \ref{thm:DELhigh} is similar to that implemented by Queff\'elec and Ramar\'e \cite{QueRam} for the one dimensional version of the DEL criterion. The third author and Jonathan Fraser wrote a proof of Theorem \ref{thm:DELhigh} in an unpublished note. We thank Jonathan for his permission to include the argument here.

To prove Theorem \ref{thm:DELhigh}, we need the following Lemma.

\begin{lem}[Lemma 7.3 \cite{QueRam}]\label{lma:sum}
Suppose $(r_N) \subset (0,\infty)$ is a sequence of reals such that $\sum_{N = 1}^\infty \frac{r_N}{N} < \infty$. Then there is a subsequence $N_j \to \infty$, $j \in \N$, such that $\sum_{j = 1}^\infty r_{N_j} < \infty$ and $\lim_{j \to \infty} \frac{N_{j+1}}{N_j} =1.$
\end{lem}

\begin{proof}[Proof of Theorem \ref{thm:DELhigh}]
Let $\mu$ be a measure on the $d$-torus $\T^d$ for which there exists some $\alpha >0$ and $C > 0$ such that the Fourier transform satisfies $|\widehat{\mu}(\xi)| \leq C|\log |\xi||^{-\alpha}$ for all $\xi \in \R^2$ with $|\xi | >1$.  Also, let $T(x) = Ax \mod 1$, $x \in \T^d$ for some $A \in \Z^{d \times d}$. Fix $\k \in \Z^d\setminus\{0\}$ and for $N \in\N$ write
$$S_N(x) := \frac{1}{N} \sum_{n = 1}^N \exp(2\pi i \k \cdot T^n(x)),$$
where $\k$ has been suppressed from the notation. Weyl's equidistribution criterion in $\T^d$ (see e.g. \cite[Proposition 1.1.2]{TaoHigher}) says that $(x_n) \subset \T^d$ is equidistributed in $\T^d$ if and only if for any $\k \in \Z^d \setminus \{0\}$ we have $\lim_{N \to \infty}\frac{1}{N} \sum_{n = 1}^N \exp(2\pi i \, \k \cdot x_n) = 0.$ Thus it is enough to prove that $S_N(x) \to 0$ as $N \to \infty$ for $\mu$ almost every $x$. Note that from this one can deduce that the result holds simultaneously for \emph{all} expanding toral endomorphisms  since such maps are represented by two dimensional matrices with integer coefficients and so there are only countably many possibilities.

Let $r_N := \int |S_N(x)|^2 \, d\mu(x)$. For each $x \in \T^d$ as $A$ has integer entries, we know that $\exp(2\pi i \k \cdot T^n(x)) = \exp(2\pi i(A^*)^n (\k) \cdot x).$ 
Then we may estimate
\begin{align*}
r_N & = \frac{1}{N^2}\sum_{m,n  = 1}^N \int \exp(-2\pi i \k \cdot (T^m(x)-T^n(x))) \, d\mu(x) \\
& =  \frac{1}{N^2}\sum_{m,n  = 1}^N \int \exp(-2\pi i ((A^*)^m(\k) - (A^*)^n(\k)) \cdot x) \, d\mu(x) \\
& \leq \frac{1}{N} + \frac{2}{N^2}\sum_{m = 2}^N \sum_{n = 1}^{m-1} |\widehat{\mu}((A^*)^m(\k) - (A^*)^n(\k))|.
\end{align*}
Let $\sigma_d = \inf\{|Ax| : |x| = 1\}$ be the smallest singular value of $A$, which is the same as the smallest singular value of $A^*$. Thus we have for any $n \in \N$ and $y \in \R^d$ that $|(A^*)^n y| \geq \sigma_d^n |y|$. Since $\sigma_d > 1$, we have $\sigma_d^n \to \infty$ as $n \to \infty$. Thus there exists $n_0 \in \N$ such that whenever $m,n \in \N$ with $m > n \geq n_0$, we have
$$\sigma_d^{n} \geq  1 + \frac1{\sigma_d - 1} \geq 1 + \frac1{\sigma_d^{m-n} - 1} = \frac{\sigma_d^{m-n}}{\sigma_d^{m-n} - 1}.$$
Thus for $m > n \geq n_0$ we can bound $(\sigma_d^{m-n} - 1) \sigma_d^n \geq  \sigma_d^{m-n}$, giving
\begin{align*}
    |(A^*)^m(\k) - (A^*)^n(\k)| &\geq ||(A^*)^m(\k)| - |(A^*)^n(\k)|| \geq (\sigma_d^{m-n} - 1)\sigma_d^{n}|\k| \geq \sigma_d^{m-n}
\end{align*}
again, valid for $m > n \geq n_0$. By the Fourier decay and $|(A^*)^n y| \geq \sigma_d^n |y|$ we obtain
\begin{align*}
\sum_{m = 2}^N \sum_{n = 1}^{m-1} |\widehat{\mu}((A^*)^m(\k) - (A^*)^n(\k))| & \leq N(n_0 - 1) + \sum_{m = 2}^N \sum_{n = n_0}^{m - 1} |\widehat{\mu}((A^*)^m(\k) - (A^*)^n(\k))|\\ &\lesssim  N + \sum_{m = 2}^N \sum_{n = n_0}^{m-1} (\log |(A^*)^m(\k) - (A^*)^n(\k)|)^{-\alpha}  \\
& \leq N + \sum_{m = 2}^N \sum_{n = n_0}^{m-1} (\log \sigma_d^{m-n})^{-\alpha} \\
& = N + \log(\sigma_d)^{-\alpha}\sum_{m = 2}^N \sum_{n = n_0}^{m-1}  \frac{1}{(m-n)^{\alpha} }\\
& = \cO(N + N^{2-\alpha}),
\end{align*}
 Since $\alpha > 0$ we have that $\sum_{N = 1}^\infty \frac{r_N}{N} < \infty$. Thus by Lemma \ref{lma:sum} there is a subsequence $N_j \to \infty$, $j \in \N$, such that $\sum_{j = 1}^\infty r_{N_j} < \infty$ and $\lim_{j \to \infty} \frac{N_{j+1}}{N_j} = 1.$ In particular, the former condition yields that
$$\int \sum_{j = 1}^\infty |S_{N_j}(x)|^2 \, d\mu(x) < \infty$$
and so $S_{N_j}(x) \to 0$ for $\mu$ almost every $x$ as $j \to \infty$. Now we just check that the latter condition actually yields $S_{N}(x) \to 0$ for $\mu$ almost every $x$ as $N \to \infty$, which is what we need. Fix $N \in \N$ and find $j \in \N$ such that $N_j \leq N \leq N_{j+1}$. This yields that
$$|NS_N(x) - N_j S_{N_j}(x)| = \Big|\sum_{n = N_j+1}^{N} \exp(-2\pi i \k \cdot T^n(x))\Big| \leq N-N_j \leq N_{j+1} - N_j.$$
Hence
$$|S_N(x)| \leq |S_{N_j}(x)| + \frac{N_{j+1} - N_j}{N_j},$$
which converges to $0$ as $j \to \infty$ at $\mu$ almost every $x$.
\end{proof}

\section{Proof of Lemma \ref{Lemma:Maximising multinomial}}\label{appendix:Multinomial}
In this section we will prove lemma \ref{Lemma:Maximising multinomial}. Let us recall its statement. 
\begin{lem}
	Let $(p_{a})_{a\in \cA}$ be a probability vector and $n\in \mathbb{N}$, then 
	$$\frac{n!}{\prod_{a\in \cA} k_{a}!}\prod_{a\in \cA}p_{a}^{k_{a}}\ll n^{-1/2(\#\cA-1)}$$ for any $(k_{a})_{a\in \cA}\in \mathbb{N}^{\#\cA}$ satisfying $\sum_{a\in \cA}k_{a}=n.$ Here the underlying constants only depend upon the probability vector.
\end{lem}

\begin{proof} 
Let $(k_{a}^{*})_{a\in \cA}$ be such that 
$$\max_{(k_{a})_{a\in \cA}:\sum_{a\in \cA} k_{a}=n}\frac{n!}{\prod_{a\in \cA} k_{a}!}\prod_{a\in \cA}p_{a}^{k_{a}}=\frac{n!}{\prod_{a\in \cA} k_{a}^{*}!}\prod_{a\in \cA}p_{a}^{k_{a}^{*}}.$$
We claim that $(k_{a}^{*})_{a\in \cA}$ must satisfy 
\begin{equation}
	\label{Eq:Maximising regime}
k_{a}^{*}\in \set{\lfloor p_{a}n\rfloor - 10\#\cA,\ldots, \lfloor p_{a}n\rfloor+ 10\#\cA}
\end{equation} for all $a\in \cA$. Suppose for instance that $k_{a'}^{*}<\lfloor p_{a'}n\rfloor - 10\#\cA$ for some $a'\in \cA$, then there must exists $a''\in \cA$ such that $k_{a''}^{*}\geq \lfloor p_{a''}n\rfloor+2$. We have 
\begin{align}
	\label{Eq:Comparing probabilities}
	&\frac{n!}{(k_{a'}^{*}+1)!(k_{a''}^{*}-1)!\prod_{a\in \cA\setminus\set{a',a''}} k_{a}^{*}!}p_{a'}^{k_{a'}^{*}+1}p_{a''}^{k_{a''}^{*}-1}\prod_{a\in \cA\setminus\set{a',a''}}p_{a}^{k_{a}^{*}}\nonumber\\
	=&\frac{k_{a''}^{*}p_{a'}}{(k_{a'}^{*}+1)p_{a''}}\frac{n!}{\prod_{a\in \cA} k_{a}^{*}!}\prod_{a\in \cA}p_{a}^{k_{a}^{*}}
\end{align}
However, we have the following lower bound for $\frac{k_{a''}^{*}p_{a'}}{(k_{a'}^{*}+1)p_{a''}}$:
\begin{equation}
	\label{Eq:>1 bound}
	\frac{k_{a''}^{*}p_{a'}}{(k_{a'}^{*}+1)p_{a''}}\geq \frac{(\lfloor p_{a''}n\rfloor+2)p_{a'}}{(\lfloor p_{a'}n\rfloor - 10\#\cA)p_{a''}}\geq \frac{p_{a''}p_{a'}n+p_{a'}}{p_{a'}p_{a''}n-10\#\cA p_{a''}}>1.\end{equation}
Substituting \eqref{Eq:>1 bound} into \eqref{Eq:Comparing probabilities} we have $$\frac{n!}{(k_{a'}^{*}+1)!(k_{a''}^{*}-1)!\prod_{a\in \cA\setminus\set{a',a''}} k_{a}^{*}!}p_{a'}^{k_{a'}^{*}+1}p_{a''}^{k_{a''}^{*}-1}\prod_{a\in \cA\setminus\set{a',a''}}p_{a}^{k_{a}^{*}}>\frac{n!}{\prod_{a\in \cA} k_{a}^{*}!}\prod_{a\in \cA}p_{a}^{k_{a}^{*}}.$$ However this contradicts our assumption that $(k_{a}^{*})$ was maximizing. The case where $k_{a'}^{*}>\lfloor p_{a'}n\rfloor - 10\#\cA$ follows similarly. Thus we have proved our claim.

We now claim that if $(k_{a})_{a\in \cA}$ satisfies 
\begin{equation}
	\label{Eq:k_a maximising regime}
k_{a}\in \set{\lfloor p_{a}n\rfloor - 10\#\cA,\ldots, \lfloor p_{a}n\rfloor+ 10\#\cA }
\end{equation} for all $a\in \cA$ then 
\begin{equation}
	\label{Eq:Bounding the maximum}
	\frac{n!}{\prod_{a\in \cA} k_{a}^{*}!}\prod_{a\in \cA}p_{a}^{k_{a}^{*}}\ll \frac{n!}{\prod_{a\in \cA} k_{a}!}\prod_{a\in \cA}p_{a}^{k_{a}}.	
\end{equation}We remark that if $(k_a)$ satisfies \eqref{Eq:k_a maximising regime} then $$\frac{n!}{\prod_{a\in \cA} k_{a}^{*}!}\prod_{a\in \cA}p_{a}^{k_{a}^{*}}$$ can be obtained from $$\frac{n!}{\prod_{a\in \cA} k_{a}!}\prod_{a\in \cA}p_{a}^{k_{a}}$$ by multiplying the latter term a bounded number of times by terms each of which is bounded above by \begin{equation}
\label{Eq:Change of probability factor bound}
\frac{\max_{a\in \cA}\set{\lfloor p_{a}n\rfloor+ 10\#\cA}\max_{a\in \cA}\set{p_{a}}}{\min_{a\in \cA}\set{\lfloor p_{a}n\rfloor- 10\#\cA}\min_{a\in \cA}\set{p_{a}}}.
\end{equation} Since \eqref{Eq:Change of probability factor bound} is uniformly bounded from above our claim follows. 

It follows from \eqref{Eq:Bounding the maximum} that to prove our lemma we just need to show that the desired bound holds for a specific $(k_{a})$ satisfying \eqref{Eq:k_a maximising regime}. This we do now. Pick $a'\in \cA$ arbitrarily, then by Stirling's formula we have
\begin{align*}
&\frac{n!}{(n-\sum_{a\in \cA\setminus\set{a'}}\lfloor p_an\rfloor)!\prod_{a\in \cA\setminus\set{a'}}\lfloor p_{a}n\rfloor!}p_{a'}^{(n-\sum_{a\in \cA\setminus\set{a'}}\lfloor p_an\rfloor)}\prod_{a\in \cA\setminus\set{a'}}p_{a}^{\lfloor p_{a}n\rfloor}\\
&\ll \frac{\sqrt{n}n^{n}p_{a'}^{(n-\sum_{a\in \cA\setminus\set{a'}}\lfloor p_an\rfloor)}}{\sqrt{(n-\sum_{a\in \cA\setminus\set{a'}}\lfloor p_an\rfloor)}(n-\sum_{a\in \cA\setminus\set{a'}}\lfloor p_an\rfloor)^{(n-\sum_{a\in \cA\setminus\set{a'}}\lfloor p_an\rfloor)}} \prod_{a\in \cA\setminus \set{a'}}\frac{p_{a}^{\lfloor p_{a}n\rfloor}}{\sqrt{\lfloor p_{a}n\rfloor}\lfloor p_{a}n\rfloor^{\lfloor p_{a}n\rfloor}}.
\end{align*}
The following bounds are straightforward to verify \begin{equation}
	\label{Eq:Integer part bounds}
	\frac{(p_{a}n)^{\lfloor p_{a}n\rfloor}}{\lfloor p_{a}n\rfloor^{\lfloor p_{a}n\rfloor}}\ll 1\qquad { and }\qquad \frac{(p_{a'}n)^{(n-\sum_{a\in \cA\setminus\set{a'}}\lfloor p_an\rfloor)}}{(n-\sum_{a\in \cA\setminus\set{a'}}\lfloor p_an\rfloor)^{(n-\sum_{a\in \cA\setminus\set{a'}}\lfloor p_an\rfloor)}}\ll 1.
\end{equation} Substituting \eqref{Eq:Integer part bounds} into the above yields:
\begin{align*}
&\frac{n!}{(n-\sum_{a\in \cA\setminus\set{a'}}\lfloor p_an\rfloor)!\prod_{a\in \cA\setminus\set{a'}}\lfloor p_{a}n\rfloor!}p_{a'}^{(n-\sum_{a\in \cA\setminus\set{a'}}\lfloor p_an\rfloor)}\prod_{a\in \cA\setminus\set{a'}}p_{a}^{\lfloor p_{a}n\rfloor}\\
\ll &\frac{\sqrt{n}}{\sqrt{(n-\sum_{a\in \cA\setminus\set{a'}}\lfloor p_an\rfloor)}\prod_{a\in\cA\setminus\set{a'}}\sqrt{p_{a}n}}\\
\ll &n^{-1/2(\#\cA-1)}.
\end{align*} Thus our result follows.
\end{proof}


\section{Proof of Proposition \ref{prop:spectralgap}}\label{appendix:Spectral gap}
The purpose of this section is to prove Proposition \ref{prop:spectralgap}, which we recall here:

\begin{prop}[Spectral gap]  Let $\Phi$ be a conformal IFS satisfying the strong separation condition and $\Sigma_{A}$ be a subshift of finite type. Suppose that $(\Phi,\Sigma_{A})$ satisfies the UNI condition, then for any $C^{1}$ potential $\psi$ satisfying $P(\psi)=0,$  there exists $0 < \rho < 1$ such that for all large enough $|b|$, $n \in \N$ and $h \in C^{1,b}(U)$, we have
$$\|\mathcal{L}_{ib}^{n} h\|_{b} \ll \rho^{n}|b|^{1/2}\|h\|_b.$$ In particular, if $(\Phi,\Sigma_{A})$ satisfies the UNI condition, then for any $C^{1}$ potential $\psi$ satisfying $P(\psi)=0$ the triple $(\Phi,\Sigma_{A},\psi)$ has a spectral gap.
\end{prop}
In our proof of this proposition we will make use of a uniform non-integrability condition introduced in \cite{LiPan}. It follows from Proposition \ref{prop:weaktostrongUNI} that this condition is equivalent to our UNI condition. Thus in this section, by Proposition \ref{prop:weaktostrongUNI} we may freely assume the existence of $r>0$ and $\epsilon_{0}>0$ such that the following holds for any large $n\in \N$. 
For any $x\in X_A$ and unit vector $e\in\mathbb{R}^{d}$, there exist $\a_1, \a_2\in \cW_{A}\cap \cA^n$ such that:
    \begin{itemize}
    \item For all $y\in B(x,r)$ we have $$\left|\partial_{e}\left(\log |\lambda_{\a_1}(y)|-\log |\lambda_{\a_2}(y)|\right)\right|\geq \epsilon_{0}.$$
        \item For all $y\in B(x,r)$ we have $\a_1\leadsto y$ and $\a_2\leadsto y$.
    \end{itemize}

Without loss of generality we may assume that our (untwisted) transfer operator satisfies 
\begin{equation}
    \label{Eq:Eigenfunction 1}
    \sum_{\stackrel{a\in \cA}{a\leadsto x}}w_{a}(x)=1 
\end{equation}for all $x\in X_{A}.$ This is permissible as one can conjugate our transfer operator by a multiplication operator so that the resulting transfer operator satisfies \eqref{Eq:Eigenfunction 1}. Moreover, it can be shown that if this new operator satisfies the conclusion of Proposition \ref{prop:spectralgap}, then our original transfer operator also satisfies this conclusion. Thus there is no loss of generality in making this assumption. For further details of this reduction we refer the reader to \cite[Section 5.1]{Naud-Cantor}.

Recall that $U = \bigcup_{a \in \cA} U_a$ is a choice of open set satisfying $X_{\Phi}\subset U$, where the $U_a$ are disjoint open sets with $f_a(U) \subset U_a$ for all $a \in \cA$.

In our proof of Proposition \ref{prop:spectralgap} we will make regular use of the following set of functions: For $A>0$ let
$$C_{A}:=\{u\in C^{1}(U): u>0, \|D_{x}u\|\leq A u(x)\}.$$ We record here a useful property of the cone $C_{A}.$ If $u\in C_{A}$ then 
\begin{equation}
    \label{Eq:Property of cone}
e^{-A|x-y|}\leq\frac{u(x)}{u(y)}\leq e^{A|x-y|}
\end{equation} for all $x,y\in U.$ 

We will also make regular use of the following lemma. It records various well known bounded distortion estimates whose proofs we omit.
\begin{lem}
\label{Lemma:Bounded distortion collection}
Let $\Phi=\{f_{a}\}_{a\in \cA}$ be a conformal IFS satisfying the strong separation condition and $\psi:U\to\mathbb{R}$ be a $C^{1}$ potential. Then there exists $C_{1},C_{2},C_{3}>0$ such that the following statements are true:
\begin{itemize}
\item For any $x,y\in X_{\Phi}$ and $\a\in \cA^{*}$ we have $$e^{-C_{1}|x-y|}\leq\frac{w_{\a}(x)}{w_{\a}(y)}\leq e^{C_{1}|x-y|}.$$
\item For any $x\in X_{\Phi}$ and $\a\in \cA^*$ we have $$\|D_{x}w_{\a}\|\leq C_{2}w_{\a}(x).$$
\item For any $x\in X_{\Phi}$ and $\a\in \cA^{*}$ we have $$\|D_{x}\log |\lambda_{\a}|\|\leq C_{3}.$$
\end{itemize}
\end{lem}

\subsection{Construction of Dolgopyat operators}

In what follows we will fix an IFS $\Phi$ and a subshift of finite type $\Sigma_{A}$ such that the assumptions of Proposition \ref{prop:spectralgap} hold. We will also fix a $C^1$ potential $\psi$ satisfying $P(\psi)=0.$ In this section we will denote the corresponding Gibbs measure by $\mu$.

In this section we will define a family of \textit{Dolgopyat operators} $(\cN_{b}^{J})_{J\in \cE_{b}}$ that will allow us to prove Proposition \ref{prop:spectralgap} via properties of complex transfer operators. First of all, let us fix $A>1$ sufficiently large such that 
\begin{equation}
	\label{Eq:A large 1}
	\left(C_2+2|b|+A|b|\gamma^{n}\right)\leq A|b|
\end{equation}
 and 
\begin{equation}
	\label{Eq:A large 2}
	\left(2C_{2}+2C_{3}|b|+2A|b|\gamma^{n}\right)\leq A|b|
\end{equation} for all $|b|>1$ and $n$ sufficiently large. This is slightly abusing notation with the matrix $A$ defining the subshift of finite type. Here $\gamma\in(0,1)$ is the uniform contraction factor in our IFS and $C_{2},C_{3}>0$ are as in Lemma \ref{Lemma:Bounded distortion collection}.

We now state the key proposition where the UNI condition manifests. 

\begin{prop}
	\label{Proposition:Arguments cancellation}
There exist $\eta_{0}=\eta_{0}(\Phi,\Sigma_{A})\in(0,1),$ $D=D(\Phi,\Sigma_{A})>0,$ 
$\epsilon_{1}=\epsilon_{1}(\Phi,\Sigma_{A})$ and $\epsilon_{2}=\epsilon_{2}(\Phi,\Sigma_{A})>0$
such that for $N\in \N$ sufficiently large, $|b|>1$, $H\in C_{A|b|},$ and $u\in C^{1}(U),$ if $u$ satisfies $|u|\leq H$ and $\|Du\|\leq A|b|H,$ then for any $y\in X_{A}$ there exists $x'\in X_{A}\cap B\left(y,\frac{\epsilon_{1} D}{|b|}\right)$ such that 
$$B\left(x',\tfrac{\epsilon_{2}}{|b|}\right)\subset  B\left(y,\tfrac{\epsilon_{1} D}{|b|}\right)$$ 
and for which we have the following: there exists $\a_1, \a_2\in \cW_{A}\cap \cA^N$ such that $\a_{1}\leadsto x$ and $\a_{2}\leadsto x$ for all $x\in B\left(y,\frac{\epsilon_{1}D}{|b|}\right),$ and one of the following holds for all $x\in B\left(x',\frac{\epsilon_{2}}{|b|}\right)$:
\begin{align}\label{eq:cancel1}
	&\left|w_{\a_{1}}(x)|\lambda_{\a_{1}}(x)|^{ib}u(f_{\a_{1}}(x))+w_{\a_{2}}(x)|\lambda_{\a_{2}}(x)|^{ib}u(f_{\a_{2}}(x))\right|\\
 &\leq \eta_{0}w_{\a_{1}}(x)H(f_{\a_{1}}(x))+w_{\a_{2}}(x)H(f_{\a_{2}}(x)), 
 \nonumber
 \end{align}
 or
 \begin{align}\label{eq:cancel2}
&\left|w_{\a_{1}}(x)|\lambda_{\a_{1}}(x)|^{ib}u(f_{\a_{1}}(x))+w_{\a_{2}}(x)|\lambda_{\a_{2}}(x)|^{ib}u(f_{\a_{2}}(x))\right|\\
&\leq w_{\a_{1}}(x)H(f_{\a_{1}}(x))+\eta_{0}w_{\a_{2}}(x)H(f_{\a_{2}}(x))\nonumber
\end{align}
\end{prop}

The proof of Proposition~\ref{Proposition:Arguments cancellation} will be given in Section~\ref{sec:Proof of cancellation} below.
Having found the parameters $\eta_0,D,\eps_1,\eps_2$ in Proposition \ref{Proposition:Arguments cancellation}, we can now define the family of Dolgopyat operators.

\begin{definition}[Dolgopyat operators $(\cN_{b}^J)_{J \in \cE_b}$]
Assume $|b|> 1$ is fixed and $N\in \N$ is a large number so that Proposition \ref{prop:spectralgap} applies. We now define our Dolgopyat operators. We fix a maximal set of points $\{y_l\}_{l\in S(b)}\subset X_{A}$ with some index set $S(b)$ such that $$B\left(y_{l},\frac{\epsilon_{1}D}{|b|}\right)\cap B\left(y_{l'},\frac{\epsilon_{1}D}{|b|}\right)=\emptyset$$ for distinct $l,l'\in S(b)$. By Proposition \ref{Proposition:Arguments cancellation}, applied to $y_l$ for each $l\in S(b)$, we can always find a point $x_{l}'\in B\left(y_{l},\frac{\epsilon_{1}D}{|b|}\right)$ satisfying  $$B\left(x_{l}',\frac{\epsilon_{2}}{|b|}\right)\subset B\left(y_{l},\frac{\epsilon_{1}D}{|b|}\right)$$ 
and a word $\a_{l}\in \cA^{N}$ satisfying $\a_{l}\leadsto x$ for all $x\in B\left(y_{l},\frac{\epsilon_{1}D}{|b|}\right)$ such that either \eqref{eq:cancel1} holds with $\a_1 = \a_l$ or \eqref{eq:cancel2} holds with $\a_2 = \a_l$. Let $\eta_0 \in (0,1)$ be the value from Proposition \ref{Proposition:Arguments cancellation} and fix $\eta\in [\eta_{0},1)$, then we can choose a $C^{1}$ function $\chi:U\to [\eta,1]$ such that $\chi(x)=1$ outside of $\cup_{l\in S(b)}f_{\a_{l}}(B(x_{l}',\frac{\epsilon_{2}}{|b|}))$ and $\chi(x)=\eta$ within $\cup_{l\in S(b)}f_{\a_{l}}(B(x_{l}',\frac{\epsilon_{2}}{3|b|}))$. Moreover, we can choose $\eta$ in such a way that $\|D(\chi\circ f_{\a})\|\leq |b|$ for all $\a\in \cA^N$. We can also freely assume $\eta\geq 1/2$. We emphasize that the function $\chi$ only depends upon the pairs $\{(x_l',\a_{l})\}_{l\in S(b)}.$ We let
\begin{align*}
    &\cE_{b} :=
    \left\lbrace (x_{l}',\a_{l})_{l\in S(b)}:
    \begin{array}{cc}
         B\left(x_{l}',\frac{\epsilon_{2}}{|b|}\right)\subset B\left(y_{l},\frac{\epsilon_{1}D}{|b|}\right)\, \textrm{ and }  \\
         \a_{l}\in \cA^{N} \textrm{ satisfies } \a_{l}\leadsto x\, \forall x\in B\left(y_{l},\frac{\epsilon_{1}D}{|b|}\right)
    \end{array}
    \right\rbrace.
\end{align*}
 Given $J\in \cE_{b}$ we define the \textit{Dolgopyat operator} $\cN_{b}^{J}$ on $C^{1}(U)$ according to the rule $$\cN_{b}^{J}(H):=\cL_{0}^{N}(\chi H).$$
\end{definition}

Crucially, Dolgopyat operators satisfy the following key properties, which we will prove using Proposition \ref{Proposition:Arguments cancellation}, the doubling property of the Gibbs measure $\mu,$ and basic inequalities for complex transfer operators.

\begin{lem}[Key properties of Dolgopyat operators]
\label{Lemma:Important lemma}
There exists $N\in \N,$ $A>1$ and $\rho\in(0,1)$ such that for $|b|$ sufficiently large the Dolgopyat operators $(\mathcal{N}_{b}^{J})_{J\in \cE_{b}}$ on $C^{1}(U)$ satisfy the following properties:
\begin{enumerate}
	\item The cone $C_{A|b|}$ is stable under $\mathcal{N}_{b}^{J}$ for all $J\in \cE_{b}$, that is, if $H \in C_{A|b|}$, then
	$$\|D_{x}\mathcal{N}_{b}^{J}(H)\| \leq A |b| \mathcal{N}_{b}^{J}(H)(x), \quad x \in U.$$
	\item For all $H\in C_{A|b|}$ and $J\in \mathcal{E}_{b}$, $$\int |\cN_{b}^{J}H|^{2}d\mu\leq \rho \int |H|^{2}d\mu.$$
	\item Given $H\in C_{A|b|}$ and $u\in C^{1}(U)$ such that $|u|\leq H$ and $\|Du\|\leq A|b|H$, there exists $J\in \cE_{b}$ such that $$|\cL^{N}_{ib}u|\leq \mathcal{N}_{b}^{J}(H)$$
	and $$\|D(\cL^{N}_{ib}u)\|\leq A|b|\mathcal{N}_{b}^{J}(H).$$
\end{enumerate} 
\end{lem}
The proof of this lemma will be given in Section~\ref{sec:Dolgopyat operators properties} below.

\subsection{Reduction of spectral gap to Dolgopyat operators}

Let us now show how Proposition \ref{prop:spectralgap} follows using Lemma \ref{Lemma:Important lemma}. Our proof is based upon arguments of Naud \cite{Naud-Cantor}. For the purposes of brevity we will omit certain calculations from our proof. 

\begin{proof}[Proof of Proposition \ref{prop:spectralgap}]
We start by proving that Lemma \ref{Lemma:Important lemma} implies an $L^{2}$ decay statement. Let $u\in C^{1}(U)$ satisfy $\|u\|_{b}\leq 1$. Set $H=\|u\|_{b}1$. Then $H\in C_{A|b|}$ and $|u|\leq H$, $\|Du\|\leq  |b|\|u\|_{b}\leq A|b|H.$ Thus by the definition of $\cL_{ib}$ and iterating Lemma \ref{Lemma:Important lemma}(3), we get for all $n\geq 1$, $|\cL_{ib}^{nN}u|\leq \cN_{b}^{J_n}\cN_{b}^{J_{n-1}}\ldots \cN_{b}^{J_{1}}H$ and $\|D\cL_{b}^{nN}u\|\leq A|b|\cN_{b}^{J_n}\cN_{b}^{J_{n-1}}\ldots \cN_{b}^{J_{1}}H$ for some $J_1,\ldots, J_{n}\in \cE_{b}$. Hence 
\begin{equation}
	\label{Eq:L^{2} decay}
\int_{X}|\cL_{b}^{nN}u|^{2}\, d\mu\leq \int_{X}|\cN_{b}^{J_n}\cN_{b}^{J_{n-1}}\ldots \cN_{b}^{J_{1}}H|^{2}\, d\mu \leq \rho^{n}\int_{X}|H|^{2} \, d\mu\leq \rho^{n}.
\end{equation}
Now we will show how \eqref{Eq:L^{2} decay} leads to a spectral gap. We will use the following well known inequalities: Let $u\in C^{1}(U)$. There exists $C_1,C_{2}>0$ and $\rho_{1}\in(0,1)$ such that
\begin{itemize}
	\item[(i)] $\sup_{x\in U}\|D_{x}\cL_{ib}^{n}u\|\leq C_{1}|b|\|\cL_{0}^{n}u\|_{\infty}+\rho_{1}^{n}\|\cL_{0}^{n}(\|Du\|)\|_{\infty}$
	\item[(ii)] $\|\cL_{ib}^{n}u\|_{\infty}\leq \int|u|\, d\mu+C_{2}\rho_{1}^{n}\sup_{x\in U}\|D_{x}u\|$
\end{itemize}
We remark that the key inequality for complex transfer operators (i) was proved in \cite{ParryPollicott} in the symbolic setting and (ii) follows from the quasi-compactness of $\cL_{ib}$.

Now we prove our spectral gap result. Fix $h\in C^{1,b}(U)$ with $\|h\|_{b}\leq 1$. Using the above and the Cauchy-Schwartz inequality, we have
\begin{align*}
	|\cL_{ib}^{2nN}h(x)|\leq |\cL_{ib}^{nN}(\cL_{ib}^{nN}h)(x)|&\leq \cL_{0}^{nN}(|(\cL_{ib}^{nN}h|)(x)\\
	&\ll (\cL_{0}^{nN}(|(\cL_{ib}^{nN}h|^{2})(x))^{1/2}\\
	&\leq \left(\int_{X}|\cL_{ib}^{nN}h|^{2}\, d\mu +C_{2}\rho_{1}^{nN}\sup_{x\in U}\|D_{x}|\cL_{ib}^{nN}h|^{2}\|\right)^{1/2}\\
	&\ll \left(\int_{X}|\cL_{ib}^{nN}h|^{2}\, d\mu +C_{2}\rho_{1}^{nN}|b|\right)^{1/2}.
\end{align*} 
In the penultimate line we used the second of the inequalities listed above.
In the final line we used the product rule, the first of the inequalities listed above, and the fact $\|h\|_{b}\leq 1$ to assert that $\sup_{x\in U}\|D_{x}|\cL_{ib}^{nN}h|^{2}\|\ll |b|.$ Now using \eqref{Eq:L^{2} decay} we have 
$$|\cL_{ib}^{2nN}h(x)|\ll \left(\rho^{n}+C_{2}\rho_{1}^{nN}|b|\right)^{1/2}\ll \max\{\rho^{1/4N},\rho_{1}^{1/4}\}^{2nN}|b|^{1/2}.$$ This implies 
\begin{equation}
\label{Eq:inftynorm bound}
\|\cL_{ib}^{2nN}\|_{\infty}\ll\max\{\rho^{1/4N},\rho_{1}^{1/4}\}^{2nN}|b|^{1/2}\|h\|_{b}.
\end{equation}This establishes the sup norm part of our spectral gap bound. It remains to prove the derivative part. Using the first of the inequalities listed above we have:
\begin{align*}
	\frac{1}{|b|}\sup_{x\in U}\|D_{x}\cL_{ib}^{2nN}h\|&\leq C_{1}\|\cL^{nN}_{0}(\cL_{ib}^{nN}h)\|_{\infty}+\frac{\rho_{1}^n}{|b|}\|\cL_{0}^{nN}(\|D\cL_{ib}^{nN}h\|)\|_{\infty}\\
	&\ll C_{1}\|\cL^{nN}_{0}(\cL_{ib}^{nN}h)\|_{\infty}+\rho_{1}^n\|\cL_{ib}^{nN}h\|_{b}.
\end{align*}
The first term in the above can be bounded using similar ideas to those used above to establish \eqref{Eq:inftynorm bound}. The second term can be bounded by evaluating $D_{x}\cL_{ib}^{nN}h$ and bounding accordingly. 
\end{proof}

\subsection{Proof of Proposition \ref{Proposition:Arguments cancellation}}
\label{sec:Proof of cancellation}
Let us now prove Proposition \ref{Proposition:Arguments cancellation}. To this end we will need the following two lemmas. Firstly, we need the following higher dimensional analogue of Lemma 5.11 from \cite{Naud-Cantor}:

\begin{lem}
\label{Lemma:fH lemma}
Let $Z\subset U$ be a set satisfying $Diam(Z)\leq \frac{c}{|b|}$. Let $H\in C_{A|b|}$ and $u\in C^{1}(U)$ satisfy $|u|\leq H$ and $\|Du\|\leq A|b|H$. Then for $c$ sufficiently small, we have either $|u(x)|\leq \frac{3}{4}H(x)$ for all $x\in Z$, or $|u(x)|\geq \frac{1}{4}H(x)$ for all $x\in Z$.
\end{lem}

\begin{proof}The proof is same as in \cite{Naud-Cantor} and we include it for completeness.  If $x_0 \in Z$ has $|u(x_0)| \leq \frac{1}{4} H(x_0)$, then at any $x \in Z$, we have for all small enough $c > 0$:
$$|u(x)| \leq |u(x) - u(x_0)|  +   \frac{1}{4} H(x_0) \leq A|b| \diam Z \sup_{x \in Z} H  +   \frac{1}{4} H(x_0) \leq (cA + \tfrac{1}{4})e^{Ac} H(x) \leq \frac{3}{4}H(x).$$ In the penultimate inequality we used \eqref{Eq:Property of cone}.
\end{proof}

We also recall the following triangle lemma that follows from applying trigonometric identities:

\begin{lem}[Triangle lemma]
\label{Lemma: different arguments}
Let $z_{1},z_{2}\geq 0$ be two complex numbers such that $|z_{1}/z_{2}|\leq L$ and $2\pi-\epsilon\geq |arg(z_1)-arg(z_2)|\geq \epsilon>0$. Then there exists $0<\delta(L,\epsilon)<1$ such that $$|z_{1}+z_{2}|\leq (1-\delta)|z_{1}|+|z_{2}|.$$
\end{lem}

\begin{proof}[Proof of Proposition \ref{Proposition:Arguments cancellation}]
Let $N\in \N$ be large and $|b|>1$. Let us also fix $H\in C_{A|b|}$ and $u\in C^{1}(U)$ satisfying the assumptions of Proposition \ref{Proposition:Arguments cancellation}. Let $y\in X_{A}$. Let $\epsilon_1>\epsilon_{2}>0$ be parameters that only depend upon $\Phi$ and $\Sigma_{A}$. Using properties of $\Phi$ and the fact $\Sigma_{A}$ is topologically mixing, by the strong separation of the IFS $\Phi$, there exists $D>D'>2$ depending only on $\Phi$ and $\Sigma_{A}$ such that there exists $x'\in X_{A}$ satisfying 
$$x'\in B\left(y,\frac{\epsilon_{1} D}{|b|}\right)\setminus B\left(y,\frac{\epsilon_{1} D'}{|b|}\right)$$ 
and $$B\left(x',\frac{\epsilon_{2}}{|b|}\right)\subset  B\left(y,\frac{\epsilon_{1} D}{|b|}\right).$$ Let $B_{1}=B\left(y,\frac{\epsilon_{2}}{|b|}\right),$ $B_{2}= B\left(x',\frac{\epsilon_{2}}{|b|}\right)$. Let $\hat{B}$ be the smallest ball containing $B_{1}\cup B_{2}$. We can assume that $\epsilon_{1}$ and $\epsilon_{2}$ are sufficiently small that $\hat{B}\subset B(y,r)$. Here $r$ is as in the uniform non-integrability condition stated at the beginning of this section (recall that by Proposition \ref{prop:weaktostrongUNI} this is equivalent to our UNI condition).  We remark that for any $x_{1}\in B_{1}$ and $x_{2}\in B_{2}$ we have 
\begin{equation}
\label{Eq:Distance between balls}
\frac{\epsilon_{1}}{b}\leq \|x_1-x_2\|\leq (D+2)\frac{\epsilon_{1}}{b}.
\end{equation}

Let $e_0=\frac{(y-x')}{\|y-x'\|}$. Applying the uniform non-integrability condition stated at the start of this section for this $e_0$ and $y,$ we see that there exists $\a_{1},\a_{2}\in \cW_{A}\cap \cA^{N}$ such that 	
\begin{equation}
\label{Eq:UNI lower bound}
\left|\partial_{e_0}\left(\log |\lambda_{\a_{1}}(x)|-\log |\lambda_{\a_{2}}(x)|\right)\right|\geq \epsilon_{0}
\end{equation}
and $$\a_{1}\leadsto x\quad \textrm{and} \quad \a_{2}\leadsto x$$ for all $x\in B(y,r)$.  

Since our IFS is contracting, and we can choose $\epsilon_{1}$ to be small, we can assume that for $m\in \{1,2\}$ the set $f_{\a_m}(\hat{B})$ is sufficiently small so that we can apply Lemma \ref{Lemma:fH lemma} and assert that for $m\in \{1,2\}$ either $|u(f_{\a_m}(x))|\leq \frac{3}{4}H(f_{\a_m}(x))$ for all $x\in \hat{B}$ or $|u(f_{\a_m}(x))|\geq \frac{1}{4}H(f_{\a_m}(x))$ for all $x\in \hat{B}.$ If $|u(f_{\a_m}(x))|\leq \frac{3}{4}H(f_{\a_m}(x))$ for all $x\in \hat{B}$ for some $m\in \{1,2\}$ then our result follows immediately. Thus we will assume in what follows that 
\begin{equation}
	\label{eq:quarter bound}
	|u(f_{\a_m}(x))|\geq \frac{1}{4}H(f_{\a_m}(x))
\end{equation}
for all $x\in \hat{B}$ for $m\in \{1,2\}.$ 

For $x\in \hat{B},$ we set $$z_{1}(x)=w_{\a_1}(x)|\lambda_{\a_1}(x)|^{ib}u(f_{\a_1}(x)) \textrm{ and } z_{\a_2}(x)=w_{\a_2}(x)|\lambda_{\a_2}(x)|^{ib}u(f_{\a_2}(x)).$$

We claim that there exists $M>0$ such that either
$|z_{1}(x)|/|z_{2}(x)|\leq M$ for all $x\in \hat{B}$ or $|z_{2}(x)/z_{1}(x)|\leq M$ for all $x\in \hat{B}$. Using \eqref{eq:quarter bound} and our assumptions on $u$ we see that for all $x\in \hat{B}$ we have
$$\frac{w_{\a_1}(x)H(f_{\a_1}(x))}{4w_{\a_2}(x)H(f_{\a_2}(x))}\leq \left|\frac{z_{1}(x)}{z_{2}(x)}\right|\leq \frac{4w_{\a_1}(x)H(f_{\a_1}(x))}{w_{\a_2}(x)H(f_{\a_2}(x))}.$$
If there exists $x_{0}\in \hat{B}$ such that 	$$\frac{w_{\a_1}(x_0)H(f_{\a_1}(x_0))}{w_{\a_2}(x_0)H(f_{\a_2}(x_0))}\leq 1,$$ then using  \eqref{Eq:Property of cone}, and Lemma \ref{Lemma:Bounded distortion collection}, for all $x\in \hat{B}$ we have 
$$\frac{w_{\a_1}(x)H(f_{\a_1}(x))}{w_{\a_2}(x)H(f_{\a_2}(x))}\leq \frac{w_{\a_1}(x_0)e^{(2D+2)\epsilon_{1} (A+C_{1})}H(f_{\a_1}(x_0))}{w_{\a_2}(x_0)e^{-(2D+2)\epsilon_{1} (A+C_{1})}H(f_{\a_2}(x_0))}\leq e^{2(2D+2)\epsilon_{1} (A+C_{1})}.$$ Here $C_{1}>0$ is as in Lemma \ref{Lemma:Bounded distortion collection}. Thus in this case we can take $M:=4e^{2(2D+2)\epsilon_{1} (A+C_{1})}$. Similarly it can be shown that a suitable $M'$ exists such that if $$\frac{w_{\a_1}(x)H(f_{\a_1}(x))}{w_{\a_2}(x)H(f_{\a_2}(x))}\geq 1$$ for all $x\in \hat{B}$ then $|z_{2}(x)/z_{1}(x)|\leq M'$ for all $x\in \hat{B}.$ This completes the proof of our claim.

For each $m\in \{1,2\}$ define the function $g_{m}:[0,1]\to \mathbb{C}$ by $g_{m}(t)=z_{m}(y+t(x'-y))$. Then define $L_{m}:[0,1]\to \mathbb{C}$ by $$L_{m}(t)=\int_{0}^{t}\frac{g_{m}'(s)}{g_{m}(s)}\, ds + l_{0,m}$$ where $l_{0,m}$ satisfies $e^{l_{0,m}}=z_{m}(y)$.  $L_{m}$ is well defined as $g_{m}(t)\neq 0$ for all $t\in [0,1]$ and $e^{L_{m}(t)}=z_{m}(t)$ for all $t\in[0,1].$ Set $\Phi(t)=Im(L_{1}(t))-Im(L_{2}(t))$. Since the potential $\psi$ defining the Gibbs measure is differentiable, we have:
\begin{align*}
	\Phi'(t)&=Im\left(\frac{g_{1}'(t)}{g(t)}-\frac{g_{2}'(t)}{g_{2}(t)}\right)\\
	&=\, Im\left(\frac{\partial_{e_0}w_{\a_1}(y+t(x'-y))\|x'-y\||\lambda_{\a_1}(y+t(x'-y))|^{ib}u(f_{\a_1}(y+t(x'-y)))}{w_{\a_1}(y+t(x'-y))|\lambda_{\a_1}(y+t(x'-y))|^{ib}u(f_{\a_1}(y+t(x'-y)))}\right)\\
	&+Im\left(\frac{w_{\a_1}(y+t(x'-y))ib\partial_{e_0}(\log |\lambda_{\a_1}(y+t(x'-y))|)|x'-y|u(f_{\a_1}(y+t(x'-y)))}{w_{\a_1}(y+t(x'-y))u(f_{\a_1}(y+t(x'-y)))}\right)\\
	&+Im\left(\frac{w_{\a_1}(y+t(x'-y))|\lambda_{\a_1}(y+t(x'-y))|^{ib}D_{f_{\a_1}(y+t(x'-y))}uD_{y+t(x'-y)}f_{\a_1}(x'-y)}{w_{1}(y+t(x'-y))|\lambda_{\a_1}(y+t(x'-y))|^{ib}u(f_{\a_1}(y+t(x'-y)))}\right)\\
	&-Im\left(\frac{\partial_{e_0}w_{\a_2}(y+t(x'-y))\|x'-y\||\lambda_{\a_2}(y+t(x'-y))|^{ib}u(f_{\a_2}(y+t(x'-y)))}{w_{\a_2}(y+t(x'-y))|\lambda_{\a_2}(y+t(x'-y))|^{ib}u(f_{\a_2}(y+t(x'-y)))}\right)\\
	&-Im\left(\frac{w_{\a_2}(y+t(x'-y))ib\partial_{e_0}(\log |\lambda_{\a_2}(y+t(x'-y))|)|x'-y|u(f_{\a_2}(y+t(x'-y)))}{w_{\a_2}(y+t(x'-y))u(f_{\a_2}(y+t(x'-y)))}\right)\\
	&-Im\left(\frac{w_{\a_2}(y+t(x'-y))|\lambda_{\a_2}(y+t(x'-y))|^{ib}D_{f_{\a_2}(y+t(x'-y))}uD_{y+t(x'-y)}f_{\a_2}(x'-y)}{w_{\a_2}(y+t(x'-y))|\lambda_{\a_2}(y+t(x'-y))|^{ib}u(f_{\a_2}(y+t(x'-y)))}\right).
	\end{align*}
Note that the first and fourth term are equal to $0$ as we are taking imaginary parts of real numbers. Thus after some cancellation we are left with the identity:
\begin{align*}	
	\Phi'(t)&=b\|x'-y\|\left(\partial_{e_0}(\log |\lambda_{\a_1}(y+t(x'-y))|-\log |\lambda_{\a_2}(y+t(x'-y))|)\right)\\
	&+Im\left(\frac{D_{f_{\a_1}(y+t(x'-y))}uD_{y+t(x'-y)}f_{\a_1}(x'-y)}{u(f_{\a_1}(y+t(x'-y)))} \right)\\
	&-Im\left(\frac{D_{f_{\a_2}(y+t(x'-y))}uD_{y+t(x'-y)}f_{\a_2}(x'-y)}{u(f_{\a_2}(y+t(x'-y)))} \right).
\end{align*}
By our assumptions on $u$ and \eqref{eq:quarter bound}, we have
$$\left|\frac{D_{f_{\a_1}(y+t(x'-y))}uD_{y+t(x'-y)}f_{\a_1}(x'-y)}{u(f_{\a_1}(y+t(x'-y)))}\right|\leq \frac{4A|b|H(f_{\a_1}(y+t(x'-y)))\gamma^{n_{0}}|x'-y|}{H(f_{\a_1}(y+t(x'-y))}=4A|b|\gamma^{N}|x'-y|.$$
Similarly, 
$$\left|\frac{D_{f_{\a_2}(y+t(x'-y))}uD_{y+t(x'-y)}f_{\a_2}(x'-y)}{u(f_{\a_2}(y+t(x'-y)))}\right|\leq 4A|b|\gamma^{N}|x'-y|.$$ 
Appealing to bounded distortions it can be shown that there exists $m>0$ depending only on our IFS such that 
\begin{equation}
    \label{Eq:UNI upper bound}
\left| \partial_{e_0}(\log |\lambda_{\a_1}(y+t(x'-y))|-\log |\lambda_{\a_2}(y+t(x'-y))|)\right|\leq m, \end{equation} for any $t\in[0,1]$ (see also \eqref{Eq:directions and Jacobians}). Combining \eqref{Eq:UNI lower bound} and \eqref{Eq:UNI upper bound} with the above we have $$\left(\epsilon_{0}-8A\gamma^{N}\right)|b|\|x'-y\|\leq |\Phi'(t)|\leq \left(m+8A\gamma^{N}\right)|b|\|x'-y\|.$$ Thus for $N$ sufficiently large we have
$$\frac{\epsilon_{0}|b|}{2}\|x'-y\|\leq |\Phi'(t)|\leq 2m|b|\|x'-y\|.$$ Therefore $$\frac{\epsilon_{0}|b|}{2}\|x'-y\|\leq |\Phi(1)-\Phi(0)|\leq 2m|b|\|x'-y\|.$$ By \eqref{Eq:Distance between balls} we have $$\frac{\epsilon_{1}}{|b|} \leq \|x'-y\|\leq (D+2)\frac{\epsilon_{1}}{|b|}.$$ This implies 
$$\frac{\epsilon_{0}\epsilon_{1}}{2}\leq |\Phi(1)-\Phi(0)|\leq 2m(D+2)\epsilon_{1}.$$ Now choose $\epsilon_{1}$ sufficiently small that $(2m(D+2)+\epsilon_{0}/2)\epsilon_{1}\leq \pi$ and let $\epsilon'= \frac{\epsilon_{0}\epsilon_{1}}{8}.$ Suppose now that $$\Phi(1),\Phi(0)\in \cup_{k\in \mathbb{Z}}[2\pi k-\epsilon',2\pi k+\epsilon'].$$ Since $|\Phi(1)-\Phi(0)|\leq 2m(D+2)\epsilon_{1}$ we cannot have $\Phi(1)\in [2\pi k_{1}-\epsilon',2\pi k_{1}+\epsilon']$ and $\Phi(0)\in [2\pi k_{2}-\epsilon',2\pi k_{2}+\epsilon']$ for $k_{1}\neq k_{2}$. As we would then have $$2m(D+2)\epsilon_{1}\geq |\Phi(1)-\Phi(0)|\geq 2\pi -2\epsilon'=2\pi-\frac{\epsilon_{0}\epsilon_{1}}{4}$$ which is not possible given our choice of $\epsilon_{1}$. Therefore we must have $$\frac{\epsilon_{0}\epsilon_{1}}{2}\leq |\Phi(1)-\Phi(0)|\leq 2\epsilon' =\frac{\epsilon_{0}\epsilon_{1}}{4}.$$ Which is also not possible. Recall now the definition of $\Phi(t)=Im(L_{1}(t))-Im(L_{2}(t))$, where for $m\in \{1,2\}$ one defined $L_{m}(t)=\int_{0}^{t}\frac{g_{m}'(s)}{g_{m}(s)}\, ds + l_{0,m}$ and $g_{m}(t)=z_{m}(y+t(x'-y))$. It follows from the above that we cannot have $$|arg(z_{1}(y))-arg(z_{2}(y))|<\epsilon'$$ and 
$$|arg(z_{1}(x'))-arg(z_{2}(x'))|<\epsilon'$$ for both $y$ and $x'$ simultaneously. 

 Let us assume that $$|arg(z_{1}(x'))-arg(z_{2}(x'))|\geq \epsilon'.$$ The other case when $|arg(z_{1}(y))-arg(z_{2}(y))|\geq\epsilon'$ follows similarly. Appealing to the bound $$\left|\partial_{e}\left(\log |\lambda_{\a_{1}}(y)|-\log |\lambda_{\a_{2}}(y)|\right)\right|\leq m$$ which holds for all $y\in U$ and unit vectors $e\in\mathbb{R}^{d}$, it can be shown using a similar argument to that given above describing how arguments change along paths that there exists $\epsilon_{2}>0$ depending only on our IFS such that if $x\in B\left(x',\frac{\epsilon_{2}}{|b|}\right)\subset B\left(y,\frac{\epsilon_1 D}{|b|}\right)$ then $$|arg(z_{1}(x))-arg(z_{2}(x))|\geq \epsilon'/2.$$ Applying Lemma \ref{Lemma: different arguments} and our assumption $|u|\leq H,$ we can assert that there exists $\eta_{0}=\eta_{0}(M,\epsilon')$ such that one of the following holds for all $x\in B\left(x',\frac{\epsilon_{2}}{|b|}\right)$: either
 $$\left|w_{\a_1}(x)|\lambda_{\a_1}(x)|^{ib}u(f_{\a_1}(x))+w_{\a_2}(x)|\lambda_{\a_2}(x)|^{ib}u(f_{\a_2}(x))\right|\leq \eta_{0}w_{\a_1}(x)H(f_{\a_1}(x))+w_{\a_2}(x)H(f_{\a_2}(x)),$$
 or 
$$\left|w_{\a_1}(x)|\lambda_{\a_1}(x)|^{ib}u(f_{\a_1}(x))+w_{\a_2}(x)|\lambda_{\a_2}(x)|^{ib}u(f_{\a_2}(x))\right|\leq w_{\a_1}(x)H(f_{\a_1}(x))+\eta_{0}w_{\a_2}(x)H(f_{\a_2}(x)),$$ 
depending on whether $|z_{1}(x)|/|z_{2}(x)|\leq M$ for all $x\in B\left(x',\frac{\epsilon_{2}}{|b|}\right)$ or $|z_{2}(x)|/|z_{1}(x)|\leq M$ for all $x\in B\left(x',\frac{\epsilon_{2}}{|b|}\right)$. This completes our proof. 
\end{proof}

\subsection{Verifying Lemma \ref{Lemma:Important lemma}}
\label{sec:Dolgopyat operators properties}

Let us now verify all the parts of Lemma \ref{Lemma:Important lemma}. We first note that Proposition \ref{Proposition:Arguments cancellation} gives the following immediate corollary:
\begin{cor}
Let $N\in\mathbb{N}$ be sufficiently large and $|b|>1$. If $u\in C^{1}(U)$ satisfies $|u|\leq H$ and $\|Du\|\leq A|b|H,$ then there exists $J\in \cE_{b}$ such that
$$|\cL_{ib}^{N}u|\leq \cN_{b}^{J}(H).$$
\end{cor}
Notice the corollary above gives us the first part of item 3 from Lemma \ref{Lemma:Important lemma}. It remains to prove the remaining parts.

We now prove the first part of Lemma \ref{Lemma:Important lemma}.

\begin{proof}[Proof of Lemma \ref{Lemma:Important lemma}(1)]
Let $J\in \cE_{b}$ and $H\in C_{A|b|}$. Then we have
$$\cN_{b}^{J}(H)=\sum_{\substack{\a\in \cA^{N} \\ \a \leadsto x}}w_{\a}(x)\chi(f_{\a}(x))H(f_{\a}(x)).$$ Applying the product rule for differentiation this yields
\begin{align*}
	D_{x}\cN_{b}^{J}(H)&=\sum_{\substack{\a\in \cA^{N} \\ \a \leadsto x}}D_{x}w_{\a}\chi(f_{\a}(x))H(f_{\a}(x))
	+w_{\a}(x)D_{x}\chi(f_{\a})H(f_{\a}(x))
	+w_{\a}(x)\chi(f_{\a}(x))D_{x}H(f_{\a}).
\end{align*}We will focus on each of the terms appearing in this summation individually. Applying Lemma \ref{Lemma:Bounded distortion collection} to the first term we have 
$$\Big\|\sum_{\substack{\a\in \cA^{N} \\ \a \leadsto x}}D_{x}w_{\a}\chi(f_{\a}(x))H(f_{\a}(x))\Big\|\leq  C_{2}\sum_{\substack{\a\in \cA^{N} \\ \a \leadsto x}}w_{\a}(x)\chi(f_{\a}(x))H(f_{\a}(x))\leq C_{2}\cN_{b}^{J}(H)(x).$$ Here $C_{2}$ is as in the statement of Lemma \ref{Lemma:Bounded distortion collection}. For the second term, it follows by our assumptions on $\chi$ that 
\begin{align*}
	\|\sum_{\substack{\a\in \cA^{N} \\ \a \leadsto x}}w_{\a}(x)D_{x}\chi(f_{\a})H(f_{\a}(x))\|&\leq |b|\sum_{\substack{\a\in \cA^{N} \\ \a \leadsto x}}w_{\a}(x)H(f_{\a}(x))\\
	&\leq 2|b|\cN_{b}^{J}(H)(x).
\end{align*}
In the above we used that $H=\frac{\chi H}{\chi}\leq \frac{\chi H}{\eta}\leq 2 \chi H$ (recall $\chi$ takes values in $[\eta,1]$ and $\eta\geq 1/2$). 
For the third term we have 
\begin{align*}
\Big\|\sum_{\substack{\a\in \cA^{N} \\ \a \leadsto x}}w_{\a}(x)\chi(f_{\a}(x))D_{x}H(f_{\a})\Big\|&=	\Big\|\sum_{\substack{\a\in \cA^{N} \\ \a \leadsto x}}w_{\a}(x)\chi(f_{\a}(x))D_{f_{\a}(x)}HD_{x}f_{\a}\Big\|\\
&\leq \sum_{\substack{\a\in \cA^{N} \\ \a \leadsto x}}w_{\a}(x)\chi(f_{\a}(x))A|b|H(f_{\a}(x))\gamma^{N}\\
&\leq A|b|\gamma^{N}\cN_{b}^{J}(H)(x).
\end{align*}
Combining the bounds collected above we have $$\|D \cN_{b}H\|\leq \left(C_2+2|b|+A|b|\gamma^{N}\right)\cN_{b}^{J}(H).$$ Therefore \eqref{Eq:A large 1} implies that for $N$ sufficiently large
$$\|D \cN_{b}^{J}H\|\leq A|b|\cN_{b}^{J}(H).$$ Thus the cone is preserved and our proof is complete.
\end{proof}

We now set out to prove the second part of Lemma \ref{Lemma:Important lemma}.

\begin{proof}[Proof of Lemma \ref{Lemma:Important lemma}(2)]
Applying the Cauchy Schwartz inequality and the fact $\chi\leq 1,$ for all $x\in X_{A}$ we have
\begin{align*}
|\cN_{b}^{J}H(x)|^{2}&= |\sum_{\substack{\a\in \cA^{N} \\ \a \leadsto x}}w_{\a}(x)\chi(f_{\a}(x))H(f_{\a}(x))|^{2}\\
&\leq \sum_{\substack{\a\in \cA^{N} \\ \a \leadsto x}}w_{\a}(x)\chi(f_{\a}(x))^{2}\sum_{\substack{\a\in \cA^{N} \\ \a \leadsto x}}w_{\a}(x)H(f_{\a}(x))^{2}\\
&\leq \sum_{\substack{\a\in \cA^{N} \\ \a \leadsto x}}w_{\a}(x)\chi(f_{\a}(x))\sum_{\substack{\a\in \cA^{N} \\ \a \leadsto x}}w_{\a}(x)H(f_{\a}(x))^{2}.
\end{align*}
Using \eqref{Eq:Eigenfunction 1} and the fact $\chi\leq 1$, we see that the above implies the following bound for all $x\in X_{A}$:
$$|\cN_{b}^{J}H(x)|^{2}\leq (\cL_{0}H^{2})(x).$$
If $x\in B(x_{l}',\frac{\epsilon_{2}}{3|b|})\cap X_{A}$ then $\chi(f_{\a_{l}}(x))=\eta,$ by the definition of the Dolgopyat operators, thus for these $x$ the above implies that 
\begin{align*}
|\cN_{b}^{J}H(x)|^{2}&= \left(\sum_{\substack{\a\in \cA^{N} \\ \a \leadsto x}}w_{\a}(x)-w_{\a_{l}}(x)(1-\eta)\right)\left(\sum_{\substack{\a\in \cA^{N} \\ \a \leadsto x}}w_{\a}(x)H(f_{\a}(x))^{2}\right)\\
&\leq \kappa(\cL_{0}^{N}H^{2})(x)
\end{align*}
where $\kappa\in(0,1)$ is given by $$\kappa=1-(1-\eta)\min_{\a\in \cA^{N},x\in X_{A}}\{w_{\a}(x)\}.$$ Here we have again used our assumption \eqref{Eq:Eigenfunction 1}.

Combining the above bounds, we have 
\begin{align*}
\int |\cN_{b}^{J}H|^{2}\, d\mu&\leq \kappa\int_{\cup_{l\in S(b)}B(x_{l}',\frac{\epsilon_{2}}{3|b|})}\cL_{0}^{N}H^{2}\, d\mu+\int_{(\cup_{l\in S(b)}B(x_{l}',\frac{\epsilon_{2}}{3|b|}))^{c}}\cL_{0}^{N}H^{2}\, d\mu\\
&=\int \cL_{0}^{N}H^{2}\, d\mu- (1-\kappa)\int_{\cup_{l\in S(b)}B(x_{l}',\frac{\epsilon_{2}}{3|b|})}\cL_{0}^{N}H^{2}\, d\mu\\
&=\int H^{2}\, d\mu- (1-\kappa)\int_{\cup_{l\in S(b)}B(x_{l}',\frac{\epsilon_{2}}{3|b|})}\cL_{0}^{N}H^{2}\, d\mu
\end{align*}
In the final line we used \eqref{Eq:Gibbs measure invariance}. Therefore to complete our proof it suffices to show that there exists $\delta>0$ such that 
$$\int_{\cup_{l\in S(b)}B(x_{l}',\frac{\epsilon_{2}}{3|b|})}\cL_{0}^{N}H^{2}\, d\mu\geq \delta \int H^{2}\, d\mu.$$ We start by remarking that since we choose $\{y_{l}\}_{l\in S(b)}\subset X_{A}$ to be a maximal set for which $B\left(y_{l},\frac{\epsilon_{1}D}{|b|}\right)\cap B\left(y_{l'},\frac{\epsilon_{1}D}{|b|}\right)=\emptyset$ for distinct $l,l'\in S(b),$  we must have 
\begin{equation}
    \label{Eq:Maximal set covering}
    X_{A}\subset \bigcup_{l\in S(b)}B\left(y_{l},\frac{2\epsilon_{1}D}{|b|}\right).
\end{equation}
 Moreover, using the doubling property of Gibbs measures, there must exists $C>0$ such that 
$$\mu\left(B\left(y_{l},\frac{2\epsilon_{1}D}{|b|}\right)\right)\leq C\mu\left(B\left(x_{l}',\frac{\epsilon_{2}}{3|b|}\right)\right).$$ Combining this bound with \eqref{Eq:Gibbs measure invariance}, \eqref{Eq:Property of cone}, \eqref{Eq:Maximal set covering} and the fact $\cL_{0}^{N}H^{2}\in C_{A|b|}$ for $N$ sufficiently large, we have 
\begin{align*}
	\int H^{2}d\mu= \int \cL_{0}^{N}H^{2}\, d\mu
	&\leq \sum_{l\in S(b)}\int_{B\left(y_{l},\frac{2\epsilon_{1}D}{|b|}\right)}\cL_{0}^{N}H^{2}\, d\mu\\
	&\leq e^{4\epsilon_{1}DA}\sum_{l\in S(b)}\mu\left(B\left(y_{l},\frac{2\epsilon_{1}D}{|b|}\right)\right)\min_{x\in B\left(y_{l},\frac{2\epsilon_{1}D}{|b|}\right)}\cL_{0}^{N}H^{2}(x)\\
	&\leq e^{4\epsilon_{1}DA}C\sum_{l\in S(b)}\mu\left(B\left(x_{l}',\frac{\epsilon_{2}}{3|b|}\right)\right)\min_{x\in B\left(y_{l},\frac{2\epsilon_{1}D}{|b|}\right)}\cL_{0}^{N}H^{2}(x)\\
	&\leq e^{4\epsilon_{1}DA}C\int_{\cup_{l\in S(b)}B(x_{l}',\frac{\epsilon_{2}}{3|b|})}\cL_{0}^{N}H^{2}\, d\mu
\end{align*}In the last line we used that $B\left(x_{l}',\frac{\epsilon_{2}}{3|b|}\right)\subset B\left(y_{l},\frac{2\epsilon_{1}D}{|b|}\right)$. Taking $\delta=(e^{4\epsilon_{1}DA}C)^{-1}$ this completes our proof.
\end{proof}

We will now show that the second part of the third statement in Lemma \ref{Lemma:Important lemma} is satisfied and in doing so complete our proof of this lemma.

\begin{proof}[Proof of Lemma \ref{Lemma:Important lemma}(3)]
We have 

\begin{align*}
	\|D_{x}\cL_{ib}^{N}u\|=\|\sum_{\substack{\a\in \cA^{N} \\ \a \leadsto x}}D_{x}w_{\a}|\lambda_{\a}(x)|^{ib}u(f_{\a}(x))+w_{\a}(x)D_{x}|\lambda_{\a}|^{ib}u(f_{\a}(x))+w_{\a}(x)|\lambda_{\a}(x)|^{ib}D_{x}u(f_{\a})\|.
\end{align*}	We will consider each of the three terms above individually. In what follows $C_{2}$ and $C_{3}$ are as in the statement of Lemma \ref{Lemma:Bounded distortion collection}. Applying Lemma \ref{Lemma:Bounded distortion collection} to the first term we have
\begin{align*}
	\|\sum_{\substack{\a\in \cA^{N} \\ \a \leadsto x}}D_{x}w_{\a}|\lambda_{\a}(x)|^{ib}u(f_{\a}(x))\|\leq C_{2}\sum_{\substack{\a\in \cA^{N} \\ \a \leadsto x}}w_{\a}(x)H(f_{\a}(x)).
\end{align*}
Here we have used that $|u|\leq H$. Recalling that $H\leq 2\chi H,$ the above implies 
$$\|\sum_{a\in \cA^{N}}D_{x}w_{\a}|\lambda_{\a}(x)|^{ib}u(f_{\a}(x))\|\leq 2C_{2}\cN_{b}^{J}H.$$ For the second term, using Lemma \ref{Lemma:Bounded distortion collection} and $H\leq 2\chi H$ we have
$$\|\sum_{\substack{\a\in \cA^{N} \\ \a \leadsto x}}w_{\a}(x)D_{x}|\lambda_{\a}|^{ib}u(f_{\a}(x))\|\leq C_{3}|b|\sum_{\substack{\a\in \cA^{N} \\ \a \leadsto x}}w_{\a}(x)|u(f_{\a}(x))|\leq 2C_{3}|b|\cN_{b}^{J}H.$$
For the third term we have 
\begin{align*}
	\|\sum_{\substack{\a\in \cA^{N} \\ \a \leadsto x}}w_{\a}(x)|\lambda_{\a}(x)|^{ib}D_{x}u(f_{\a})\|&=\|\sum_{\substack{\a\in \cA^{N} \\ \a \leadsto x}}w_{\a}(x)|\lambda_{\a}(x)|^{ib}D_{f_{\a}(x)}uD_{x}f_{\a}\|\\
	&\leq \sum_{\substack{\a\in \cA^{N} \\ \a \leadsto x}}w_{\a}(x)\|D_{f_{\a}(x)}u\|\gamma^{n}\\
	&\leq \sum_{\substack{\a\in \cA^{N} \\ \a \leadsto x}}w_{\a}(x)A|b|H(f_{\a}(x))\gamma^{n}\\
	&\leq 2A|b|\gamma^{n}\cN_{b}^{J}H(x).
\end{align*}In the final line we used $H\leq 2\chi H$. Summarizing the above, we have 
$$	\|D_{x}\cL_{ib}^{N}u\|\leq \left(2C_{2}+2C_{3}|b|+2A|b|\gamma^{n}\right)\cN_{b}^{J}H.$$ Which by \eqref{Eq:A large 2} implies
$$	\|D_{x}\cL_{ib}^{N}u\|\leq A|b|\cN_{b}^{J}H(x).$$ This completes our proof.
\end{proof}


\section{Equivalence of the UNI conditions}\label{appendix:Weak uni}
In this section we will prove that our UNI condition is equivalent to the UNI condition stated by Li and Pan \cite{LiPan}. To help with our exposition we recall its statement below. Our proof of this proposition is based upon an argument of Avila, Gou\"{e}zel and Yoccoz \cite{AGY}.

\begin{prop}\label{prop:weaktostrongUNI}
Let $\Phi = \{f_a : a \in \cA\}$ be a conformal IFS (i.e. $D_x f_a = \lambda_a(x) O_a(x)$, $x \in \R^d$, for some $\lambda_a(x) \in (-1,1)$ and $O_a(x) \in \mathrm{SO}(d)$) satisfying the strong separation condition and $\Sigma_{A}$ be a subshift of finite type. Assume that there exists $\epsilon_{0}>0$ such that the following holds for infinitely many $n\in \N$: There exists $x\in X_{A}$ such that for any unit vector $e\in\mathbb{R}^{d}$ there exists $\a_{1},\a_{2}\in \cW_{A}\cap \cA^{n}$ such that: 
    \begin{itemize}
    \item We have $$\left|\partial_{e}\left(\log |\lambda_{\a_1}(x)|-\log |\lambda_{\a_2}(x)|\right)\right|\geq \epsilon_{0}.$$
        \item $\a_1\leadsto x$ and $\a_2\leadsto x$.
    \end{itemize} 
Then there exists $r>0$ and $\epsilon_{0}>0$ such that for any large $n\in \N$, any $x\in X_A$ and unit vector $e\in\mathbb{R}^{d}$, there exist $\a_1, \a_2\in \cW_{A}\cap \cA^n$ such that:
    \begin{itemize}
    \item For all $y\in B(x,r)$ we have $$\left|\partial_{e}\left(\log |\lambda_{\a_1}(y)|-\log |\lambda_{\a_2}(y)|\right)\right|\geq \epsilon_{0}.$$
        \item For all $y\in B(x,r)$ we have $\a_1\leadsto y$ and $\a_2\leadsto y$.
    \end{itemize} 
\end{prop}
\begin{proof}
Recall that $U = \bigcup_{a \in \cA} U_a$ is a choice of open set satisfying $X_{\Phi}\subset U$, where the $U_a$ are disjoint open sets with $f_a(U) \subset U_a$ for all $a \in \cA$. We begin our proof by observing the following equality: 
$$\log |\lambda_{\a}(x)|=\sum_{i=1}^{n}\log |\lambda_{a_{i}}(f_{a_{i+1}\ldots a_{n}}(x))|$$ for any $x\in U$ and $\a\in \cW_{A}\cap \cA^{n}$ satisfying $\a\leadsto x.$ Consequently, for any unit vector $e\in \mathbb{R}^{d},$ if we apply the chain rule we have 
\begin{equation}
\label{Eq:directions and Jacobians}
\partial_{e} \log |\lambda_{\a}(x)|=\left(\sum_{i=1}^{n}\frac{D_{f_{a_{i+1}\ldots a_{n}}(x)}|\lambda_{a_{i}}|D_{x}f_{a_{i+1}\ldots a_{n}}}{|\lambda_{a_{i}}(f_{a_{i+1}\ldots a_{n}}(x))|}\right)\cdot e.
\end{equation}Now using the fact that our IFS is uniformly contracting and therefore $\|D_{x}f_{a_{i+1}\ldots a_{n}}\|\ll \gamma^{n-i}$ for some $\gamma\in (0,1)$, we can assert that for any $x\in U$ and $\a\in \cW_{A}\cap \cA^{n}$ satisfying $\a\leadsto x$ we have the following for any $1\leq m\leq n$:
$$\left|\partial_{e} \log |\lambda_{\a}(x)|-\left(\sum_{i=m}^{n}\frac{D_{f_{a_{i+1}\ldots a_{n}}(x)}|\lambda_{a_{i}}|D_{x}f_{a_{i+1}\ldots a_{n}}}{|\lambda_{a_{i}}(f_{a_{i+1}\ldots a_{n}}(x))|}\right)\cdot e\right|\ll \gamma^{n-m}.$$ This is equivalent to \begin{equation}
    \label{Eq:Lead and tail}
    \left|\partial_{e} \log |\lambda_{\a}(x)|-\partial_{e} \log |\lambda_{a_{m+1}\ldots a_{n}}(x)\right|\ll \gamma^{n-m}
\end{equation} for $1\leq m\leq n$. Equation \eqref{Eq:Lead and tail} has two important consequences. The first is that for $n$ sufficiently large depending on $\epsilon_{0}$, if \begin{equation}
\label{Eq:nonlinear gap} \left|\partial_{e}\left(\log |\lambda_{\a_1}(x)|-\log |\lambda_{\a_2}(x)|\right)\right|\geq \epsilon_{0}
\end{equation}for some $x\in U$ and $\a_{1},\a_{2}\in \cW_{A}\cap \cA^{n}$ satisfying $\a_{1}\leadsto x$ and $\a_{2}\leadsto x$, then 
$$\left|\partial_{e}\left(\log |\lambda_{\b_{1}\a_1}(x)|-\log |\lambda_{\b_{2}\a_2}(x)|\right)\right|\geq \frac{\epsilon_{0}}{2}$$ for any $\b_{1},\b_{2}\in \cW_{A}$ satisfying $\b_{1}\leadsto \a_{1}$ and $\b_{2}\leadsto \a_{2}.$ Consequently, at the cost of swapping $\epsilon_{0}$ with a potentially smaller constant, we can assume that the hypothesis of our lemma holds for all $n$ sufficiently large. The second consequence of \eqref{Eq:Lead and tail} is that there exists $r_{1}>0$ depending only on our IFS and $\epsilon_{0},$ such that if \eqref{Eq:nonlinear gap} holds for some $x\in X_{A},$ $\a_{1},\a_{2}\in \cW_{A}\cap \cA^{n}$ satisfying $\a_{1}\leadsto x$ and $\a_{2}\leadsto x$ and unit vector $e\in \R^{d},$ then for all $y\in B(x,r_{1})$ we have 
\begin{equation}
    \label{Eq:eps/2 bound}
    \left|\partial_{e}\left(\log |\lambda_{\a_1}(y)|-\log |\lambda_{\a_2}(y)|\right)\right|\geq \frac{\epsilon_{0}}{2},
\end{equation}
  $\a_{1}\leadsto y$ and $\a_{2}\leadsto y$. 

Now let $n\in \N$ be a large number and $x\in X_{A}$ be such that for all unit vectors $e\in \R^{d}$ there exists $\a_{1},\a_{2}\in \cW_{A}\cap \cA^{n}$ such that \eqref{Eq:nonlinear gap} is satisfied, $\a_{1}\leadsto x$, and $\a_{2}\leadsto x.$ Using that $\Sigma_{A}$ is topologically mixing and $r_{1}$ depends only on our IFS and $\epsilon_{0}$, we can assert that there exists $m\in \N$ depending only on our IFS, $\Sigma_{A}$ and $\epsilon_{0}$ such that for any $x'\in X_{A}$ there exists $\c\in \cW_{A}\cap \cA^{m}$ so that $\c\leadsto x'$ and \begin{equation}
\label{Eq:x' into x ball}
    f_{\c}(x')\in B(x,r_{1}). 
\end{equation} For such an $x'$ and $\c$ we have the following for any unit vector $e\in \R^{d}$ and $\a_{1},\a_{2}\in \cW_{A}\cap \cA^{n}$ satisfying $\a_{1}\leadsto \c$ and $\a_{2}\leadsto \c$:
\begin{align*}
&\left|\partial_{e}\left(\log |\lambda_{\a_1\c}(x')|-\log |\lambda_{\a_2\c}(x')|\right)\right|\\
=&\left|\partial_{e}\left(\log |\lambda_{\a_1}(f_{\c}(x'))|-\log |\lambda_{\a_2}(f_{\c}(x'))|\right)\right|\\
=&\left(\sum_{i=1}^{n}\frac{D_{f_{a_{i+1,1}\ldots a_{n,1}}(x)}|\lambda_{a_{i,1}}|D_{x}f_{a_{i+1,1}\ldots a_{n,1}}}{|\lambda_{a_{i,1}}(f_{a_{i+1,1}\ldots a_{n,1}}(x))|}-\sum_{i=1}^{n}\frac{D_{f_{a_{i+1,2}\ldots a_{n,2}}(x)}|\lambda_{a_{i,2}}|D_{x}f_{a_{i+1,2}\ldots a_{n,2}}}{|\lambda_{a_{i,2}}(f_{a_{i+1,2}\ldots a_{n,2}}(x))|}\right)D_{x}f_{\c}\cdot e.
\end{align*}In the final line we used the chain rule and \eqref{Eq:directions and Jacobians}. If we let $v_{x',e}=D_{x'}f_{\c}\cdot e$ then it follows from the above and \eqref{Eq:directions and Jacobians} that 
\begin{equation}
\label{Eq:v direction} \left|\partial_{e}\left(\log |\lambda_{\a_1\c}(x')|-\log |\lambda_{\a_2\c'}(x')|\right)\right|=\left|\partial_{v_{x',e}}\left(\log |\lambda_{\a_{1}}(f_{\c}(x'))|-\log |\lambda_{\a_2}(f_{\c}(x'))|\right)\right|.
\end{equation}The vector $v_{x',e}$ is not necessarily of unit length and instead satisfies the lower bound $\|v_{x',e}\|\geq \kappa^{m}$ for some $\kappa\in (0,1)$ depending on the IFS.

Using our assumptions on $x,$ we can assert that there exists $\a_{1},\a_{2}\in \cW_{A}$ such that $\a_{1}\leadsto x,$ $\a_{2}\leadsto x$ and
$$\left|\partial_{\frac{v_{x',e}}{\|v_{x',e}\|}}\left(\log |\lambda_{\a_{1}}(x)|-\log |\lambda_{\a_2}(x)|\right)\right|\geq \epsilon_{0}.$$ Now using the above together with our lower bound for the norm of $v_{x',e},$ \eqref{Eq:eps/2 bound}, \eqref{Eq:x' into x ball}, and \eqref{Eq:v direction}, yields that for this choice of $\a_{1},\a_{2}$ we have $\a_1\c\leadsto x',$ $\a_{2}\c\leadsto x'$ and 
$$\left|\partial_{e}\left(\log |\lambda_{\a_1\c}(x')|-\log |\lambda_{\a_2\c'}(x')|\right)\right|\geq \frac{\epsilon_{0}\kappa^{m}}{2}.$$ Summarizing the above, we have shown that for any $n\in \N$ sufficiently large, for any $x\in X_{A}$ and unit vector $e\in \R^{d}$ there exists $\a_{1},\a_{2}\in \cW_{A}\cap \cA^{n}$ with $\a_{1}\leadsto x,$ and $\a_{2}\leadsto x$ such that $$\left|\partial_{e}\left(\log |\lambda_{\a_1}(x)|-\log |\lambda_{\a_2}(x)|\right)\right|\geq \frac{\epsilon_{0}\kappa^{m}}{2}.$$ 
Appealing again to \eqref{Eq:Lead and tail} as in the derivation of \eqref{Eq:eps/2 bound}, it can be shown that there exists $r>0$ depending only on our IFS and $\epsilon_{0}$, such that for any $n\in \mathbb{N}$ sufficiently large, for any $x\in X_{A}$ and unit vector $e\in \mathbb{R}^{d},$ there exits $\a_1,\a_{2}\in \cW_{A}\cap \cA^{n}$ such that for all $y\in B(x,r)$ and $\a_{1}\leadsto y,$ and $\a_{2}\leadsto y$, we have 
$$\left|\partial_{e}\left(\log |\lambda_{\a_1}(y)|-\log |\lambda_{\a_2}(y)|\right)\right|\geq \frac{\epsilon_{0}\kappa^{m}}{4}.$$
Recalling that $m$ and $\kappa$ only depend on our IFS, $\Sigma_{A}$ and $\epsilon_{0}$, we see that we can take $\frac{\epsilon_{0}\kappa^{m}}{4}$ as our choice of $\epsilon_{0}$. Taking $r$ as chosen above completes our proof. 
\end{proof}

\bibliography{bibliography}
\bibliographystyle{plain}

\end{document}

%% file: macros.tex
\newtheorem{thm}{Theorem}[section]
\newtheorem{lem}[thm]{Lemma}

\usepackage{xcolor}

\newtheorem{prop}[thm]{Proposition}

\newtheorem{cor}[thm]{Corollary}

\theoremstyle{definition}
\newtheorem{definition}[thm]{Definition}
\newtheorem*{definition-nono}{Definition}

\newtheorem{example}[thm]{Example}

\newtheorem{remark}[thm]{Remark}

\newtheorem*{acknowledgement}{Acknowledgements}

\newtheoremstyle{case}{}{}{}{}{}{:}{ }{}
\theoremstyle{case}

\newcommand{\N}{\mathbb{N}}
\newcommand{\Z}{\mathbb{Z}}

\newcommand{\R}{\mathbb{R}}
\newcommand{\C}{\mathbb{C}}
\renewcommand{\H}{\mathbb{H}}

\newcommand{\mc}{\mathcal}

\newcommand{\mf}{\mathfrak}
\newcommand{\mbf}{\mathbf}
\newcommand{\mrm}{\mathrm}

\renewcommand{\a}{\alpha}
\renewcommand{\b}{\beta}
\newcommand{\g}{\gamma}
\newcommand{\G}{\Gamma}
\renewcommand{\d}{\delta}
\newcommand{\e}{\varepsilon}
\renewcommand{\l}{\lambda}
\renewcommand{\L}{\Lambda}

\newcommand{\vp}{\varphi}
\renewcommand{\t}{\tau}

\renewcommand{\k}{\kappa}
\newcommand{\oa}{\overline{a}}
\newcommand{\ob}{\overline{b}}

\newcommand{\set}[1]{\left\{#1\right\}}

\def\multiset#1#2{\ensuremath{\left(\kern-.3em\left(\genfrac{}{}{0pt}{}{#1}{#2}\right)\kern-.3em\right)}}
\newcommand{\norm}[1]{\left\lVert#1\right\rVert}

\newcommand{\Acal}{\mc{A}}
\newcommand{\Bcal}{\mc{B}}

\newcommand{\Jcal}{\mc{J}}

\newcommand{\Pcal}{\mc{P}}
\newcommand{\Qcal}{\mc{Q}}

\newcommand{\Wcal}{\mc{W}}

\newcommand{\cB}{\mathcal{B}}
\newcommand{\cN}{\mathcal{N}}

\newcommand{\cL}{\mathcal{L}}

\newcommand{\cE}{\mathcal{E}}

\newcommand{\cO}{\mathcal{O}}
\newcommand{\cW}{\mathcal{W}}

\newcommand{\cA}{\mathcal{A}}

\newcommand{\diam}[1]{\mrm{diam}\left(#1\right)}

\numberwithin{equation}{section}

\renewcommand{\a}{\mathbf{a}}
\renewcommand{\b}{\mathbf{b}}
\renewcommand{\c}{\mathbf{c}}

\newcommand{\bd}{\mbf{d}}
\newcommand{\be}{\mbf{e}}

\newcommand{\ps}{\mu^{\mrm{PS}}}

\newcommand{\Ad}{\mrm{Ad}}

\newcommand{\dist}{\mrm{dist}}

\newcommand{\id}{\mrm{Id}}

\newcommand{\murhoj}{\mu^u_{y^j_\rho}}

\newcommand{\1}{\mathrm{\mathbf{1}}}
\newcommand{\eps}{\varepsilon}

\newcommand{\p}{\mathbf{p}}
\newcommand{\q}{\mathbf{q}}

\newcommand{\T}{\mathbb{T}}
\renewcommand{\k}{\mathrm{\mathbf{k}}}

\newcommand{\yrho}{y_{\rho}}

\newcommand{\supp}{\mrm{supp}}

\usepackage{amsfonts}